\documentclass{amsart}

\usepackage[mathscr]{eucal}
\usepackage{amssymb}
\usepackage[usenames,dvipsnames]{xcolor} 
\usepackage[normalem]{ulem}
\usepackage{amsthm}
\usepackage{mathtools}
\usepackage{bbm}
\DeclareSymbolFont{bbold}{U}{bbold}{m}{n} 
\DeclareSymbolFontAlphabet{\mathbbold}{bbold}
\usepackage{mathdots}
\usepackage{float}
\usepackage[stable]{footmisc} 
\usepackage{enumitem}
\setlist[itemize]{align=parleft,left=1em..2em}
\usepackage{amsmath}
\usepackage{wasysym}
\usepackage{tikz}
\usetikzlibrary{calc}
\usepackage{tikz-cd}
\usetikzlibrary{arrows}
\usepackage[all]{xy}
\SelectTips{cm}{} 

\usepackage[unicode]{hyperref} 

\hypersetup{colorlinks=true,citecolor=Brown,urlcolor=Brown,linkcolor=Brown}

\usepackage{microtype}
\usepackage[nameinlink,capitalise,noabbrev]{cleveref}

\numberwithin{equation}{section}
\swapnumbers 

\newtheorem{Thm}[equation]{Theorem}
\newtheorem*{Thm*}{Theorem}
\newtheorem{Prop}[equation]{Proposition}
\newtheorem{Lem}[equation]{Lemma}
\newtheorem{Cor}[equation]{Corollary}
\theoremstyle{remark}
\newtheorem{Def}[equation]{Definition}
\newtheorem{Ter}[equation]{Terminology}
\newtheorem{Not}[equation]{Notation}
\newtheorem{Exa}[equation]{Example}

\crefname{Exa}{Example}{Examples}

\newtheorem{Hyp}[equation]{Hypothesis}

\newtheorem{Rem}[equation]{Remark}

\newtheorem{Que}[equation]{Question}

\newcommand{\nc}{\newcommand}
\nc{\dmo}{\DeclareMathOperator}

\nc{\Beren}[1]{{\color{MidnightBlue}#1}}
\nc{\Bout}[1]{\Beren{\sout{#1}}}

\usepackage{todonotes}

\newcommand*{\tikztwocircle}[2]{%
  \setbox0=\hbox{\strut}%
  \begin{tikzpicture}
    \useasboundingbox (0,0.0em) rectangle (0,0em);
\draw (0,-0.03) node {$\scriptscriptstyle{\color{#1}\LEFTCIRCLE}\kern-.80em{\color{#2}\RIGHTCIRCLE}$};
\draw (0,0.05) circle[radius=.32em];
  \end{tikzpicture}%
}
\newcommand*{\tikzbulletaux}[2]{%
  \setbox0=\hbox{\strut}%
  \begin{tikzpicture}
    \useasboundingbox (-.10em,0) rectangle (.10em,0);
    \filldraw[draw=#1,fill=#2] (0,0) circle[radius=.20em];
  \end{tikzpicture}%
}

\newcommand*{\tikzbullet}[1]{%
	\tikzbulletaux{black}{#1}
}

\renewcommand{\emptyset}{\varnothing}
\nc{\specializesto}{\rightsquigarrow}
\nc{\AK}{A_{\cat K}}
\nc{\RK}{R_{\cat K}}
\nc{\RT}{R_{\cat T}}
\nc{\Kp}{\cat K_{\mathfrak p}}
\nc{\KP}{\cat K_{\cat P}}
\nc{\unitation}[1]{#1_{\hspace{-0.15ex}\langle\unit\rangle}}
\nc{\concentration}[2]{#1_{\hspace{-0.15ex}\langle #2 \rangle}}
\nc{\Tconc}[1]{\concentration{\cat T}{#1}}
\nc{\SpcTconc}[1]{\Spc\hspace{-0.45ex}\big(\concentration{\cat T}{#1}^c\big)}
\nc{\Spcconc}[2]{\Spc\hspace{-0.45ex}\big(\concentration{#1}{#2}^c\big)}
\nc{\Spcunit}[1]{\Spc\hspace{-0.45ex}\big(\unitation{#1}^c\big)}
\nc{\SpcT}{\Spc(\cat T^c)}
\nc{\SpcS}{\Spc(\cat S^c)}
\nc{\cTc}{\cat T^c}
\nc{\cSc}{\cat S^c}
\nc{\cT}{\cat T}
\nc{\cS}{\cat S}
\nc{\altmathbb}[1]{\mathbbold{#1}}
\nc{\eh}{\altmathbb{e}_h}
\nc{\fh}{\altmathbb{f}_h}
\nc{\en}{\altmathbb{e}_n}
\nc{\fn}{\altmathbb{f}_n}
\nc{\eY}{\altmathbb{e}_Y}
\nc{\fY}{\altmathbb{f}_Y}
\nc{\tY}{\altmathbb{t}_Y}
\nc{\eZ}{\altmathbb{e}_Z}
\nc{\fZ}{\altmathbb{f}_Z}
\nc{\tZ}{\altmathbb{t}_Z}
\nc{\gP}{\altmathbb{g}_{\cat P}}
\nc{\gm}{\altmathbb{g}_{\mathfrak m}}
\nc{\SpecRK}{\Spec(R_{\cat K})}
\nc{\Pone}{\mathbb{P}^1_k}
\nc{\PS}{\mathbb{P}^1_{\hspace{-0.2ex}S}}

\nc{\frakm}{\mathfrak m}
\nc{\overbar}[1]{\mkern 1.5mu\overline{\mkern-1.5mu#1\mkern-1.5mu}\mkern 1.5mu}
\nc{\bbullet}{{\scriptscriptstyle\hspace{-1pt}\bullet}}
\nc{\bullett}{{\scriptscriptstyle\bullet}\hspace{-1pt}}
\nc{\LF}{L\hspace{-0.2ex}F}
\nc{\SpG}{\Sp^G}
\nc{\Prst}{{\cat P}\mathrm{r^{st}}}
\nc{\Mack}{\mathcal{M}ack}
\nc{\SC}{S\cat C}
\nc{\OX}{\cat O_{\hspace{-0.2ex}X}}
\nc{\OXx}{\cat O_{\hspace{-0.2ex}X,x}}
\nc{\OXX}{\OX\hspace{-0.1ex}(X)}
\nc{\OK}{\cat O_{\hspace{-0.075ex}\cat K}}
\nc{\OKP}{\cat O_{\hspace{-0.075ex}\cat K,\cat P}}
\nc{\OT}{\cat O_{\hspace{-0.075ex}\cat T}}
\nc{\OS}{\cat O_{\hspace{-0.075ex}\cat S}}
\renewcommand{\complement}{{\smash{\mathsf{c}}}}
\newcommand{\cc}{\complement}
\nc{\frakp}{{\mathfrak p}}
\nc{\frakq}{{\mathfrak q}}
\dmo{\Mot}{Mot}
\dmo{\GW}{GW}
\dmo{\StMod}{StMod}
\dmo{\stmod}{stmod}
\dmo{\Stab}{Stab}
\dmo{\Aff}{Aff}
\dmo{\Ext}{Ext}
\dmo{\gen}{gen}
\dmo{\Pic}{Pic}
\dmo{\DM}{DM}
\dmo{\DTM}{DTM}
\dmo{\DATM}{DATM}
\dmo{\DAM}{DAM}
\dmo{\DRep}{DRep}
\dmo{\DPerm}{DPerm}
\dmo{\Perm}{Perm}
\dmo{\Gal}{Gal}
\nc{\DMeff}{{\DM^{{eff}}}}
\nc{\DMeffkR}{\DMeff\hspace{-0.3ex}(k;R)}
\nc{\DMQ}{\DM_Q}
\nc{\Dbcoh}[1]{\Der^b\hspace{-0.2ex}(\mathrm{coh}(#1))}
\nc{\Dbmod}[1]{\Der^b\hspace{-0.2ex}(\mathrm{mod}(#1))}
\nc{\DbSmod}[1]{\Der_S^b\hspace{-0.2ex}(\mathrm{mod}(#1))}
\nc{\Dbkmod}[1]{\Der_{\{k\}}^b\hspace{-0.2ex}(\mathrm{mod}(#1))}
\nc{\Derdg}{\Der_{\mathrm{dg}}}
\dmo{\DerKal}{DMack}
\nc{\DMack}{\DerKal}
\dmo{\rep}{rep}
\dmo{\Rep}{Rep}
\dmo{\Inj}{Inj}
\dmo{\Der}{D}
\nc{\Derqc}{\Der_{\mathrm{qc}}}
\dmo{\DMot}{DMot}
\dmo{\rmH}{H}
\dmo{\piu}{\underline{\pi}}
\dmo{\Sphere}{\mathbb{S}}
\nc{\HA}{{\rmH \hspace{-0.2em}\bbA}}
\nc{\Hk}{{\rmH \hspace{-0.15em}k}}
\nc{\Hkbar}{{\rmH \hspace{-0.15em}\underline{k}}}
\nc{\fieldk}{\mathbb{k}}
\nc{\Nerve}{{\rm N}}
\nc{\dgNerve}{{\rm N}_{\rm dg}}
\nc{\NAb}{{\rm N} \hspace{-0.15em}\Ab}
\nc{\HR}{{\rmH \hspace{-0.15em}R}}
\nc{\HQ}{{\rmH \hspace{-0.15em}\bbQ}}
\nc{\HZ}{{\rmH \hspace{-0.15em}\bbZ}}
\nc{\HZp}{{\rmH \hspace{-0.15em}\bbZ_{(p)}}}
\nc{\HZbar}{{\rmH \hspace{-0.15em}\underline{\bbZ}}}
\nc{\Fp}{{\bbF_{\hspace{-0.1em}p}}}
\nc{\Ftwo}{{\bbF_{\hspace{-0.1em}2}}}
\nc{\HFp}{{\rmH \hspace{-0.15em}\bbF_{\hspace{-0.1em}p}}}
\nc{\HZG}{\rmH_{G,\mathbb{Z}}}
\nc{\HRG}{\rmH_{G,R}}
\nc{\DHZpG}{\Der(\rmH_{G,\mathbb{Z}_{(p)}})}
\nc{\DHZG}{\Der(\rmH_{G,\mathbb{Z}})}
\nc{\DHQG}{\Der(\rmH_{G,\mathbb{Q}})}
\nc{\DHZH}{\Der(\rmH_{H,\mathbb{Z}})}
\nc{\DHZGN}{\Der(\rmH_{G/N,\mathbb{Z}})}
\nc{\DHZGOp}{\Der(\rmH_{G/{O^p(G)},\mathbb{Z}})}
\nc{\DHZpGOp}{\Der(\rmH_{G/{O^p(G)},\mathbb{Z}_{(p)}})}
\nc{\DHZpGK}{\Der(\rmH_{G/K,\mathbb{Z}_{(p)}})}
\nc{\DHZGK}{\Der(\rmH_{G/K,\mathbb{Z}})}
\nc{\DHZpCp}{\Der(\rmH_{C_p,\mathbb{Z}_{(p)}})}
\nc{\DHZqCp}{\Der(\rmH_{C_p,\mathbb{Z}_{(q)}})}
\nc{\DHZCp}{\Der(\rmH_{C_p,\mathbb{Z}})}
\nc{\DHZ}{\Der(\HZ)}
\nc{\DHZp}{\Der(\HZ_{(p)})}
\nc{\Z}{\mathbb{Z}}
\nc{\SSG}{\text{sSet}_*^G}
\nc{\sSet}{\text{sSet}}
\dmo{\Loc}{Loc}
\dmo{\Coloc}{Coloc}
\dmo{\Locideal}{Locid}
\dmo{\Colocideal}{Colocid}
\nc{\LOCO}{\Locideal}
\nc{\COLOCO}{\Colocideal}
\nc{\Loco}[1]{\LOCO\langle #1 \rangle}
\nc{\Coloco}[1]{\COLOCO\langle #1 \rangle}
\dmo{\Con}{Conj}
\dmo{\Sub}{Sub}
\dmo{\Id}{Id}
\dmo{\rmK}{\textrm{\rm K}}
\dmo{\Spc}{Spc}
\dmo{\thick}{thick}
\dmo{\thickid}{thickid}
\nc{\thickt}[1]{\thickid\langle #1 \rangle}
\dmo{\cone}{cone}
\dmo{\End}{End}
\dmo{\eend}{\mathsf{end}}
\dmo{\Mor}{Mor}
\dmo{\Hom}{Hom}
\dmo{\id}{id}
\dmo{\incl}{incl}
\dmo{\Img}{Im}
\dmo{\im}{im}
\dmo{\Ker}{Ker}
\dmo{\ind}{ind}
\dmo{\CoInd}{coind}
\dmo{\res}{res}
\dmo{\infl}{infl}
\dmo{\triv}{triv}
\dmo{\Tel}{Tel}
\dmo{\Mod}{Mod}
\dmo{\opname}{op}
\dmo{\SH}{SH}
\nc{\SHcell}{\SH_{\mathrm{cell}}}
\dmo{\smallb}{b}
\dmo{\Spec}{Spec}
\dmo{\supp}{supp}
\dmo{\Supp}{Supp}
\dmo{\Cosupp}{Cosupp}
\nc{\SHc}{{\SH^c}}
\nc{\SHp}{{\SH_{(p)}}}
\nc{\SHcp}{{\SH^c_{(p)}}}
\nc{\SHGN}{\SH_{G/N}}
\nc{\SHG}{\SH_G}
\nc{\SHGp}{(\SH_G)_{(p)}}
\nc{\SHGc}{\SHG^c}
\nc{\SHGcp}{\SHG^c_{(p)}}
\nc{\quadtext}[1]{\quad\textrm{#1}\quad}
\nc{\qquadtext}[1]{\qquad\textrm{#1}\qquad}
\nc{\adj}{\dashv}
\nc{\adjto}{\rightleftarrows}
\nc{\bbL}{\mathbb{L}}
\nc{\bbA}{\mathbb{A}}
\nc{\bbN}{\mathbb{N}}
\nc{\bbQ}{\mathbb{Q}}
\nc{\bbZ}{\mathbb{Z}}
\nc{\bbF}{\mathbb{F}}
\nc{\bbR}{\mathbb{R}}
\nc{\bbC}{\mathbb{C}}
\nc{\bbE}{\mathbb{E}}
\nc{\cat}[1]{\mathscr{#1}}
\nc{\ie}{{\sl i.e.}, }
\nc{\into}{\mathop{\rightarrowtail}}
\nc{\inv}{^{-1}}
\nc{\isoto}{\mathop{\overset{\sim}\to}}
\nc{\isotoo}{\mathop{\overset{\sim}\too}}
\nc{\onto}{\mathop{\twoheadrightarrow}}
\nc{\too}{\mathop{\longrightarrow}\limits}
\nc{\mapstoo}{\longmapsto}
\nc{\adh}[1]{\overline{#1}}
\nc{\adhpt}[1]{\adh{\{#1\}}}
\nc{\aka}{{a.\,k.\,a.}\ }
\nc{\calF}{\mathcal{F}}
\nc{\eg}{{\sl e.\,g.}}
\nc{\Homcat}[1]{\Hom_{\cat #1}}
\nc{\hook}{\hookrightarrow}
\nc{\ideal}[1]{\langle #1\rangle}
\nc{\ihomname}{\mathsf{hom}}
\nc{\ihom}[1]{\mathsf{hom}(#1)}
\nc{\Mid}{\,\big|\,}
\nc{\MMod}{\,\text{-}\Mod}%
\nc{\op}{^{\opname}}
\nc{\oto}[1]{\overset{#1}\to}
\nc{\otoo}[1]{\overset{#1}{\,\too\,}}
\nc{\sminus}{\!\smallsetminus\!}
\nc{\poplus}[1]{^{\oplus #1}}%
\nc{\potimes}[1]{^{\otimes #1}}
\nc{\sbull}{{\scriptscriptstyle\bullet}}
\nc{\SET}[2]{\big\{\,#1\Mid#2\,\big\}}
\nc{\SETT}[1]{\big\{\,#1\,\big\}}
\nc{\SpcK}{\Spc(\cat K)}
\nc{\then}{\Rightarrow}
\nc{\unit}{\mathbbmtt{1}}
\nc{\unitT}{\unit_{\cat T}}
\nc{\unitS}{\unit_{\cat S}}
\nc{\xra}{\xrightarrow}
\nc{\phigeom}[1]{\widetilde{\Phi}^{#1}}
\nc{\phigeomb}[1]{\Phi^{#1}}
\dmo{\Oname}{O}
\dmo{\proper}{proper}
\dmo{\lenormal}{\unlhd}
\dmo{\lnormal}{\lhd}
\nc{\normal}{\trianglelefteq}
\nc{\Op}{\Oname^p}
\nc{\Oq}{\Oname^q}
\dmo{\Sp}{Sp}
\dmo{\Ho}{Ho}
\dmo{\Fin}{Fin}
\dmo{\add}{add}
\dmo{\smd}{smd}
\dmo{\Fun}{Fun}
\dmo{\Ind}{Ind}
\dmo{\CAlg}{CAlg}
\dmo{\CMon}{CMon}
\dmo{\Map}{Map}
\dmo{\Span}{Span}
\dmo{\N}{N}
\dmo{\Cat}{Cat}
\dmo{\colim}{colim}
\dmo{\Ch}{Ch}
\dmo{\A}{\mathbb{A}^{eff}}
\nc{\AGeff}{\mathbb{A}_G^{\mathrm{eff}}}
\nc{\BGeff}{\mathcal{B}_G^{\mathrm{eff}}}
\nc{\BG}{{\mathcal{B}_G}}
\nc{\NBGeff}{{\N}{\BGeff}}
\dmo{\Ab}{Ab}
\dmo{\Set}{Set}
\dmo{\ev}{ev}
\dmo{\Spcl}{Spcl}
\nc{\Funadd}{\Fun_{\add}}
\dmo{\proj}{proj}
\dmo{\cof}{cof}

\newcounter{enum-resume-hack}
\Crefname{Thm}{Theorem}{Theorems}
\Crefname{Prop}{Proposition}{Propositions}
\Crefname{thmx}{Theorem}{Theorems}

\begin{document}


\title{The tensor triangular geometry of fully~faithful functors}
\author{Beren Sanders}
\date{\today}

\address{Beren Sanders, Mathematics Department, UC Santa Cruz, 95064 CA, USA}
\email{beren@ucsc.edu}
\urladdr{http://people.ucsc.edu/$\sim$beren/}


\begin{abstract}
	We prove that the map on Balmer spectra induced by a fully \mbox{faithful} geometric functor is a quotient map whose fibers are connected. This is an analogue of the Zariski Connectedness Theorem in algebraic geometry and it can be applied to a plethora of examples in equivariant and motivic \mbox{mathematics.} We isolate a significant source of examples by introducing the ``concentration'' of a tt-category at a well-behaved chosen set of compact generators. Various categories of cellular objects arise in this way. In particular, the ``unitation'' of a tt-category is the concentration at the unit object. We compute the Balmer spectrum of the unitation of the equivariant stable homotopy category as well as related categories arising in equivariant higher algebra. We also apply our results to the study of the comparison map of a tt-category. Among other results, we prove that the comparison map of a connective category is a quotient map with connected fibers. This involves studying the tt-geometry of weight complex functors, which may be of independent interest. We also study the relationship between the comparison map and the affinization of the Balmer spectrum viewed as a locally ringed space. These results provide a layered approach to understanding the spectrum of a given tt-category, by starting with the Zariski spectrum of the endomorphism ring of the unit, and then passing backwards to larger and larger concentrations (through quotient maps with connected fibers). Significant stages along the way include the passage to the unitation and the passage to the concentration at the Picard group.
	\vspace*{-2.3em}
\end{abstract}

\vspace*{-1.0cm}
\maketitle

{
\hypersetup{linkcolor=black}
\tableofcontents
}

\newpage
\section{Introduction}\label{sec:intro}

This paper is a study of fully faithful functors in tensor triangular geometry. Throughout $\cS$ and $\cT$ will denote rigidly-compactly generated tt-categories. We prove the following: 

\begin{Thm}\label{thm:intro-fully-faithful}
	Let $f^*:\cT\to \cS$ be a fully faithful geometric functor. Then the induced map $\varphi:\SpcS\to\SpcT$ is a spectral quotient map whose fibers are connected.
\end{Thm}

This is an analogue in tensor triangular geometry of the Zariski Connectedness Theorem in algebraic geometry. It is established as \cref{thm:main-thmb}. The notion of a spectral quotient map (\cref{ter:spectral-quotient}) is a slight generalization of the usual notion of a topological quotient map in point-set topology. It coincides with the usual notion if the domain is noetherian. In fact, the spectral quotient map of \cref{thm:intro-fully-faithful} is a topological quotient map whenever either of the spaces is noetherian; see \cref{rem:spectral-is-strong-ff}. The first part of the theorem holds more generally:

\begin{Thm}\label{thm:intro-faithful}
	Let $f^*:\cT \to \cS$ be a faithful geometric functor. Then the induced map $\varphi:\SpcS \to \SpcT$ is a spectral quotient map.
\end{Thm}

This is established as \cref{thm:faithful-are-quotient}. Examples of faithful and fully faithful geometric functors are plentiful in the algebraic, topological and motivic domains where tt-geometry has taken root and a large portion of the paper is devoted to explaining the consequences of the above two theorems in these various settings. We formalize a significant source of examples as follows: For any set $\cat G \subseteq \cTc$ of compact objects such that $\thick\langle\cat G\rangle \subseteq \cTc$ is a rigid tensor-subcategory, the localizing subcategory
	\[
		\concentration{\cT}{\cat G} \coloneqq \Loc\langle \cat G \rangle \subseteq \cT
	\]
is itself a rigidly-compactly generated tt-category and the inclusion $\concentration{\cT}{\cat G} \hookrightarrow \cT$ is a fully faithful geometric functor. We call $\concentration{\cT}{\cat G}$ the \emph{concentration} of $\cT$ at $\cat G$. For example, the category of cellular motivic spectra $\SHcell(\bbC) \subseteq \SH(\bbC)$ is the concentration of $\SH(\bbC)$ at the motivic spheres and the derived category of Tate motives $\DTM(\bbC)\subseteq \DM(\bbC)$ is the concentration of $\DM(\bbC)$ at the Tate twists.

Although the concentration at $\cat G$ is a tt-subcategory $\concentration{\cT}{\cat G} \subseteq \cat T$, we regard concentration as a process $\cT \rightsquigarrow \concentration{\cT}{\cat G}$ which geometrically effects a quotient with connected fibers: $\Spc(\cTc)\to\Spc(\concentration{\cT}{\cat G}^c)$. We call the extreme case $\cat G=\{ \unit \}$ the \emph{unitation} of $\cT$:
	\[
		\unitation{\cT} \coloneqq \Loc\langle\unit\rangle \subseteq \cT
	\]
Note that $\unitation{\cT}=\cT$ says that $\cT$ is \emph{unigenic} meaning that it is generated by the tensor unit. We thus also call $\unitation{\cT}$ the \emph{unigenic core} of $\cT$.

\subsection*{Examples in equivariant higher algebra}
We study the process of concentration for various examples arising in equivariant higher algebra with a particular emphasis on unitation. For example, in \cref{thm:SHG} we completely determine the map
	\[
		\Spc(\SHG^c) \to \Spc(\unitation{(\SHG)}^c)
	\]
induced by the unitation of the equivariant stable homotopy category $\SHG$ for any finite group $G$. The $G=C_p$ case is depicted in \cref{fig:SHCp}.
\begin{figure}[h]
\resizebox{\textwidth}{!}{
\input{figure-SH-Cp}
}
\caption{The Balmer spectrum of the unitation of $\SH_{C_p}$.}
\label{fig:SHCp}
\end{figure}
In general, it glues together those points which become glued in the spectrum of the Burnside ring, \emph{but only at height infinity}.

This result for $\SHG$ is obtained by first understanding the unitation of its linearization --- the category of derived Mackey functors $\DHZG$ --- and its chromatic truncations $(\SHG)_{\le n}$. This is accomplished in \cref{thm:DHZG-comp} and \cref{thm:truncation}, respectively. We discover that unitation behaves very differently in these two cases, which reflects the curious behaviour of the unitation of~$\SHG$.

We also describe the unitations of the derived category of permutation modules~$\DPerm(G;k)$, the category of derived representations~$\DRep(G;k)$, and the derived category of the abelian category of Mackey functors $\Der(\Mack(G))$. There is a chain of tt-functors
	\begin{equation}\label{eq:chain}
		\SH_G \to \DHZG \to \Der(\Mack(G)) \to \DPerm_G \to \DRep_G
	\end{equation}
and we will see that each category behaves in a qualitatively different way with respect to unitation; see \cref{sec:equivariant}.

\subsection*{The comparison map}
These constructions are closely related to the (graded and ungraded) comparison map
	\[
		\rho:\SpcT \to \Spec(\RT)
	\]
whose target is the Zariski spectrum of the (graded or ungraded) endomorphism ring of the unit. Since $\unitation{\cT}\hookrightarrow \cT$ is fully faithful, this ring is the same for both categories, and we have a factorization
	\begin{equation}\label{eq:intro-factorization}
	\begin{tikzcd}
		\SpcT \ar[d,"\varphi"'] \ar[rd,bend left,"\rho"] & \\
		\Spcunit{\cT} \ar[r,"\rho"'] & \Spec(\RT).
	\end{tikzcd}
	\end{equation}
If the top $\rho$ is a homeomorphism then the spectral quotient $\varphi$ is injective and hence a homeomorphism. On the other hand, if the bottom $\rho$ is a homeomorphism then unitation coincides with the comparison map. Both of these extreme possibilities can occur, although the example of $\SHG$ shows that in general each map can be nontrivial.

The first step in a general strategy for understanding the Balmer spectrum of a category $\cT$ is to understand $\Spec(\RT)$. The task of computing $\SpcT$ then amounts to determining the fibers of the associated comparison map. This approach has been highly effective; see \cite{BalmerSanders17} and \cite{AroneBarthelHeardSanders24pp}, for example. From our current point of view, this method factors into two steps:
	\begin{enumerate}
		\item Understand the spectrum of the unitation $\Spc(\unitation{\cT}^c)$ by understanding the fibers of its comparison map $\rho:\Spc(\unitation{\cT}^c) \to \Spec(\RT)$; then
		\item Understand the fibers of $\varphi : \SpcT \to \Spcunit{\cT}$ where we have \cref{thm:intro-fully-faithful} at our disposal.
	\end{enumerate}
This perspective is particularly fruitful because we can prove strong results about the comparison map of a unigenic category --- which covers step (a) above --- especially when the category is connective (\cref{def:connective}):

\begin{Thm}\label{thm:intro-local}
	Let $\cT=\Ho(\cat C)$ be a rigidly-compactly generated tt-$\infty$-category which is connective and unigenic. The ungraded comparison map $\rho:\SpcT \to \Spec(\RT)$ is a closed quotient map whose fibers are local. 
\end{Thm}

This theorem is well illustrated by the familiar map $\rho:\Spc(\SH^c) \to \Spec(\bbZ)$ whose fibers each contain a unique closed point. The key to the theorem is \cref{prop:R-local-K-local} which establishes that, under the above hypotheses, each algebraic localization of $\cT$ is a local category (\cref{prop:R-local-K-local}). For this we utilize the theory of weight structures. Any connective and unigenic category~$\cT$ admits a canonical weight structure and we prove that the associated weight complex functor provides a splitting $\omega:\Spec(\RT)\hookrightarrow \SpcT$ of the comparison map, which embeds a copy of $\Spec(\RT)$ inside $\SpcT$. This copy of $\Spec(\RT)$ is precisely the subspace consisting of the unique relative closed points of the fibers of $\rho:\Spc(\cTc)\to\Spec(\RT)$. See \cref{thm:connective-unigenic}. We also provide another perspective on these weight complex functors, which may be of independent interest; see \cref{prop:weight-truncation}.

It is worth pointing out that~\cref{thm:intro-local} is false without the unigenic hypothesis. Nevertheless, combined with our results on unitation we obtain:

\begin{Thm}
	Let $\cT=\Ho(\cat C)$ be a rigidly-compactly generated tt-$\infty$-category which is connective. The ungraded comparison map $\rho:\SpcT \to \Spec(\RT)$ is a spectral quotient map whose fibers are connected.
\end{Thm}

This is proved in \cref{cor:ungraded-comparison}. As with \cref{thm:intro-fully-faithful}, $\rho$ is a topological quotient map if either of the spaces are noetherian; see \cref{rem:connective-noetherian}.
\medskip

We also obtain results for the graded comparison map:

\begin{Thm}
	Let $\cT$ be a rigidly-compactly generated tt-category. Suppose the graded ring $\RT\coloneqq \End_{\cat t}^\bullet(\unit)$ is coherent (e.g., noetherian). The graded comparison map $\rho:\SpcT\to\Spec(\RT)$ is a spectral quotient map.
\end{Thm}

This strengthens \cite[Theorem~7.3]{Balmer10b} and generalizes results in \cite[Section~2]{Lau23}. It is proved in~\cref{thm:coherent-spectral}. Again, a stronger statement is obtained if $\cT$ is further assumed to be unigenic; see \cref{prop:monogeniccoherent}.

\subsection*{Twisted cohomology and the Picard group}

Step (b) of the above method --- namely, understanding $\Spc(\cTc) \to \Spcunit{\cT}$ --- can be further refined by considering increasingly larger concentrations:
	\[
		\Spc(\cTc) \to \cdots \to \Spcconc{\cT}{\cat G_2} \to \Spcconc{\cT}{\cat G_1} \to \Spcunit{\cT}.
	\]
A significant step in this process is the concentration at the Picard group $\Pic(\cat T)$. This is of interest because all of the twisted comparison maps $\rho_u$ associated to elements $u \in \Pic(\cat T)$ factor through the map
	\begin{equation}\label{eq:picard}
		\Spc(\cT^c)\to\Spcconc{\cT}{\Pic(\cat T)}.
	\end{equation}
	A result of interest in this context is~\cref{prop:directed-subgroups}.

\subsection*{Affinization}
\addtocontents{toc}{\protect\enlargethispage{\baselineskip}}

Regarding the Balmer spectrum as a locally ringed space, we may consider its \emph{affinization}, that is, the universal morphism to an affine scheme
	\[
		\alpha: \Spec(\cTc) \to \Spec(A_{\cat T}).
	\]
The ungraded comparison map $\rho:\Spec(\cTc)\to\Spec(\RT)$ factors through the affinization and, in general, this factorization is distinct from the factorization through the unitation \eqref{eq:intro-factorization}. For derived categories of schemes $\cT=\Der(X)$, the affinization and the comparison map coincide (\cref{prop:affinization}). More generally, we establish several sufficient criteria for $\alpha$ and $\rho$ to coincide; see \cref{prop:aff-comp-domain} and \cref{prop:mayer-vietoris}. This is the case, for example, if $\cat T$ is connective and $\RT$ is a domain. These results provide another approach to understanding the comparison map. They also provide a means for easily establishing that, for certain categories, the Balmer spectrum is not a scheme.

\subsection*{Locally unigenic categories}
As we will see, it is possible for a fully faithful geometric functor to induce a homeomorphism on Balmer spectra without being an equivalence. However, we demonstrate in \cref{prop:homeo} that this only occurs due to the failure of the categories involved to be \emph{locally} unigenic (\cref{def:locally-unigenic}). This leads to the study of the \emph{unigenic locus} of a given category (\cref{def:unigenic-locus}) which we investigate for each of the equivariant examples mentioned above. For example, although the derived category of permutation modules is never unigenic, it is locally unigenic for an elementary abelian $p$-group. In general, the unigenic locus can be a proper nonempty subset of the spectrum; see the examples in \cref{sec:locally-unigenic}.

Derived categories of schemes are a prominent class of locally unigenic categories. Those schemes with \emph{unigenic} derived categories are studied in \cref{sec:affinization} and shown to properly contain the class of quasi-affine schemes. Together with the connection between affinization and the comparison map, this provides unusual examples of unigenic categories that help demonstrate the limits of our theorems. 

Overall, throughout this work, we illustrate our results with various examples and explore the extent to which the statements of our theorems are the best possible. The author hopes that these results will be helpful in the future study of the Balmer spectrum. For example, they already provide useful information about the derived category of motives over $\bbR$; see \cref{exa:vishik}.

\subsection*{Outline of the document}
We start in \cref{sec:spectral} by introducing the notion of a spectral quotient map. Various properties of these maps are established, including a characterization which will be useful in the proof of our main theorems. We then briefly recall the notion of a geometric functor in \cref{sec:geometric} and fix some notation used throughout. \Cref{thm:intro-faithful} is proved in \cref{sec:faithful} and \cref{thm:intro-fully-faithful} is proved in \cref{sec:connected-fibers}. Since faithful and fully faithful geometric functors are the objects of study in this work, we give various characterizations of these properties in~\cref{sec:characterizations}. We augment this discussion in \cref{sec:full} with a consideration of full geometric functors. This concludes \cref{part:I}.

In \cref{sec:unitation} we introduce the notions of \emph{concentration} and \emph{unitation} and illustrate these concepts with various examples. We briefly illustrate our main theorems in \cref{sec:motivic} with some motivic examples. We then turn to studying the comparison map of a tt-category in \cref{sec:graded}. Here we establish our theorems for the graded comparison map. We discuss the theory of weight structures in \cref{sec:weight} and apply this in \cref{sec:ungraded} to prove our theorems about the ungraded comparison map of a connective category. In \cref{sec:equivariant} we study in detail the unitations of various categories arising in equivariant higher algebra. We then briefly discuss the twisted comparison maps and the Picard group in \cref{sec:twisted}. We study locally unigenic categories and introduce the unigenic locus in \cref{sec:locally-unigenic}. In particular, we study the unigenic locus of the equivariant categories considered earlier. In \cref{sec:alg-geom}, we discuss examples of (fully) faithful functors in algebraic geometry before studying unigenic derived categories of schemes in \cref{sec:affinization}. Here we also consider affinization and its relation to the comparison map, which concludes the paper.

\subsection*{Notation and conventions}
\begin{enumerate}
	\item We will write $x \specializesto y$ to indicate that $x$ specializes to $y$; that is, $y \in \overbar{\{x\}}$.
	\item A tt-category is \emph{rigid} if every object is dualizable.
	\item Our schemes will always be quasi-compact and quasi-separated. We will write $\Der(X)\coloneqq \Derqc(X)$ for the derived category of complexes of $\OX$-modules whose cohomology groups are quasi-coherent.
	\item Unless stated otherwise, $G$ denotes a finite group and $k$ denotes a field of positive characteristic $p$.
\end{enumerate}

\subsection*{Terminology.}
We will utilize standard terminology and notation used in tensor triangular geometry. Any undefined terms or unspecified notation can be found in~\cite{BarthelHeardSanders23a}, \cite{Balmer10b} or \cite{DickmannSchwartzTressl19}.

\subsection*{Acknowledgements.}
The germ of the results of this paper were obtained during the author's stay at the Hausdorff Research Institute for Mathematics for the trimester program \emph{Spectral Methods in Algebra, Geometry, and Topology} funded by the Deutsche Forschungsgemeinschaft (DFG, German Research Foundation) under Germany’s Excellence Strategy – EXC-2047/1 – 390685813. The author thanks the participants of the Oberwolfach conference \emph{Tensor-triangular geometry and Interactions} for sharing their thoughts on terminology. He would also like to thank Paul Balmer, Tobias Barthel, Ivo Dell'Ambrogio, Martin Gallauer and Drew Heard for interesting discussions and helpful comments.

\newpage
\addtocontents{toc}{\vspace{-0.5\normalbaselineskip}}
\part{The geometry of fully faithful functors}\label{part:I}

In this first part, we study the tensor triangular geometry of fully faithful geometric functors. Examples and applications will be deferred to \cref{part:II}.

\section{Spectral spaces and quotient maps}\label{sec:spectral}

We begin with a discussion of spectral quotient maps. We will take for granted familiarity with basic notions concerning spectral spaces as discussed in \cite{DickmannSchwartzTressl19}.

\begin{Ter}\label{ter:spectral-quotient}
	A surjective continuous map $\varphi:X \to Y$ of topological spaces is a \emph{topological quotient map} if it has the property that a subset $C \subseteq Y$ is open whenever $\varphi^{-1}(C)$ is open. A surjective spectral map $\varphi:X \to Y$ of spectral spaces is a \emph{spectral quotient map} if a subset $C \subseteq Y$ is quasi-compact open whenever $\varphi^{-1}(C)$ is quasi-compact open.
\end{Ter}

\begin{Rem}\label{rem:spec-quot-vs-top-quot}
	If a spectral map $\varphi:X \to Y$ of spectral spaces is a topological quotient map then it is a spectral quotient map. On the other hand, \cite[Example~6.4.18]{DickmannSchwartzTressl19} provides an example of a spectral quotient map which is not a topological quotient map. Moreover, in this example the target space $Y$ is noetherian (see \cite[Example~8.1.4(iii)]{DickmannSchwartzTressl19}). Nevertheless, the two notions coincide if either:
	\begin{enumerate}
		\item $X$ is noetherian; or
		\item $Y$ is finite.
	\end{enumerate}
	The first case is immediate since every subspace of a noetherian space is quasi-compact. On the other hand, if $Y$ is finite then every point is constructible, hence the constructible topology is discrete (see \cite[Corollary~8.1.17]{DickmannSchwartzTressl19}). It follows that for any subset $C \subseteq Y$, the preimage $\varphi^{-1}(C)$ is constructible, hence is quasi-compact open if and only if it is open.
\end{Rem}

\begin{Rem}
	A composite of two spectral quotient maps is again a spectral quotient map. Moreover, if $\varphi:X \to Y$ is a spectral quotient map, then for any quasi-compact open $U \subseteq Y$, the corestriction $\varphi^{-1}(U) \to U$ is again a spectral quotient map.
\end{Rem}

\begin{Rem}
	The notion of a spectral quotient map has many equivalent characterizations, spelled out in \cite[Theorem~6.4.9]{DickmannSchwartzTressl19}. In particular, a surjective spectral map $\varphi:X\to Y$ is a spectral quotient if and only if a \emph{constructible} set $C \subseteq Y$ is quasi-compact open whenever $\varphi^{-1}(C)$ is quasi-compact open. Moreover, since the complement of a constructible set is again constructible, this is equivalent to saying that $\varphi$ is a spectral quotient if and only if a constructible set $C \subseteq Y$ is Thomason closed whenever $\varphi^{-1}(C)$ is Thomason closed.
\end{Rem}

\begin{Rem}
	Recall that a \emph{basic constructible set} is a set of the form $U \cap V^\cc$ with~$U$ and $V$ quasi-compact open, and that every constructible set is a finite union of basic constructible sets; see \cite[Proposition~1.3.13 and Corollary~1.3.15]{DickmannSchwartzTressl19}. As we shall see in \cref{thm:faithful-are-quotient} below, the following notion will be relevant for our study of faithful functors in tt-geometry.
\end{Rem}

\begin{Def}\label{def:weak-spectral-quotient}
	A surjective spectral map $\varphi:X \to Y$ is a \emph{weak spectral quotient map} if for every \emph{basic} constructible set $B \subseteq Y$, we have that $B$ is Thomason closed whenever $\varphi^{-1}(B)$ is Thomason closed.
\end{Def}

\begin{Rem}\label{rem:not-universally-weak}
	Every spectral quotient map is a weak spectral quotient map since if~$B\subseteq Y$ is a basic constructible set then its complement $B^{\cc}=Y \setminus B$ is constructible. On the other hand,  the following diagram
	\begin{equation}\label{eq:weak-spectral-not-spectral}
		\begin{tikzcd}[column sep=tiny,row sep=small]
		\bullet \ar[dd,dash,thick] \ar[dr,dash,thick] & &&\bullet \ar[d,dash,thick] \\
													  &\bullet \ar[rr] &&\bullet \ar[d,dash,thick] \\
		\bullet  & && \bullet 
	\end{tikzcd}
	\end{equation}
	depicts a weak spectral quotient which is not a spectral quotient. Note, however, that the corestriction 
	\[\begin{tikzcd}[column sep=tiny,row sep=small]
		\times &&\times \\
		\bullet \ar[rr]&& \bullet \ar[d,dash,thick] \\
		\bullet && \bullet
	\end{tikzcd}
	\]
	is not a weak spectral quotient. This leads to:
\end{Rem}

\begin{Def}\label{def:locally-heritable}
	A surjective spectral map $\varphi:X \to Y$ is a \emph{heritable weak spectral quotient map} if for each quasi-compact open $U\subseteq Y$, the corestriction
	\[
		\varphi^{-1}(U) \to U
	\]
	is a weak spectral quotient.
\end{Def}

\begin{Exa}
	The map \eqref{eq:weak-spectral-not-spectral} is a weak spectral quotient that is not a heritable weak spectral quotient.
\end{Exa}

\begin{Rem}
	Our primary goal at present is to characterize spectral quotient maps as precisely the heritable weak spectral quotient maps; see \cref{prop:univ-weak-spec-is-spec} below. This will take some preparation.
\end{Rem}

\begin{Def}\label{def:weak-lifting}
	Let $\varphi:X\to Y$ be a spectral map. Given points $y, y' \in Y$ with $y \specializesto y'$, we say that $\varphi$ satisfies \emph{weak lifting} from $y'$ to $y$ if there exists a sequence of points $x_0, x_1,\ldots,x_n \in X$ such that $\varphi(x_0)=y'$ and $\varphi(x_n)=y$ and with the property that $x_i$ is a specialization of a point in the same fiber as $x_{i+1}$. Visually:
	
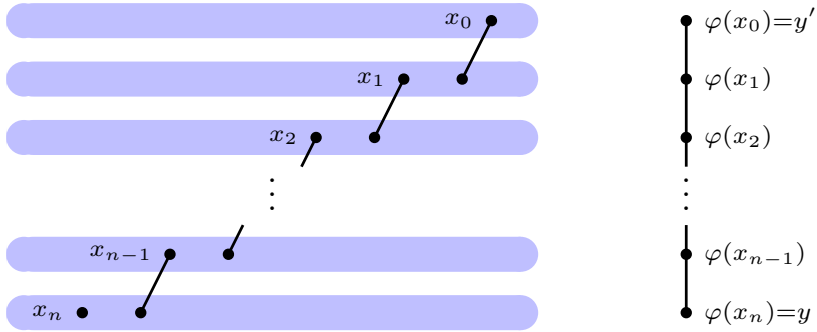
\begin{figure}[H]
		\resizebox{\textwidth}{!}{
	\makeatletter
\DeclareRobustCommand{\rvdots}{%
  \vbox{
    \baselineskip4\p@\lineskiplimit\z@
    \kern-\p@
    \hbox{.}\hbox{.}\hbox{.}
  }}
\makeatother

\def\labeloffset{0.2}
\def\diayoffbig{1.0}
\def\diayoffset{0.6}
\def\diaxoffset{0.3}
\def\bignubradius{0.6*\diaxoffset}
\def\smallnubradius{0.4*\diaxoffset}
\def\diabulletsize{0.03cm}
\def\diatinybulletsize{0.005cm}
\def\dianumbranches{9}
\def\diabranchlen{4}
\def\diabranchlenptwo{6}
\def\diasmallratio{0.25}
\def\diabigratio{0.5}
\def\diasmallopac{nearly opaque}
\def\diabigopac{nearly transparent}
\def\bigfillop{25}
\def\smallfillop{25}
\def\smallcol{blue}
\def\bigcol{blue}
\begin{tikzpicture}[trim left, scale=.25, every node/.style={scale=0.25}]

	\coordinate (base) at (-2,0);


%
%
%

\def\diayoffset{-0.6}
	\coordinate (start) at (base);
	\draw[fill] (start) circle (\diabulletsize);
	\node [anchor=west] at ($(start)+(\labeloffset*\diaxoffset,0)$) {$\scriptstyle \varphi(x_0)=y'$};
	\def\diabranchlen{2}
	\foreach \n in {1,...,\diabranchlen} {
		\draw [fill] ($(start)+(0,\n*\diayoffset)$) circle (\diabulletsize);
		\node [anchor=west] at ($(start)+(0,\n*\diayoffset)+(\labeloffset*\diaxoffset,0)$) {$\scriptstyle \varphi(x_\n)$};
	}
	\draw [very thin] (start) -- ($(start)+(0,\diabranchlen*\diayoffset)$);
	\draw [very thin] ($(start)+(0,\diabranchlen*\diayoffset)$) -- ($(start)+(0,\diabranchlen*\diayoffset+0.5*\diayoffset)$);
	\node [anchor=center]  at ($(start)+(0,\diabranchlen*\diayoffset+1*\diayoffset)$) {$\rvdots$};
	\coordinate (start) at ($(base)+(0,\diabranchlen*\diayoffset+2*\diayoffset)$);
	\draw [fill] ($(start)+(0,0*\diayoffset)$) circle (\diabulletsize);
	\node [anchor=west] at ($(start)+(0,0*\diayoffset)+(\labeloffset*\diaxoffset,0)$) {$\scriptstyle \varphi(x_{n-1})$};
	\draw [fill] ($(start)+(0,1*\diayoffset)$) circle (\diabulletsize);
	\node [anchor=west] at ($(start)+(0,1*\diayoffset)+(\labeloffset*\diaxoffset,0)$) {$\scriptstyle \varphi(x_{n})=y$};
	\draw [very thin] ($(start)+(0,-0.5*\diayoffset)$) -- ($(start)+(0,1*\diayoffset)$);


	\coordinate (start) at ($(base)+(-2,0)$);

	\coordinate (fibstart) at (start);
	\foreach \n in {0,...,2} {
		\path [fill=\bigcol!\bigfillop,rounded corners=2pt, thin] ($(fibstart)+(\diaxoffset,\n*\diayoffset)$) circle (\bignubradius);
		\path [fill=\bigcol!\bigfillop,rounded corners=2pt, thin] ($(fibstart)+(-16*\diaxoffset,\n*\diayoffset)$) circle (\bignubradius);
		\path [fill=\bigcol!\bigfillop,rounded corners=2pt, thin] 
		($(fibstart)+(\diaxoffset,\n*\diayoffset)+(0,\bignubradius)$)
		-- ($(fibstart)+(-16.6*\diaxoffset,\n*\diayoffset)+(0,\bignubradius)$)
		-- ($(fibstart)+(-16.6*\diaxoffset,\n*\diayoffset)-(0,\bignubradius)$)
		-- ($(fibstart)+(\diaxoffset,\n*\diayoffset)-(0,\bignubradius)$);
	}

		\draw[fill] (start) circle (\diabulletsize);
			\node [anchor=east] at ($(start)+(-\labeloffset*\diaxoffset,0)$) {$\scriptstyle x_0$};
	\foreach \n in {1,...,2} {
	\draw [very thin] 
		($(start)+(-3.0*\n*\diaxoffset+2*\diaxoffset,\n*\diayoffset)$) -- ($(start)+(-3.0*\n*\diaxoffset+3*\diaxoffset,\n*\diayoffset-\diayoffset)$);
		\draw[fill] ($(start)+(-3.0*\n*\diaxoffset+2*\diaxoffset,\n*\diayoffset)$) circle (\diabulletsize);
		\draw[fill] ($(start)+(-3.0*\n*\diaxoffset,\n*\diayoffset)$) circle (\diabulletsize);
			\node [anchor=east] at ($(start)+(-3.0*\n*\diaxoffset,\n*\diayoffset)+(-\labeloffset*\diaxoffset,0)$) {$\scriptstyle x_\n$};
	}
	\draw [very thin] ($(start)+(-3.0*3*\diaxoffset+2.5*\diaxoffset,2.5*\diayoffset)$) -- ($(start)+(-3.0*3*\diaxoffset+3*\diaxoffset,3*\diayoffset-\diayoffset)$);
	\node [anchor=center]  at ($(start)+(-3.0*3*\diaxoffset+1.5*\diaxoffset,\diabranchlen*\diayoffset+1*\diayoffset)$) {$\rvdots$};
	\coordinate (fibstart) at ($(start)+(0,\diabranchlen*\diayoffset+2*\diayoffset)$);
	\coordinate (start) at ($(start)+(-3.0*3*\diaxoffset,\diabranchlen*\diayoffset+2*\diayoffset)$);
	\foreach \n in {0,...,1} {
		\path [fill=\bigcol!\bigfillop,rounded corners=2pt, thin] ($(fibstart)+(\diaxoffset,\n*\diayoffset)$) circle (\bignubradius);
		\path [fill=\bigcol!\bigfillop,rounded corners=2pt, thin] ($(fibstart)+(-16*\diaxoffset,\n*\diayoffset)$) circle (\bignubradius);
		\path [fill=\bigcol!\bigfillop,rounded corners=2pt, thin] 
		($(fibstart)+(\diaxoffset,\n*\diayoffset)+(0,\bignubradius)$)
		-- ($(fibstart)+(-16.6*\diaxoffset,\n*\diayoffset)+(0,\bignubradius)$)
		-- ($(fibstart)+(-16.6*\diaxoffset,\n*\diayoffset)-(0,\bignubradius)$)
		-- ($(fibstart)+(\diaxoffset,\n*\diayoffset)-(0,\bignubradius)$);
	}
		\def\n{0} 
		\draw [fill] ($(start)+(-3.0*\n*\diaxoffset,\n*\diayoffset)$) circle (\diabulletsize);
		\draw [very thin] ($(start)+(-3.0*\n*\diaxoffset,\n*\diayoffset)$) --
($(start)+(-3.0*\n*\diaxoffset+0.5*\diaxoffset,\n*\diayoffset-0.5*\diayoffset)$);
		\draw[fill] ($(start)+(-3.0*\n*\diaxoffset-2*\diaxoffset,\n*\diayoffset)$) circle (\diabulletsize);
		\node [anchor=east] at ($(start)+(-3.0*\n*\diaxoffset-2*\diaxoffset,\n*\diayoffset)+(-\labeloffset*\diaxoffset,0)$) {$\scriptstyle x_{n-1}$};
		\def\n{1} 
	\draw [very thin] ($(start)+(-3.0*\n*\diaxoffset,\n*\diayoffset)$) -- ($(start)+(-3.0*\n*\diaxoffset+1*\diaxoffset,\n*\diayoffset-\diayoffset)$);
		\draw[fill] ($(start)+(-3.0*\n*\diaxoffset,\n*\diayoffset)$) circle (\diabulletsize);
		\draw[fill] ($(start)+(-3.0*\n*\diaxoffset-2*\diaxoffset,\n*\diayoffset)$) circle (\diabulletsize);
			\node [anchor=east] at ($(start)+(-3.0*\n*\diaxoffset-2*\diaxoffset,\n*\diayoffset)+(-\labeloffset*\diaxoffset,0)$) {$\scriptstyle x_n$};

\end{tikzpicture}
	}
	\end{figure}
	\noindent
	We say that $\varphi$ satisfies the \emph{weak lifting property} if the above holds for all specializations $y \specializesto y'$ in $Y$.
\end{Def}

\begin{Rem}
	The weak lifting property generalizes the lifting property described in \cite[Proposition~6.4.13]{DickmannSchwartzTressl19} which in turn generalizes both the going-up and going-down properties.
\end{Rem}

\begin{Prop}\label{prop:weak-lifting-is-quotient}
	If a spectral map $\varphi:X \to Y$ has the weak lifting property, then it is a topological quotient.	
\end{Prop}

\begin{proof}
	This is proved in a manner similar to \cite[Proposition~6.4.13]{DickmannSchwartzTressl19}. First note that the weak lifting property implies that $\varphi$ is surjective. Next let $Z \subseteq Y$ be an arbitrary subset and suppose $\varphi^{-1}(Z)$ is closed. Then in particular $\varphi^{-1}(Z)$ is closed under specialization. We claim that~$Z$ is similarly closed under specialization. To this end, let $y \in Z$ and let $y'$ be a specialization of~$y$. By the weak lifting property, there exist $x_0,x_1,\ldots,x_n \in X$ as in \cref{def:weak-lifting}. Note that $x_n \in \varphi^{-1}(Z)$ since $\varphi(x_n) = y \in Z$. Moreover, $x_{n-1}$ is a specialization of a point $w\in X$ with $\varphi(w) = \varphi(x_n) \in Z$. Thus $w \in \varphi^{-1}(Z)$ and hence $x_{n-1} \in \varphi^{-1}(Z)$ since the latter is closed under specialization. Continuing in this fashion we establish that each of the~$x_i$ is contained in $\varphi^{-1}(Z)$. In particular, $x_0 \in \varphi^{-1}(Z)$ so that $y'=\varphi(x_0)\in Z$. This establishes that $Z$ is specialization closed. It follows that $Z$ is closed since $Z=\varphi(\varphi^{-1}(Z))$ is proconstructible, being the image of a proconstructible set under a spectral map (\cite[Corollary~1.3.23]{DickmannSchwartzTressl19}). We have established that if $\varphi^{-1}(Z)$ is closed then $Z$ is closed. 
\end{proof}

\begin{Lem}\label{lem:spec-to-finite}
	A heritable weak spectral quotient $\varphi:X \to Y$ with $Y$ \emph{finite} is a topological quotient.
\end{Lem}

\begin{proof}
	We will establish that $\varphi$ satisfies the weak lifting property (\cref{def:weak-lifting}). Let $y_0 \in Y$. By considering the corestriction $\varphi^{-1}(\gen(y_0))\to\gen(y_0)$ we may assume that $Y$ is local with unique closed point $y_0$. Consider the collection $S \subseteq Y$ of all points $y$ to which we have weak lifting from $y_0$. We claim that $S=Y$. If not, then since $Y$ is finite, we can choose a point $y \in Y$ which is minimal among all points in $Y\setminus S$; that is, $y \not\in S$ but $\overbar{\{y\}} \setminus \{y\}$ is contained in $S$. Since $Y$ is finite, the singleton $\{y\}$ is basic constructible and hence the preimage $\varphi^{-1}(\{y\})$ is also basic constructible. Recall that a constructible set is specialization closed if and only if it is Thomason closed. Thus, since $\varphi$ is a weak spectral quotient and $\{y\}$ is not specialization closed (since $y_0$ is a specialization of $y$) we conclude that $\varphi^{-1}(\{y\})$ is not specialization closed. Hence there exists $x \in \varphi^{-1}(\{y\})$ and a specialization $x \rightsquigarrow x'$ with $x' \not\in\varphi^{-1}(\{y\})$. Thus $y=\varphi(x) \rightsquigarrow \varphi(x')$ is a proper specialization and hence $\varphi(x') \in S$. Hence we have weak lifting from the closed point $y_0$ to $\varphi(x')$. Since $x'$ is a specialization of $x$, we have established that $y_0$ has weak lifting to $\varphi(x)=y\not\in S$, which is a contradiction. We conclude that $S=Y$. This establishes that the original map has the weak lifting property and we can invoke \cref{prop:weak-lifting-is-quotient}.
\end{proof}

\begin{Prop}\label{prop:univ-weak-spec-is-spec}
	A surjective spectral map $\varphi:X \to Y$ is a spectral quotient if and only if it is a heritable weak spectral quotient.
\end{Prop}

\begin{proof}
	The ($\Rightarrow$) direction follows from the fact that spectral quotients are weak spectral quotients together with the readily checked fact that the corestriction of a spectral quotient to any quasi-compact open subset is again a spectral quotient. The main part is the ($\Leftarrow$) direction. The surjective spectral map $\varphi:X\to Y$ is an epimorphism in the category of spectral spaces (\cite[Theorem~5.2.5]{DickmannSchwartzTressl19}), hence it is a spectral quotient if (and only if) $\psi\circ \varphi:X \to Z$ is a spectral quotient for each spectral quotient $\psi:Y\to Z$ to a finite $Z$ by \cite[Proposition~6.4.12(vi)]{DickmannSchwartzTressl19}. Such a composite $\psi\circ \varphi$ is a heritable weak spectral quotient (as the composite of two such, invoking the $(\Rightarrow)$ direction) and hence is a topological (and hence spectral) quotient by \cref{lem:spec-to-finite}.
\end{proof}

Our strongest results will require the following strengthened versions:

\begin{Def}\label{def:strongly-heritable}
	A spectral map $\varphi:X \to Y$ is a \emph{strong spectral quotient map} if the corestriction $\varphi^{-1}(U) \to U$ is a spectral quotient map for every $U\subseteq Y$ which is the complement of a Thomason subset. Similarly, a surjective spectral map $\varphi:X \to Y$ is a \emph{strongly heritable weak spectral quotient map} if the corestriction $\varphi^{-1}(U) \to U$ is a weak spectral quotient for each $U \subseteq Y$ which is the complement of a Thomason subset.
\end{Def}

\begin{Cor}\label{cor:univ-weak-spec-is-strong-spec}
	A surjective spectral map $\varphi:X \to Y$ is a strong spectral quotient if and only if it is a strongly heritable weak spectral quotient.
\end{Cor}

\begin{proof}
	This readily follows from \cref{prop:univ-weak-spec-is-spec} and the definitions.
\end{proof}

\begin{Def}
	Similarly, we say that $\varphi:X \to Y$ is a \emph{strong topological quotient map} if the corestriction $\varphi^{-1}(U) \to U$ is a topological quotient map for every $U\subseteq Y$ which is the complement of a Thomason subset.
\end{Def}

\begin{Prop}\label{prop:strong-topological-equiv}
	Let $\varphi: X \to Y$ be a spectral map with $Y$ noetherian. The following are equivalent:
	\begin{enumerate}
		\item $\varphi$ is a strong topological quotient map.
		\item Any immediate specialization $y \specializesto y'$ in $Y$ lifts to a specialization $x \specializesto x'$ in~$X$.
		\item $\varphi$ satisfies the weak lifting property.
	\end{enumerate}
\end{Prop}

\begin{proof}
	$(a) \Rightarrow (b)$: Let $y\specializesto y'$ be an immediate specialization. Since $\varphi$ is a \emph{strong} topological quotient, the corestriction $\varphi^{-1}(\gen(y'))\to \gen(y')$ is a topological quotient. We may thus assume that $\varphi$ is a topological quotient to a local space~$Y$ and show that we can lift any immediate specialization $y \specializesto y'$ to the unique closed point $y' \in Y$. If there were no lift then $\varphi^{-1}(\{y\})$ would be specialization closed. Moreover, since the fiber $\varphi^{-1}(\{y\})$ is proconstructible, this would imply that it is closed (\cite[Theorem~1.5.4]{DickmannSchwartzTressl19}). Since $\varphi$ is a topological quotient, this would imply that $\{y\}$ is closed, which is false.

	$(b) \Rightarrow (c)$: Since $Y$ is noetherian, every specialization $y\specializesto y'$ is realized as a finite succession of immediate specializations (\cite[Theorem~8.1.11]{DickmannSchwartzTressl19}). Thus, the lifting of each immediate specialization shows that every specialization satisfies weak lifting according to~\cref{def:weak-lifting}.

	$(c) \Rightarrow (a)$: Since the weak lifting property is strongly heritable (because the complement of a Thomason subset is closed under generalization), this follows from~\cref{prop:weak-lifting-is-quotient}.
\end{proof}

\begin{Rem}
	The implication $(c) \Rightarrow (a)$ in \cref{prop:strong-topological-equiv} does not require the noetherian assumption.
\end{Rem}

\begin{Exa}\label{exa:closed-quotient-is-strong}
	If a spectral map $\varphi:X \to Y$ is a closed quotient map then it satisfies the going-up property and thus trivially satisfies the weak lifting property. Hence it is a strong topological quotient map.
\end{Exa}

\begin{Rem}
	We will use the above characterization of (strong) spectral quotients in the proof of \cref{thm:faithful-are-quotient} below. For the remainder of this section we will collect a few more facts about spectral quotients, especially pertaining to spectral quotients whose fibers are connected, with a view toward \cref{sec:connected-fibers}.
\end{Rem}

\begin{Ter}\label{ter:saturated}
	Let $\varphi:X \to Y$ be a function. A subset $S \subseteq X$ is \emph{saturated} (with respect to $\varphi$) if $S =\varphi^{-1}(\varphi(S))$.
\end{Ter}

\begin{Lem}\label{lem:spec-saturated}
	For a surjective spectral map $\varphi:X \to Y$, the following are equivalent:
	\begin{enumerate}
		\item $\varphi$ is a spectral quotient map;
		\item A saturated subset $S \subseteq X$ is quasi-compact open if and only if $\varphi(S)$ is quasi-compact open.
	\end{enumerate}
\end{Lem}

\begin{proof}
	This is routine from the definitions.
\end{proof}

\begin{Cor}\label{cor:spectral-homeo}
	A spectral quotient map $\varphi:X \to Y$ is a homeomorphism if and only if it is injective.
\end{Cor}

\begin{proof}
	If $\varphi$ is injective then every subset $S \subseteq X$ is saturated. Hence by \cref{lem:spec-saturated}, if $S$ is quasi-compact open then $\varphi(S)$ is quasi-compact open. Since the quasi-compact opens form a basis for the topology of $X$, it follows that $\varphi$ is an open map, and hence is a homeomorphism.
\end{proof}

\begin{Lem}\label{lem:clopen-saturated}
	Let $\varphi:X \to Y$ be a continuous map whose non-empty fibers are connected. If $S\subseteq X$ is closed and open then $S$ is saturated: $S=\varphi^{-1}(\varphi(S))$.
\end{Lem}

\begin{proof}
	This is routine from the definitions.
\end{proof}

\begin{Prop}\label{prop:preimage-connected}
	Let $\varphi :X\to Y$ be a spectral quotient map whose fibers are connected. If $Y$ is connected then $X$ is connected.
\end{Prop}

\begin{proof}
	Suppose $X=U_1 \sqcup U_2$ is a disjoint union of open subsets. In particular, each~$U_i$ is both closed and open. Hence by \cref{lem:clopen-saturated}, each $U_i$ is saturated. Moreover, since closed subsets of quasi-compact spaces are quasi-compact, each $U_i$ is quasi-compact open. Hence \cref{lem:spec-saturated} implies that $\varphi(U_i)$ is quasi-compact open in $Y$. Moreover, note that $\varphi^{-1}(\varphi(U_1)\cap \varphi(U_2)) = \varphi^{-1}(\varphi(U_1)) \cap \varphi^{-1}(\varphi(U_2))=U_1 \cap U_2=\emptyset$ so that $\varphi(U_1)\cap \varphi(U_2) = \emptyset$. Thus $Y=\varphi(X)=\varphi(U_1)\cup \varphi(U_2)$ is a disjoint union of open sets. Hence, since $Y$ is connected, we have $Y=\varphi(U_i)$ for some $i$ so that $X=\varphi^{-1}(\varphi(U_i))=U_i$.
\end{proof}

\begin{Cor}\label{cor:preimage-connected}
	Let $\varphi:X \to Y$ be a spectral quotient map whose fibers are connected. Then $\varphi^{-1}(C)$ is connected for any connected subset $C \subseteq Y$ which is either quasi-compact open or Thomason closed.
\end{Cor}

\begin{proof}
	If $C \subseteq Y$ is quasi-compact open or the complement of a quasi-compact open then the corestriction $\varphi^{-1}(C) \to C$ is itself a spectral quotient map (with connected fibers). The result thus follows from \cref{prop:preimage-connected}.
\end{proof}

\begin{Prop}\label{prop:strong-spec-topo}
	Let $\varphi:X \to Y$ be a strong spectral quotient map whose fibers are connected. If $Y$ is noetherian then $\varphi$ is a strong topological quotient map.
\end{Prop}

\begin{proof}
	By \cref{prop:weak-lifting-is-quotient}, it suffices to establish that any immediate specialization $y \specializesto y'$ in $Y$ lifts to a specialization $x \specializesto x'$ in $X$. Since $\varphi$ is a \emph{strong} spectral quotient map, the corestriction $\varphi^{-1}(\gen(y')) \to \gen(y')$ is a spectral quotient map. Thus, it suffices to assume $\varphi$ is a spectral quotient map to a local space $Y$ and prove that we can lift every immediate specialization $y \specializesto y'$ to the unique closed point $y' \in Y$. (This is the same reduction used in the proof of $(a) \Rightarrow (b)$ in \cref{prop:strong-topological-equiv}.) Note that $\{y,y'\} = \overbar{\{y\}}$ is a Thomason closed subset. Hence $\varphi^{-1}(\{y,y'\})$ is connected by \cref{cor:preimage-connected}. Note that $\varphi^{-1}(\{y,y'\}) = \overbar{\varphi^{-1}(\{y\})} \cup \varphi^{-1}(\{y'\})$ is a union of nonempty closed sets, and hence has a nontrivial intersection. Since the fiber $\varphi^{-1}(\{y\})$ is proconstructible, its closure is the same as its specialization closure. We conclude that there exists a specialization $x \specializesto x'$ which lifts $y \specializesto y'$. This completes the proof.
\end{proof}

\begin{Lem}\label{lem:get-rid-of-weak}
	Let $X$ be a spectral space and let $\{S_i\}_{i \in I}$ be a set of Thomason closed subsets with the property that any finite intersection of the $S_i$ is connected. Then the whole intersection $\bigcap_{i \in I} S_i$ is connected.
\end{Lem}

\begin{proof}
	Suppose $Z_1$ and $Z_2$ are Thomason closed subsets of $X$ such that
	\[
		\bigcap_{i \in I} S_i \subseteq Z_1 \cup Z_2 \quad\text{ and }\quad \bigcap_{i \in I} S_i \cap Z_1 \cap Z_2 = \emptyset.
	\]
	We need to show that $\bigcap_{i \in I} S_i \subseteq Z_k$ for some $k=1,2$. It suffices to only consider \emph{Thomason} closed subsets $Z_i$ since $\bigcap_{i \in I} S_i$ is closed and hence proconstructible; see \cite[Proposition~6.6.1]{DickmannSchwartzTressl19}. The first hypothesis can be rewritten as $Z_1^\cc \cap Z_2^\cc \subseteq \bigcup_{i \in I} S_i^\cc$ which is an open covering of the \emph{quasi-compact} open $Z_1^\cc \cap Z_2^\cc$. Hence there exists a finite subset $J_1 \subseteq I$ such that $\bigcap_{j \in J_1} S_j \subseteq Z_1 \cap Z_2$. On the other hand, the second hypothesis can be rewritten as $X=Z_1^\cc \cup Z_2^\cc \cup \bigcup_{i \in I} S_i^\cc$. Thus, the quasi-compactness of $X$ implies that there is a finite $J_2 \subseteq I$ such that ${\bigcap_{j \in J_2} S_j \cap Z_1 \cap Z_2 = \emptyset}$. Taking $J\coloneqq J_1 \cup J_2$, we obtain a finite intersection $\bigcap_{j \in J} S_j$ satisfying both $\bigcap_{j \in J} S_j \subseteq Z_1 \cup Z_2$ and $\bigcap_{j \in J} S_j \cap Z_1 \cap Z_j = \emptyset$. By hypothesis this finite intersection is connected. Hence there is a $1 \le k \le 2$ such that $\bigcap_{j \in J} S_j \subseteq Z_k$. In particular $\bigcap_{i \in I} S_i \subseteq Z_k$ and the proof is complete.
\end{proof}

\begin{Prop}\label{prop:cofiltered}
	Let $(X_i,\varphi_{ij})$ be a cofiltered diagram in the category of spectral spaces and let $X\coloneqq \lim_{i \in I} X_i$.
	\begin{enumerate}
		\item If the transition maps $\varphi_{ij}$ are surjective then each map $f_i :X\to X_i$ is surjective.
		\item If the transition maps $\varphi_{ij}$ are spectral quotient maps then each map $f_i:X \to X_i$ is a spectral quotient map.
	\end{enumerate}
\end{Prop}

\begin{proof}
	First recall that the forgetful functor from the category of spectral spaces to the category of topological spaces creates limits; see \cite[Section~11.1]{DickmannSchwartzTressl19}. Also note that we can assume without loss of generality that our diagram is indexed on a cofiltered \emph{set} (as our notation above suggests); see \cite{AndrekaNemeti82} or \cite[Section~1.A]{AdamekRosicky94}.

	$(a)$: Because we can equip each space with its constructible topology, we may assume without loss of generality that the spaces are quasi-compact Hausdorff spaces. It is well-known that a cofiltered limit of non-empty quasi-compact Hausdorff spaces is non-empty; see, e.g., \cite[Chapter VIII, Theorem 3.6, page 217]{EilenbergSteenrod52} or \cite[\href{https://stacks.math.columbia.edu/tag/0A2R}{Lemma 0A2R}]{stacks-project}. With this in hand, let $i \in I$ and suppose $x_i \in X_i$. For each $j \in I$ we define a subset $A_j \subseteq X_j$ as follows. If $i \le j$ then take the singleton set $A_j\coloneqq \{\varphi_{ij}(x_i)\}$. (In particular, $A_i = \{x_i\}$.) On the other hand, if $i \not\le j$ then choose any $k\in I$ such that $k \le i$ and $k\le j$ and take $A_j\coloneqq \varphi_{kj}(\varphi^{-1}_{ki}(\{x_i\}))$. A straightforward exercise, using the hypothesis that the transition maps are surjective, establishes that this definition does not depend on the choice of $k$. (For example, if $j \le i$ then $A_j = \varphi^{-1}_{ji}(\{x_i\})$.) From the definitions, one can check that for each $j \le j'$, the transition map $\varphi_{jj'}:X_j \to X_{j'}$ restricts to a map $A_j \to A_{j'}$. Thus the~$A_j$ form a cofiltered subdiagram of the the $X_j$. Note that these subsets $A_j$ are quasi-compact (as closed subsets in the constructible topology on $X_j$) and non-empty since the transition maps are surjective by hypothesis. Thus $\lim_i A_i \subseteq \lim_i X_i$ is non-empty. Any element of $\lim_i A_i$ provides an element in $X=\lim_i X_i$ which maps to $x_i$ under~$f_i$. This proves~$(a)$.

	$(b)$: Let $C \subseteq X_i$ be a subset and suppose $f_i^{-1}(C)$ is quasi-compact open in $X$. Then  \cite[Lemma~0A2P]{stacks-project} implies that there is a $j\in I$ and an open $U \subseteq X_j$ such that $f_i^{-1}(C) = f_j^{-1}(U)$. Moreover, since $f_j$ is surjective by part $(a)$, the open $U$ is quasi-compact. Since the diagram is cofiltered, there exists a $k$ such that $k\le i$ and $k\le j$. In other words $f_j$ and $f_i$ both factor through $f_k:X\to X_k$. Since $f_k$ is surjective by $(a)$, it follows that $\varphi_{ki}^{-1}(C) = \varphi_{kj}^{-1}(U)$. Thus $\varphi_{ki}^{-1}(C)$ is a quasi-compact open set. Since $\varphi_{ki}$ is a spectral quotient map, this implies $C$ is quasi-compact open. This establishes that $f_i$ is a spectral quotient map.
\end{proof}

\section{Geometric functors}\label{sec:geometric}

We take for granted familiarity with common notation and terminology in tensor triangular geometry as discussed, for example, in \cite[Sections~1-2]{BarthelHeardSanders23a}.

\begin{Hyp}\label{hyp:geometric-functor}
	In this paper, $f^*:\cT\to \cS$ will always denote a \emph{geometric functor} (\aka coproduct-preserving tensor-triangulated functor) between \emph{rigidly-compactly generated} tensor-triangulated categories. It restricts to a tensor-triangulated functor $f^*:\cTc \to \cSc$ between the compact(=dualizable) objects and hence induces a spectral map $\varphi\coloneqq \Spc(f^*):\SpcS\to \SpcT$ of spectral spaces.
\end{Hyp}

\begin{Rem}\label{rem:corestriction-is-fully-faithful}
	For any Thomason subset $Y \subseteq \SpcT$ with $V\coloneqq Y^\cc$, we have an induced geometric functor $\cT|_V \to \cS|_{\varphi^{-1}(V)}$ which realizes the corestriction $\smash{\varphi^{-1}(V)} \to V$ on spectra; see, e.g. \cite[Proposition~1.30]{BarthelHeardSanders23a} or \cite[Remark~5.10]{Sanders22}. Note that  \emph{both} squares in 
	\[\begin{tikzcd}
		\cT \ar[r,"f^*"] \ar[d,shift right=0.75ex] & \cS \ar[d,shift right=0.75ex] \\
		\cT|_V \ar[r] \ar[u,hook,shift right=0.75ex] & \cS|_{\varphi^{-1}(V)} \ar[u,hook,shift right=0.75ex]
	\end{tikzcd}\]
	commute. Hence, if $f^*$ is (fully) faithful then the induced functor is also (fully) faithful. In particular, for any prime $\cat P \in \SpcT$, we have an induced geometric functor
	\[
		\cT_{\cat P} = \cT|_{\gen(\cat P)} \to \cS|_{\varphi^{-1}(\gen(\cat P))}
	\]
	on the local category at $\cat P$ which is (fully) faithful if $f^*$ is (fully) faithful. See \cite[Terminology~1.11, Remark~1.21 and Definition~1.25]{BarthelHeardSanders23a} for further discussion.
\end{Rem}

\begin{Exa}
	If $\cat C$ is a rigidly-compactly generated tt-$\infty$-category (in the sense of \cite[Section~5]{BCHNPS_descent}) and $A \in \CAlg(\cat C)$ then $A\text{-Mod}_{\cat C}$ is also rigidly-compactly generated. The base-change functor $\Ho(\cat C)\to \Ho(A\text{-Mod}_{\cat C})$ is a geometric functor.
\end{Exa}

\begin{Exa}
	If $\cat K=\Ho(\cat B)$ is an essentially small idempotent-complete rigid \mbox{tt-category} which has an underlying model then $\cT \coloneqq \Ho(\Ind(\cat B))$ is a rigidly-compactly generated tt-category with $\cTc = \cat K$. Moreover, any symmetric monoidal exact functor $\cat A \to \cat B$ of underlying models extends (essentially uniquely) to a geometric functor $\Ho(\Ind(\cat A))\to\Ho(\Ind(\cat B))$. See \cite[Section~2.2]{BarthelCastellanaHeardValenzuela18} and \cite[Section~5]{BCHNPS_descent} for further details.
\end{Exa}

\begin{Not}
	We will write $\cT=\Ho(\cat C)$ on occasions when we are further assuming that $\cT$ is the homotopy category of a rigidly-compactly generated tt-$\infty$-category~$\cat C$. See \cite[Section~2]{Zou25pp} for further discussion concerning this terminology. Also, although the Balmer spectrum is an invariant of the homotopy category, we will occasionally write $\Spc(\cat C^c) \coloneqq \Spc(\Ho(\cat C^c))=\SpcT$ when notationally convenient.
\end{Not}

\section{Faithful functors are quotient maps}\label{sec:faithful}

\begin{Thm}\label{thm:faithful-are-quotient}
	Let $f^*:\cT \to \cS$ be a faithful geometric functor. Then the induced map $\varphi:\SpcS\to\SpcT$ is a strong spectral quotient map.
\end{Thm}

\begin{proof}
	Since $f^*$ is faithful, the surjectivity of $\varphi$ follows from \cite[Theorem~1.3]{Balmer18} or \cite[Theorem~1.4]{BarthelCastellanaHeardSanders24}. Suppose $U$ and $V$ are two quasi-compact open subsets of $\SpcT$. We can consider the exact triangle 
	\[
		\eY \to \unit \to \fY \to \Sigma \eY
	\]
	in $\cT$ associated to the Thomason closed set $Y \coloneqq U^\cc$ and we can consider an object~$x \in \cTc$ with $\supp(x)=V^\cc$. Then \cite[Proposition~2.29]{PatchkoriaSandersWimmer22} asserts that the morphism
	\[
		\fY \otimes x \to \Sigma \eY \otimes x
	\]
	vanishes if and only if $U \cap V^\cc$ is Thomason. If we apply the geometric functor $f^*$, we obtain the morphism
	\[
		f^*(\fY) \otimes f^*(x) \to \Sigma f^*(\eY) \otimes f^*(x)
	\]
	and note that $f^*(\eY) = \altmathbb{e}_{\varphi^{-1}(Y)}$ and $f^*(\fY)=\altmathbb{f}_{\varphi^{-1}(Y)}$ by \cite[Theorem~6.3]{BalmerFavi11} or \cite[Proposition~5.11]{BalmerSanders17}. Moreover, $\supp(f^*(x)) = \varphi^{-1}(\supp(x)) = \varphi^{-1}(V)^\cc$. Thus, if $f^*:\cT \to \cS$ is faithful, then $U \cap V^\cc$ is Thomason whenever $\varphi^{-1}(U) \cap \varphi^{-1}(V)^\cc$ is Thomason. In other words, a basic constructible set $C \subseteq \SpcT$ is Thomason whenever $\varphi^{-1}(C)$ is Thomason. Moreover, recall that a constructible set is Thomason if and only if it is Thomason closed. Thus, if $f^*$ is faithful then for any basic constructible set $C \subseteq \SpcT$, we have that $C$ is Thomason closed whenever $\varphi^{-1}(C)$ is Thomason closed. That is, $\varphi$ is a weak spectral quotient (\cref{def:weak-spectral-quotient}).

	Now, if $W \subseteq \SpcT$ is any complement of a Thomason subset, we can consider the corestriction $\cT|_W \to \cS|_{\varphi^{-1}(W)}$ which is again faithful by \cref{rem:corestriction-is-fully-faithful}. Hence, the above argument --- applied to these corestrictions --- shows that the map induced by a faithful functor is a \emph{strongly heritable} weak spectral quotient (\cref{def:strongly-heritable}) and we can invoke \cref{cor:univ-weak-spec-is-strong-spec}.
\end{proof}

\begin{Rem}\label{rem:faithful-not-closed}
	The spectral quotient map induced by a faithful functor need not satisfy the going-up property nor need it satisfy the going-down property (even if the spaces involved are noetherian); see \cref{exa:brenner} below. In particular, it need not be a closed map nor an open map; see \cite[Section~5.3]{DickmannSchwartzTressl19}.
\end{Rem}

\begin{Rem}
	It is possible for $\varphi$ to be a strong spectral quotient without $f^*$ being faithful. Thus, the converse of \cref{thm:faithful-are-quotient} is false. For example, if $f^*:\cT\to \cS$ is a descendable finite {\'e}tale morphism in the sense of~\cite{Sanders22} which has finite degree in the sense of \cite{Balmer14} then $\varphi$ is a closed quotient map by \cite[Theorem~1.5]{Balmer16b}; cf.~\cite[Remark~13.26]{BarthelCastellanaHeardSanders23app}. However, $f^*$ need not be faithful. For an explicit example, let $G$ be a finite $p$-group and let $k$ be a field of characteristic $p$. Restriction to elementary abelian subgroups provides a descendable finite {\'e}tale morphism $\StMod(kG) \to \prod_{E \le G} \StMod(kE)$ of finite degree which is only faithful in the trivial case where $G$ itself is elementary abelian; see \cite[Page~911]{Balmer16b}. For further examples of non-faithful quotient maps, see \cref{rem:faithful-not-all-quotients} or \cref{exa:frobenius}.
\end{Rem}

\section{Fully faithful functors have connected fibers}\label{sec:connected-fibers}

We next prove that the map on spectra induced by a fully faithful functor has connected fibers.

\begin{Lem}\label{lem:local-ring}
	Suppose $\cT$ is a local rigidly-compactly generated category. For any non-empty Thomason closed subset $Z \subseteq \SpcT$, the endomorphism ring $\End_{\cT}(\eZ)$ is a local ring.
\end{Lem}

\begin{proof}
	Since $\eZ\in \cT$ is an idempotent coring, the switch map $\tau :\eZ \otimes \eZ \xrightarrow{\sim} \eZ \otimes \eZ$ coincides with the identity map; cf.~\cite[Example~2.41]{Sanders22}. This implies that every endomorphism of $\eZ$ is ``tensor-balanced'' in the terminology of \cite[Definition~3.1]{Sanders13}. By \cite[Lemma~3.3]{Sanders13}, every such endomorphism $f:\eZ\to \eZ$ has the property that $f^{\otimes 2}\otimes \cone(f)=0$. However, we cannot invoke \cite[Proposition~3.5]{Sanders13} directly since the assumption that $\cT$ is local only asserts that the tensor product of two nonzero \emph{compact} objects is nonzero, while $\eZ$ is not necessarily compact. Nevertheless, we can proceed as follows.

	Recall that for any object $t \in \cT$ we have $t \in \Loco{\eZ}$ if and only if $t \simeq \eZ \otimes t$. In particular, since $\cone(f) \in \Loco{\eZ}$ we have $\cone(f) \simeq \cone(f) \otimes \eZ$. Since $Z\subseteq \SpcT$ is Thomason closed, there exists a compact object $x \in \cTc$ with $\supp(x)=Z$. Then $\Loco{x}=\Loco{\eZ}$ and it readily follows (for example, from \cite[Lemma~3.6]{BarthelHeardSanders23a}) that 
	\[
		\Loco{\cone(f)\otimes x}= \Loco{\cone(f)\otimes \eZ}= \Loco{\cone(f)}.
	\]
	Note that the object $\cone(f)\otimes x$ is compact as it is the cone of the endomorphism
	\[
		\eZ \otimes x \xrightarrow{f\otimes 1} \eZ \otimes x
	\]
	of the compact object $\eZ \otimes x \simeq x \in \cTc$. In summary, the endomorphism $f$ is a nonunit if and only if the compact object $\cone(f)\otimes x$ is nonzero.

	We conclude that if $f,g \in \End_{\cT}(\eZ)$ are two nonunits then $\cone(f)\otimes x\neq 0$ and $\cone(g)\otimes x\neq 0$. Since these objects are compact and $\cT$ is assumed to be local, it follows that $\cone(f)\otimes x \otimes \cone(g)\otimes x \neq 0$ and hence that 
	\begin{equation}\label{eq:conefconeg}
		\cone(f)\otimes\cone(g)\otimes x \neq 0.
	\end{equation}
	This implies that $f+g$ is also a nonunit. Indeed, it follows from $f^{\otimes 2}\otimes \cone(f)=0$ and $g^{\otimes 2}\otimes \cone(g)=0$ that $(f+g)^{\otimes n} \otimes \cone(f)\otimes \cone(g) = 0$ for $n \ge 3$ as in the proof of \cite[Proposition~3.5]{Sanders13}. Thus, if $f+g$ were a unit, then $(f+g)^{\otimes 3}\otimes \cone(f)\otimes\cone(g)$ is an endomorphism of $\eZ^{\otimes 3}\otimes \cone(f)\otimes \cone(g)$ which is both an isomorphism and the zero morphism. It follows that $\eZ \otimes \cone(f)\otimes\cone(g)\simeq \eZ^{\otimes 3}\otimes\cone(f)\otimes\cone(g)=0$. Tensoring with $x$, we obtain $x \otimes \cone(f)\otimes\cone(g)=0$, which contradicts \eqref{eq:conefconeg}. We have thus established that the sum of two nonunits is a nonunit; that is, the nonzero ring $\End_{\cT}(\eZ)$ is local.
\end{proof}

\begin{Prop}\label{prop:connected-preimage}
	Let $f^*:\cT \to \cS$ be a fully faithful geometric functor. If $\cT$ is local then the preimage $\varphi^{-1}(Z)$ of any nonempty specialization closed subset $Z \subseteq \SpcT$ is connected.
\end{Prop}

\begin{proof}
	We first prove the result for a Thomason closed subset $Z \subseteq \SpcT$. The endomorphism ring $\End_{\cT}(\eZ)$ is local by \cref{lem:local-ring}. Since $f^*$ is fully faithful, $\End_{\cS}(f^*(\eZ))$ is local. Moreover, note that $f^*(\eZ)=\altmathbb{e}_{\varphi^{-1}(Z)}$ is the left idempotent for the Thomason closed subset~$\varphi^{-1}(Z)$; see, e.g., \cite[Remark~13.7]{BarthelCastellanaHeardSanders23app}. Recall that for any two Thomason subsets $Y_1$ and $Y_2$, we have a Mayer--Vietoris exact triangle
	\[
		\altmathbb{e}_{Y_1 \cap Y_2} \to \altmathbb{e}_{Y_1} \oplus \altmathbb{e}_{Y_2} \to \altmathbb{e}_{Y_1 \cup Y_2} \to \Sigma \altmathbb{e}_{Y_1 \cap Y_2}
	\]
	by \cite[Theorem~5.18]{BalmerFavi11}. Suppose $\varphi^{-1}(Z) = Z_1 \sqcup Z_2$ is a disjoint union of closed sets. Note that the $Z_i$ are necessarily Thomason closed sets. For example, $Z_1^\cc = (\varphi^{-1}(Z))^\cc \cup Z_2$ is a union of two quasi-compact subsets and hence is quasi-compact. Thus we would obtain $f^*(\eZ) = \altmathbb{e}_{Z_1} \oplus \altmathbb{e}_{Z_2}$ from the Mayer--Vietoris exact triangle. But since the endomorphism ring is local, the object $f^*(\eZ)$ has no nontrivial idempotent endomorphisms. Hence $\altmathbb{e}_{Z_1}=0$ or $\altmathbb{e}_{Z_2} = 0$. That is, $Z_1=\emptyset$ or $Z_2=\emptyset$. In summary, the preimage $\varphi^{-1}(Z)$ is connected.

	We have proved the result in the case of a Thomason closed subset $Z$. Now suppose that $Z$ is an arbitrary (nonempty) closed subset. We may write $Z=\bigcap_{i \in I} Z_i$ as an intersection of nonempty Thomason closed sets. Then $\varphi^{-1}(Z) = \bigcap_{i \in I} \varphi^{-1}(Z_i)$. Each of the Thomason closed sets $S_i \coloneqq \varphi^{-1}(Z_i)$ is connected. Moreover, any finite intersection of the $S_i$ is the preimage of a finite intersection of the $Z_i$; but any finite intersection of the $Z_i$ is Thomason closed (and nonempty since it contains the unique closed point) and hence its preimage is connected by what we have already proved. Therefore, \cref{lem:get-rid-of-weak} implies that the arbitrary intersection $\varphi^{-1}(Z)=\bigcap_{i \in I}\varphi^{-1}(Z_i)$ is connected.

	Finally, suppose $Z$ is an arbitrary (nonempty) specialization closed subset. We may write $Z=\bigcup_{i \in I} Z_i$ as a union of closed subsets. Then $\varphi^{-1}(Z) = \bigcup_{i \in I} \varphi^{-1}(Z_i)$ is a union of connected sets which contain a common point (since each $Z_i$ contains the unique closed point). Hence $\varphi^{-1}(Z)$ is also connected.
\end{proof}

\begin{Thm}\label{thm:connected-fibersb}
	Let $f^*:\cT \to \cS$ be a fully faithful geometric functor. The fiber $\varphi^{-1}(\{\cat P\})$ over any point $\cat P \in \SpcT$ is connected.
\end{Thm}

\begin{proof}
	Performing corestriction to $\varphi^{-1}(\gen(\cat P)) \to \gen(\cat P)$ we obtain another fully faithful functor $\cT_{\cat P} \to \cS|_{\varphi^{-1}(\gen(\cat P))}$ as in \cref{rem:corestriction-is-fully-faithful}. Thus, we can assume without loss of generality that $\cT$ is local and prove that the fiber over the unique closed point $\frakm \in \SpcT$ is connected. We may then invoke \cref{prop:connected-preimage}.
\end{proof}

\begin{Thm}\label{thm:main-thmb}
	Let $f^*:\cT \to \cS$ be a fully faithful geometric functor. The induced map $\varphi:\SpcS \to \SpcT$ is a strong spectral quotient map with connected fibers.
\end{Thm}

\begin{proof}
	This just puts together \cref{thm:connected-fibersb} and \cref{thm:faithful-are-quotient}.
\end{proof}

\begin{Rem}\label{rem:spectral-is-strong-ff}
	The strong spectral quotient $\varphi:\SpcS\to \SpcT$ in \cref{thm:main-thmb} is a strong topological quotient if either of the spaces is noetherian. This follows from \cref{prop:strong-spec-topo}. In our later discussion, we'll often use this stronger consequence of the theorem without further comment.
\end{Rem}

\begin{Rem}
	If we only assume that $f^*$ is faithful (rather than \emph{fully} faithful) then the fibers of $\varphi$ need not be connected. For example, if $k$ is a field then the ring homomorphism $k \to k\times k$ induces a faithful geometric functor $\Der(k)\to \Der(k \times k)$ which on spectra is the projection of two disconnected points to a single point. For another example, see \cref{rem:inflation-in-SHG}.
\end{Rem}

\section{Characterizations of (fully) faithful functors}\label{sec:characterizations}

Since the results of the last two sections concern (fully) faithful geometric functors, it is worth including some remarks about when this holds.

\begin{Rem}
	Recall from \cite{BalmerDellAmbrogioSanders16} that the geometric functor $f^*:\cT \to \cS$ (\cref{hyp:geometric-functor}) automatically has a right adjoint $f_*$ which itself has a right adjoint~$f^!$:
	\[
		f^*\dashv f_*\dashv f^!
	\]
	The projection formulas of \cite[(2.16) and (2.18)]{BalmerDellAmbrogioSanders16} imply that we have natural isomorphisms
	\[
		f_*f^* \simeq f_*\unitS \otimes - \qquad\text{ and }\qquad
		f_*f^! \simeq \ihom{f_*\unitS,-}.
	\]
	Note that $f_*\unitS$ is a commutative ring object in $\cT$ whose unit map $u:\unitT \to f_*\unitS$ is given by the unit $\eta$ of the $f^*\dashv f_*$ adjunction: $\unitT \to f_*f^*\unitT \simeq f_*\unitS$.
\end{Rem}

\begin{Lem}\label{lem:unit-counit}
	For any $t \in \cT$, the following two diagrams commute:
	\[\begin{tikzcd}[column sep=large]
		t \ar[r,"\eta_t"] \ar[d,"\cong"'] & f_*f^*t \ar[d,"\cong"] \\
		\unitT \otimes t \ar[r,"u\otimes 1"] & f_*\unitS \otimes t
		\end{tikzcd}\quad \text{ and }\quad\begin{tikzcd}[column sep=large]
		f_*f^!t\ar[d,"\cong"'] \ar[r,"\epsilon_t"] & t \ar[d,"\cong"]\\
		\ihom{f_*\unitS,t} \ar[r,"\ihom{u,1}"] & \ihom{\unitT,t}.
	\end{tikzcd}\]
\end{Lem}

\begin{proof}
	The first diagram follows in a routine manner from the definition of the projection formula in \cite[(2.16)]{BalmerDellAmbrogioSanders16}. We guide the reader through the second diagram as it is perhaps a bit more curious. Note that it relates the counit~$\epsilon$ of the $f_*\dashv f^!$ adjunction with $u$ which is related to the unit of the $f^*\dashv f_*$ adjunction. The isomorphism $f_*f^!t \simeq \ihom{f_*\unitS,t}$ arises from the $s=\unitS$ case of the adjunction isomorphism $f_*\ihom{s,f^!t}\simeq \ihom{f_*s,t}$ of \cite[(2.18)]{BalmerDellAmbrogioSanders16} as $f_*f^!t \simeq f_*\ihom{\unitS,f^!t} \simeq \ihom{f_*\unitS,t}$. By definition, the adjunction isomorphism is the map adjoint to the following composite:
	\[
		f_*\ihom{s,f^!t}\otimes f_*s \xrightarrow{\text{lax}} f_*(\ihom{s,f^!t}\otimes s) \xrightarrow{f_*(\text{ev})} f_*f^!t \xrightarrow{\epsilon_t}t.
	\]
	Since the evaluation map $\ihom{\unitS,f^!t}\otimes \unitS \to f^!t$ coincides with the composite $\ihom{\unitS,f^!t}\otimes \unitS \simeq f^!t\otimes \unitS \simeq f^!t$, one readily checks that the isomorphism $f_*f^!t \simeq \ihom{f_*\unitS,t}$ is adjoint to the map
	\begin{equation}\label{eq:lax}
		f_*f^!t \otimes f_*\unitS \xrightarrow{\text{lax}} f_*(f^!t \otimes \unitS) \simeq f_*f^!t \xrightarrow{\epsilon_t} t.
	\end{equation}
	Dinaturality of coevaluation provides
	\[\begin{tikzcd}[column sep=large]
		f_*f^!t \ar[r,"\text{coev}"] \ar[d,"\text{coev}"'] & \ihom{f_*\unitS,f_*f^!t \otimes f_*\unitS} \ar[d,"\ihom{u,1}"] \\
		\ihom{\unitT,f_*f^!t \otimes \unitT} \ar[r,"\ihom{1,1\otimes u}"] & \ihom{\unitT,f_*f^!t \otimes f_*\unitS}.
	\end{tikzcd}\]
	Armed with the above, one readily checks that the composite
	\[
		f_*f^!t \simeq \ihom{f_*\unitS,t} \xrightarrow{\ihom{u,1}}\ihom{\unitT,t}\simeq t
	\]
	coincides with
	\[
		f_*f^!t \simeq f_*f^!t \otimes \unitT \xrightarrow{1\otimes u} f_*f^!\unitT \otimes f_*\unitS \xrightarrow{\text{lax}} f_*(f^!t \otimes \unitS) \simeq f_*f^!\unitT \xrightarrow{\epsilon_t} t
	\]
	and that the composite $f_*f^!t \to f_*f^!t$ coincides with the identity map.
\end{proof}

\begin{Lem}\label{lem:invertible-ring}
	Let $R$ be a ring object in a symmetric monoidal category $\cat C$. If $R$ is an invertible object then the unit map $u:\unit \to R$ is an isomorphism.
\end{Lem}

\begin{proof}
	Since $R$ is invertible, the endofunctor $R\otimes -:\cat C \to \cat C$ is an equivalence of categories. Thus, there exists a morphism $\theta:R\to\unit$ such that 
	\[
		R\otimes R \xrightarrow{1\otimes\theta} R\otimes\unit \simeq R
	\]
	coincides with the multiplication map $m:R \otimes R \to R$. It follows that $R\otimes \unit \xrightarrow{1\otimes u} R\otimes R \xrightarrow{1\otimes \theta} R\otimes \unit$ coincides with $\id_{R\otimes \unit} = R\otimes \id_{\unit}$. Hence $\theta\circ u = \id_\unit$. On the other hand, the commutative diagram
	\[\begin{tikzcd}
		R \ar[r,"\theta"] \ar[d,"\simeq"] & \unit \ar[d,"\simeq"] \ar[dr,bend left,"\id"] & \\
		R\otimes \unit \ar[r,"\theta\otimes 1"] \ar[d,"1\otimes u"] & \unit \otimes \unit \ar[d,"1\otimes u"] \ar[r,"\simeq"] & \unit \ar[d,"u"] \\
		R\otimes R \ar[r,"\theta\otimes 1"] & \unit \otimes R \ar[r,"\simeq"] & R
	\end{tikzcd}\]
	shows that $u\circ \theta =\id_R$.
\end{proof}

\begin{Prop}\label{prop:fully-faithful-chars}
	The following are equivalent:
	\begin{enumerate}
		\item $f^*$ is fully faithful.
		\item The unit map $u:\unitT \to f_*\unitS$ is an isomorphism.
		\item There exists an isomorphism $\unitT \simeq f_*\unitS$.
		\item $f_*\unitS$ is invertible.
		\item $f^!$ is fully faithful.
	\end{enumerate}
\end{Prop}

\begin{proof}
	The left adjoint $f^*$ is fully faithful if and only if the unit $\eta_t : t \to f_*f^*t$ is a natural isomorphism. By \cref{lem:unit-counit}, this is equivalent to $u:\unitT\to f_*\unitS$ being an isomorphism. That is, $(a) \Leftrightarrow (b)$. The implications $(b)\Rightarrow(c)\Rightarrow(d)$ are immediate and \cref{lem:invertible-ring} provides $(d)\Rightarrow(b)$. It remains to deal with $(e)$. The right adjoint~$f^!$ is fully faithful if and only if the counit $\epsilon_t:f_*f^!t \to t$ is a natural isomorphism. By \cref{lem:unit-counit}, this is equivalent to 
	\begin{equation}\label{eq:homu}
	\ihom{u,1}:\ihom{f_*\unitS,t}\to\ihom{\unitT,t}
	\end{equation}
	being an isomorphism for all $t \in \cT$. This is certainly implied by $(b)$. On the other hand, applying $\cT(\unit,-)$ to \eqref{eq:homu}, we see that if $f^!$ fully faithful then
	\[
		\cT(f_*\unit,t) \xrightarrow{-\circ u} \cT(\unit,t)
	\]
	is a bijection for all $t \in \cT$. This implies $u$ is an isomorphism by (co)Yoneda.
\end{proof}

\begin{Prop}\label{prop:faithful-chars}
	The following are equivalent:
	\begin{enumerate}
		\item $f^*$ is faithful.
		\item The unit map $u:\unitT \to f_*\unitS$ is split monic.
		\item $\unitT$ is a direct summand of $f_*\unitS$.
		\item $f_*\unitS \otimes -$ is faithful.
		\item $\ihom{f_*\unitS,-}$ is faithful.
		\item $f^!$ is faithful.
	\end{enumerate}
\end{Prop}

\begin{proof}
	Consider the exact triangle
	\begin{equation}\label{eq:std-triangle}
		W \xrightarrow{\xi} \unitT \xrightarrow{u} f_*\unitS \rightarrow \Sigma W.
	\end{equation}
	associated with $u$. By the unit-counit equations, $f^*(u)$ is split monic; that is, $f^*(\xi)=0$. Thus, $f^*$ faithful implies $\xi = 0$, which implies $u$ is split monic. Thus $(a) \Rightarrow (b)$. The implications $(b)\Rightarrow(c)\Rightarrow(d)$ are immediate. Recall that $f_*\unitS \otimes -\simeq f_*f^*(-)$. Hence $(d) \Rightarrow (a)$.

	$(a) \Rightarrow (e)$: Since $f^*$ is faithful, we know the map $\xi$ in \eqref{eq:std-triangle} is zero. Suppose that~${\alpha:a\to b}$ is a morphism in $\cT$. Consider the following commutative diagram
	\[\begin{tikzcd}[column sep=large]
		\ihom{f_*\unitS,a} \ar[d,"\ihom{1,\alpha}"] \ar[r,"\ihom{u,1}"] & \ihom{\unitT,a} \ar[d,"\ihom{1,\alpha}"] \ar[r,"0"] & \ihom{W,a}\hphantom{.} \ar[d,"\ihom{1,\alpha}"] \\
		\ihom{f_*\unitS,b} \ar[r,"\ihom{u,1}"] & \ihom{\unitT,b} \ar[r,"0"] & \ihom{W,b}.
	\end{tikzcd}\]
	If $\ihom{f_*\unitS,\alpha}=0$ then the morphism $\alpha\simeq \ihom{\unitT,\alpha}$ factors through the zero morphism $0:a\simeq \ihom{\unit,a} \to \ihom{W,a}$.

	$(e) \Rightarrow (f)$: This follows from the isomorphism $f_*f^!(-) \simeq \ihom{f_*\unitS,-}$. Finally, we prove $(f)\Rightarrow (a)$, namely that $f^!$ faithful implies $f^*$ faithful. We have a natural isomorphism of two variables $f^!\ihom{a,t}\simeq \ihom{f^*a,f^!t}$ from \mbox{\cite[(2.19)]{BalmerDellAmbrogioSanders16}.} Let~$\alpha:a\to b$ be a morphism in $\cT$ such that $f^*(\alpha)=0$. Then for any $t \in \cT$, we have $f^!\ihom{\alpha,t}=0$. Since $f^!$ is faithful, this implies that $\ihom{\alpha,t}=0$ for all $t \in \cT$. In particular, $\ihom{\alpha,b}=0$. Consider an exact triangle on $\alpha$:
	\[
		a \xrightarrow{\alpha} b \xrightarrow{\beta} c \xrightarrow{\gamma} \Sigma a
	\]
	Applying $\ihom{-,b}$ we obtain an exact triangle
	\[
		\ihom{c,b} \to \ihom{b,b}\xrightarrow{\ihom{\alpha,b}=0} \ihom{a,b} \to \Sigma\ihom{c,b}.
	\]
	Applying $\cT(\unitT,-)$ we obtain an exact sequence
	\[
		\cT(c,b) \to \cT(b,b) \xrightarrow{0} \cT(a,b).
	\]
	Hence, there exists a morphism $\theta:c \to b$ such that $\theta\circ\beta = \id_b$. In other words, $\beta$ is split monic; hence $\alpha=0$.
\end{proof}

\begin{Rem}\label{rem:faithful-is-descendable}
	It follows from \cref{prop:faithful-chars} that a faithful geometric functor is a \emph{descendable} geometric functor in the sense that $\unit_{\cT} \in \thickt{f_*(\unit_{\cS})}$. Hence:
\end{Rem}

\begin{Prop}
	Let $f^*:\cT\to \cS$ be a faithful geometric functor. Then $f^*$ reflects compact objects.
\end{Prop}

\begin{proof}
	This follows from the proof of \cite[Proposition~3.28]{Mathew16b} using \cref{rem:faithful-is-descendable}; see \cite[Proposition~7.16]{BCHNPS_descent}.
\end{proof}

\begin{Rem}
	Another consequence of \cref{rem:faithful-is-descendable} is that homological stratification in the sense of \cite{BarthelHeardSandersZou24pp} descends along a faithful geometric functor $f^*:\cT \to \cS$. Hence, stratification in the sense of \cite{BarthelHeardSanders23a} descends from $\cS$ to $\cT$ provided these categories satisfy the Nerves of Steel conjecture.
\end{Rem}

\begin{Prop}\label{prop:ff-on-compact}
	Let $f^*:\cT \to \cS$ be a geometric functor. If $f^*|_{\cTc}:\cTc \to \cSc$ is fully faithful then $f^*:\cT\to\cS$ is fully faithful.
\end{Prop}

\begin{proof}
	According to \cref{prop:fully-faithful-chars}, $f^*$ is fully faithful if (and only if) $u:\unitT\to f_*\unitS$ is an isomorphism. Since $\cT$ is compactly generated, this is the case if $\cT(c,\unitT)\to \cT(c,f_*\unitS)$ is an isomorphism for each $c \in \cTc$. This map is readily checked to coincide with $\cT(c,\unitT)\to\cS(f^*(c),f^*(\unitT))\cong \cS(f^*(c),\unitS)\cong \cT(c,f_*\unitS)$.
\end{proof}

\begin{Rem}\label{rem:equiv-cons}
	A fully faithful geometric functor $f^*:\cT \to \cS$ is an equivalence if and only if its right adjoint $f_*:\cS \to \cT$ is conservative. Moreover, the right adjoint $f_*$ of a geometric functor $f^*:\cT \to \cS$  is conservative if and only if $\thick\langle f^*(\cTc)\rangle = \cSc$. It then follows from \cref{prop:ff-on-compact} that if $f^*|_{\cTc}:\cTc \to \cSc$ is an equivalence then $f^*:\cT \to \cS$ is an equivalence.
\end{Rem}

\begin{Rem}\label{rem:conservative-on-rings}
	Although the right adjoint $f_*$ of a nontrivial fully faithful functor $f^*:\cT \to \cS$ is not conservative, we do have the following:
\end{Rem}

\begin{Prop}
	Let $f^*:\cT\to\cS$ be a a fully faithful geometric functor. The right adjoint $f_*$ is conservative on weak rings.
\end{Prop}

\begin{proof}
	Let $A$ be a weak ring in $\cS$ and let $\unitS\to A$ be its unit. We have a commutative diagram 
	\[\begin{tikzcd}
		f^*f_*(\unitS)\ar[d,"\epsilon"] \ar[r]& f^*f_*(A)\ar[d,"\epsilon"] \\
		   \unitS\ar[r] & A
	\end{tikzcd}\]
	and the counit $f^*f_*(\unitS)\to \unitS$ is an isomorphism since $f^*$ is fully faithful. Thus, if $f_*(A)=0$ then the unit $\unitS\to A$ vanishes and hence $A=0$.
\end{proof}

\begin{Rem}\label{rem:bousfield}
	According to \cref{prop:fully-faithful-chars}, $f^*$ fully faithful is equivalent to $f^!$ fully faithful, which in turn is equivalent to $f_*:\cS \to \cT$ being a Bousfield localization. Note however that the localizing subcategory of acyclic objects $\Ker(f_*)\subseteq \cS$ need not be a tensor-ideal. An explicit counter-example is given in \cref{exa:Z2-graded-product}. (Bousfield localizations on tensor-triangulated categories are sometimes assumed to have a tensor-ideal of acyclic objects. This is the case, for example, in \cite{HoveyPalmieriStrickland97}.)
\end{Rem}

\begin{Rem}
	The author does not know if $f^*|_{\cTc}$ faithful implies $f^*$ faithful. We just remark in passing that this holds if $\cT$ is phantomless. Indeed, if $f^*|_{\cTc}$ is faithful then the map $\xi$ in \eqref{eq:std-triangle} is a phantom map.
\end{Rem}

\section{Full sometimes implies faithful}\label{sec:full}

It is well-known that a full exact functor between triangulated categories is faithful if and only if it is conservative; see the proof of \cite[Lemma~2.1]{CanonacoOrlovStellari13} for example. In the context of big tt-categories, we can augment this as follows:

\begin{Prop}
	For a full geometric functor $f^*:\cT \to \cS$, the following are equivalent:
	\begin{enumerate}
		\item $f^*$ is faithful.
		\item $f^*$ is conservative.
		\item $\varphi$ is surjective.
		\item $f^*$ reflects invertibility of endomorphisms of $\unit$; that is, the ring homomorphism $\End_{\cT}(\unit)\to \End_{\cS}(\unit)$ reflects units.
	\end{enumerate}
\end{Prop}

\begin{proof}
	We have $(a)\Rightarrow(b)\Rightarrow(c)$ by \cite[Theorem~1.4]{BarthelCastellanaHeardSanders24}.

	$(c) \Rightarrow (d)$: Let $\alpha:\unitT\to\unitT$ be a morphism in $\cT$. Since $\varphi$ is surjective,
	\[
		\supp(\cone(f^*(\alpha))) = \varphi^{-1}(\supp(\cone(\alpha))) = \emptyset
	\]
	implies $\supp(\cone(\alpha))=\emptyset$. In other words, if $f^*(\alpha)$ is an isomorphism then $\alpha$ is an isomorphism.

	$(d)\Rightarrow (a)$: We claim that $f^*$ is faithful, which by \cref{prop:faithful-chars} is equivalent to the canonical map $\unitT \to f_*(\unitS)$ being split monic. Since $f^*$ is assumed to be full, there exists a morphism $\theta:f_*f^*\unitT \to \unitT$ such that $f^*(\theta)$ is the morphism
	\[
		f^*f_*f^*\unitT \xrightarrow{\epsilon_{f^*\unitT}} f^*\unitT.
	\]
	Thus, the composite $\unitT \xrightarrow{\eta} f_*f^*\unitT \xrightarrow{\theta} \unitT$ becomes the identity morphism $f^*\unitT \to f^*\unitT$ after we apply $f^*$; that is, $f^*(\theta\circ\eta) = \id_{f^*\unitT}$. Hypothesis $(d)$ then implies that $\theta\circ \eta$ is an isomorphism; hence $\eta$ is split monic.
\end{proof}

\begin{Prop}
	For a geometric functor $f^*:\cT \to \cS$, the following are equivalent:
	\begin{enumerate}
		\item $f^*$ is full.
		\item The unit map $u:\unitT \to f_*\unitS$ is split epi.
	\end{enumerate}
\end{Prop}

\begin{proof}
	Using the definition of the projection formula \cite[(2.16)]{BalmerDellAmbrogioSanders16}, one readily checks that for any $t \in \cT$ and $s \in \cS$, the identity map on $f_*(s) \otimes t$ factors as
	\[
		f_*(s) \otimes t \xrightarrow{1\otimes \eta} f_*(s) \otimes f_*f^*(t) \xrightarrow{\mathrm{lax}} f_*(s\otimes f^*(t)) \simeq f_*(s) \otimes t.
	\]
	With this in hand, one then checks that the composite
	\[
		\cT(a,b) \xrightarrow{f^*} \cS(f^*a,f^*b) \simeq \cT(a,f_*f^*b)\simeq \cT(a,f_*\unitS\otimes b)
	\]
	sends $\alpha:a\to b$ to $(u\otimes 1)\circ \alpha$ where $u\otimes 1$ is shorthand for $b \simeq \unitT \otimes b \xrightarrow{u\otimes 1} f_*\unitS \otimes b$. Thus, $f^*$ is full if and only if, for every $a,b \in \cT$, the map
	\begin{equation}\label{eq:full}
		\cT(a,b) \xrightarrow{(u\otimes 1)\circ -}\cT(a,f_*\unitS\otimes b)
	\end{equation}
	is surjective. If $u$ is split epi then so is $u\otimes 1$ and it follows that \eqref{eq:full} is surjective. On the other hand, taking $a=f_*\unitS$ and $b=\unitT$, the surjectivity of \eqref{eq:full} implies that there exists $\alpha:f_*\unitS\to \unitT$ such that $u\circ \alpha=\id_{f_*\unitS}$, so that $u$ is split epi.
\end{proof}

\begin{Exa}
	Let $A \to B$ be a map of commutative rings. The induced functor $\Der(A) \to \Der(B)$ is full if and only if $A \to B$ is a split epimorphism in the category of $A$-modules if and only if $B=A/Ae$ for an idempotent $e\in A$ if and only if $\Spec(B)\to \Spec(A)$ is a closed immersion that is also an open immersion. These are standard exercises in commutative algebra.
\end{Exa}

\begin{Rem}
	In light of \cref{prop:ff-on-compact}, one may wonder whether $f^*|_{\cTc}:\cTc \to \cSc$ full implies that $f^*:\cT \to \cS$ is full. This is false. In \cite[Section~5]{CanonacoOrlovStellari13}, the authors consider the quotient $A\to A/\mathfrak m$ of a certain non-noetherian commutative ring by a certain maximal ideal. They prove that $\Der(A)^c\to \Der(A/\mathfrak m)^c$ is full. However, since this~$\mathfrak m$ is not finitely generated, $A/\mathfrak m$ is not a finitely presented $A$-module and hence is not a direct summand of $A$. Thus, $\Der(A) \to \Der(A/\mathfrak m)$ is not full.
\end{Rem}

\newpage
\addtocontents{toc}{\vspace{-0.5\normalbaselineskip}}
\part{Examples and applications}\label{part:II}

\section{Concentration and unitation}\label{sec:unitation}

Let $\cT$ be a rigidly-compactly generated tt-category.

\begin{Def}
	We say that $\cT$ is \emph{unigenic} if it is generated by the unit object: 
	\[
		\cT=\Loc\langle\unit\rangle.
	\]
\end{Def}

\begin{Rem}
	The term ``monogenic'' is used in \cite{HoveyPalmieriStrickland97}. However, it is slightly misleading. For example, the derived category of a quasi-compact and \mbox{quasi-separated} scheme is always generated by a single compact object by \cite[Theorem~3.1.1]{BondalVandenbergh03} and yet it is rare for it to be generated by the unit object; see \cref{sec:affinization} below.
\end{Rem}

\begin{Def}\label{def:concentration}
	Let $\cat G \subseteq \cTc$ be a set of compact objects such that $\thick\langle \cat G \rangle \subseteq \cTc$ is a rigid tensor-subcategory. The localizing subcategory $\Loc\langle \cat G \rangle$ is then a tensor-triangulated subcategory of $\cT$ which is rigidly-compactly generated by $\cat G$ and the inclusion $\Loc\langle \cat G \rangle \hookrightarrow \cT$ is a fully faithful geometric functor. We call $\Loc\langle \cat G \rangle \subseteq \cT$ the \emph{concentration of $\cT$ at $\cat G$} and denote it by $\Tconc{\cat G}$.
\end{Def}

\begin{Exa}
	If $\cat G \subseteq \cTc$ is a set of objects which contains $\unit$ and is closed under the $\otimes$-product and taking duals then $\thick\langle \cat G\rangle$ is a rigid tensor-subcategory. This is the most natural case of \cref{def:concentration}, but there are also examples of interest where $\cat G$ is not closed under the $\otimes$-product and yet $\cat G \otimes \cat G \subseteq \thick\langle \cat G\rangle$. Hence we allow more flexibility in the definition, as above.
\end{Exa}

\begin{Def}\label{def:unitation}
	In particular, we can consider the concentration at $\cat G=\{\unit\}$, which we call the \emph{unitation} of $\cT$. We also call $\unitation{\cT}\coloneqq \Loc\langle\unit\rangle \subseteq\cT$ the \emph{unigenic core} (or \emph{unicore}) of $\cT$.\footnote{In ancient mythology, a \emph{manticore} is a man-eating beast with the head of a human, the body of a lion, and the tail of a scorpion \cite{Robinson65,Nichols11}. The author hopes that the unicore of a tt-category will have more in common with a unicorn than a manticore.}
\end{Def}

\begin{Exa}\label{exa:picard}
	More generally, we can consider the concentration at any subgroup $\cat G$ of the Picard group $\Pic(\cT)$.
\end{Exa}

Various examples of ``cellular objects'' arise in this way.

\begin{Exa}
	The motivic stable homotopy category $\SH(\mathbb{C})$ over the complex numbers is rigidly-compactly generated. The category of \emph{cellular motivic spectra} $\SHcell(\mathbb{C})$ is $\Loc\langle S^{m,n} : m,n\in\mathbb{Z}\rangle \subset \SH(\mathbb{C})$. In the above terminology, it is the concentration of $\SH(\mathbb{C})$ at the motivic spheres.
\end{Exa}

\begin{Exa}\label{exa:DM}
	Let $k$ be a field and let $R$ be a commutative ring in which the exponential characteristic of $k$ is invertible. The derived category of motives $\DM(k;R)$ is rigidly-compactly generated; see \cite[Example~5.17]{Sanders22} and the references cited therein. The derived category $\DTM(k;R)$ of \emph{Tate motives} is $\Loc\langle R(j) : j\in \bbZ\rangle \subset \DM(k;R)$. In other words, it is the concentration of $\DM(k;R)$ at the Tate twists.
\end{Exa}

\begin{Exa}\label{exa:artin-motives}
	The derived category $\DAM(k;R)$ of \emph{Artin motives} is the concentration of $\DM(k;R)$ at the motives of all finite separable extensions $L/k$. Note that these generating motives are self-dual dualizable objects of $\DM(k;R)$ and so are trivially closed under taking duals.
\end{Exa}

\begin{Exa}\label{exa:DATM}
	Combining the above two examples, the category $\DATM(k;R)$ of \emph{Artin--Tate motives} is the concentration of $\DM(k;R)$ at the Tate twists of the motives of finite separable extensions of $k$.
\end{Exa}

\begin{Exa}
	Let $G$ be a finite group. The $G$-equivariant stable homotopy category $\SHG=\Ho(\Sp_G)$ is rigidly-compactly generated by the orbits $G/H_+$ associated to the subgroups $H \le G$. These generators are self-dual (since we are assuming $G$ is finite). Let $\cat F$ be a collection of subgroups of $G$ which is closed under conjugation and consider 
	\begin{equation}\label{eq:F-G/H}
		\thick\langle G/H_+ : H \in \cat F\rangle \subseteq \SHG^c.
	\end{equation}
	It follows from the Mackey formula that $\cat F$ is closed under passage to subgroups if and only if \eqref{eq:F-G/H} is a thick \emph{tensor-ideal} of $\SHG^c$. In contrast, \eqref{eq:F-G/H} is a (necessarily rigid) tt-subcategory of $\SHG^c$ if and only if $\cat F$ contains $G \le G$ and is closed under intersection of subgroups. Thus, in the latter case we can consider the concentration of $\SHG$ at $\cat G \coloneqq \SET{G/H_+}{H\in \cat F}$.
\end{Exa}

\begin{Exa}\label{exa:SHG-inflation}
	Let $N \lenormal G$ be a normal subgroup and consider the collection of subgroups $\cat F[{{\supseteq}N}] \coloneqq \SET{H \le G}{H \supseteq N}$. Inflation provides a geometric functor $\SHGN \to (\SHG)_{\langle \cat F[\supseteq N[\rangle}$. In particular, we have a geometric functor $\SH \to \unitation{(\SHG)}$. These functors are far from being equivalences in general, as is clear just by considering the endomorphisms of the unit. We will study the unitation $\unitation{(\SHG)}$ in more detail in \cref{sec:equivariant}.
\end{Exa}

\begin{Exa}
	Let $\Gamma$ be a profinite group and let $R$ be a commutative ring. The derived category of permutation modules $\DPerm(\Gamma;R)$ defined in \cite{BalmerGallauer23b} is rigidly-compactly generated by the (transitive) permutation modules $R(\Gamma/H)$ for each open subgroup $H\le \Gamma$. It again follows from the Mackey formula (see \cite[Remark~2.12]{BalmerGallauer23b}) that we can consider the concentration at $\cat G = \SET{\Gamma/H}{H \in \cat F}$ for any collection of open subgroups $\cat F$ which contains $\Gamma \le \Gamma$ and is closed under finite intersection.
\end{Exa}

\begin{Exa}\label{exa:DPerm-inflation}
	Let $N \lenormal\, \Gamma$ be a closed (but not necessarily open) subgroup and consider the concentration of $\DPerm(\Gamma;R)$ at $\cat F[{\supseteq}N] = \SET{R(\Gamma/H)}{H \supseteq N}$. Inflation induces an equivalence $\DPerm(\Gamma/N;R) \cong \DPerm(\Gamma;R)_{\langle \cat F[{\supseteq}N]\rangle}$; see \cite[Lemma~3.15]{BalmerGallauer23b}. In particular, the unitation $\unitation{\DPerm(\Gamma;R)} \cong \Der(R)$ is just the derived category of the coefficient ring $R$. Note that this behaviour differs qualitatively from the situation in \cref{exa:SHG-inflation}.
\end{Exa}

\begin{Rem}
	If $\Gamma$ is the absolute Galois group of a field $k$ then there is a \mbox{tt-equivalence} $\DPerm(\Gamma;R) \simeq \DAM(k;R)$ with the derived category of Artin motives from \cref{exa:artin-motives}. This is explained in \cite[Sections 5--7]{BalmerGallauer23b}. Note that they define $\DAM(k;R)$ as the analogous localizing subcategory in the derived category of \emph{effective} motives $\DMeffkR$ rather than in $\DM(k;R)$. However, it follows from Voevodsky's Cancellation Theorem \cite{Voevodsky10} together with \cite[Section 8]{CisinskiDeglise15} that the canonical functor $\DMeffkR \to \DM(k;R)$ is fully faithful. Thus 
	\[
		\Loc_{\DMeff}\langle L \mid L/k \text{ finite separable}\rangle \to \Loc_{\DM}\langle L \mid L/k \text{ finite separable} \rangle
	\]
	is an equivalence.
\end{Rem}

\begin{Exa}
	Let $G$ be a finite group and let $R$ be a commutative ring. The category of derived representations $\DRep(G;R)=\Ho(\Rep(G;R))$ considered in \cite{barthel2021rep1} is rigidly-compactly generated. If $R=k$ is a field of characteristic $p$ then $\DRep(G;k)\cong K(\Inj(kG))$ and we have $\DRep(G;k)^c \cong \Dbmod{kG}$. More generally, if $R$ is noetherian then $\DRep(G;k)^c$ is equivalent to $\Der^b_{R\text{-perf}}(\text{mod}(RG))$ where the latter is the subcategory of $\smash{\Dbmod{RG}}$ consisting of complexes whose underlying nonequivariant complex is $R$-perfect.
\end{Exa}

\begin{Exa}\label{exa:DRepR}
	Let $R$ be a noetherian ring. We can consider the concentration of $\DRep(G;R)$ at the (transitive) permutation modules $\cat G=\SET{R(G/H)}{H \le G}$:
	\begin{equation}\label{eq:conc-perm}
		\DRep(G;R)_{\langle\cat G\rangle} \subseteq \DRep(G;R).
	\end{equation}
	This category is a finite localization of $\DPerm(G;R)$, as established in \cite{BalmerGallauer22b}:
	\begin{equation}\label{eq:conc-perm2}
		\DPerm(G;R) \twoheadrightarrow \DRep(G;R)_{\langle\cat G\rangle} \hookrightarrow \DRep(G;R).
	\end{equation}
	If $R$ is regular (e.g., $R=k$ a field) then Mathew \cite[Theorem~A.4]{Treumann15pp} and Balmer--Gallauer \cite{BalmerGallauer23a,BalmerGallauer22b} prove that $\DRep(G;R)$ is generated by permutation modules, i.e., that \eqref{eq:conc-perm} is an equality.\footnote{This also holds if $|G|$ is invertible in $R$.} In this case \eqref{eq:conc-perm2} reduces to a localization
	\begin{equation}\label{eq:conc-perm3}
		\DPerm(G;R) \twoheadrightarrow \DRep(G;R).
	\end{equation}
	However, \eqref{eq:conc-perm} need not be an equivalence in general; see \cite[Example~A.3]{Treumann15pp}. In general, the map induced on Balmer spectra by $\DPerm(G;R)\to\DRep(G;R)$ factors as 
	\begin{equation}\label{eq:conc-perm4}
		\Spc(\DRep(G;R)^c) \twoheadrightarrow \Spc(\concentration{\DRep(G;R)}{\cat G}^c) \hookrightarrow \Spc(\DPerm(G;R)^c)
	\end{equation}
	where the first map is a quotient map with connected fibers (\cref{thm:main-thmb}) and the second map is an embedding.
\end{Exa}

\begin{Rem}\label{rem:faithful-restriction}
	Let $\cT(G;R)$ denote either $\DPerm(G;R)$ or $\DRep(G;R)$. If $H \le G$ is a subgroup whose index $[G:H]$ is invertible in $R$ then the restriction functor ${\res^G_H:\cT(G;R) \to \cT(H;R)}$ is faithful. Indeed, this holds for any $R$-linear cohomological Green 2-functor; see \cite{BalmerDellAmbrogio24,DellAmbrogio22}. In contrast, the restriction functor $\res^G_H:\SH_G \to \SH_H$ is not faithful for any proper subgroup $H \lneq G$.
\end{Rem}

\begin{Exa}
	The category of spectra $\SH=\Ho(\Sp)$ is unigenic. More generally, the derived category $\Der(\bbE)\coloneqq\Ho(\bbE\text{-Mod}_{\Sp})$ of any commutative ring spectrum $\bbE \in \CAlg(\Sp)$ is unigenic. For example, the derived category $\Der(R)\cong \Der(\HR)$ of any commutative ring $R$ is unigenic.
\end{Exa}

\begin{Rem}\label{rem:spectra}
	Suppose $\cT=\Ho(\cat C)$ is the homotopy category of an underlying presentably symmetric monoidal stable $\infty$-category. We can consider the essentially unique symmetric monoidal and colimit-preserving functor $i^*:\Sp \to \cat C$ from the \mbox{$\infty$-category} of spectra. We have an induced ring spectrum
	\[
		\eend_{\cat C}(\unit) \coloneqq i_*(\unit) \in \CAlg(\Sp)
	\]
	and a canonical geometric functor $\eend_{\cat C}(\unit)\text{-Mod}_{\Sp} \to \cat C$. See \cite[Section~5.3]{MathewNaumannNoel17}, for example. This induces an equivalence
	\[
		\eend_{\cat C}(\unit)\text{-Mod}_{\Sp} \xrightarrow{\simeq} \unitation{\cat C}
	\]
	of symmetric monoidal stable $\infty$-categories. In particular, we have a tt-equivalence
	\[
		\unitation{\cT}=\Ho(\unitation{\cat C}) \cong \Ho(\eend_{\cat C}(\unit)\text{-Mod}_{\Sp}) \eqqcolon \Der(\eend_{\cat C}(\unit)).
	\]
	For example, if $\cT=\Ho(\cat C)$ is unigenic then it is the derived category of a highly structured commutative ring spectrum. This derived Morita theory has its homotopical origins in \cite{SchwedeShipley03}.
\end{Rem}

\begin{Not}\label{not:triv}
	Let $G$ be a finite group and let $\Sp_G$ denote the $\infty$-category of \mbox{$G$-spectra}. We will denote the right adjoint of the canonical functor ${\triv_G:\Sp\to\Sp_G}$ of \cref{rem:spectra} by $\lambda^G:\Sp_G \to \Sp$. We will also write $\Sphere_G \coloneqq \triv_G(\Sphere)$ for the \mbox{$G$-equivariant} sphere spectrum.
\end{Not}

\begin{Rem}\label{rem:equiv-monadic}
	Let $\mathbb{E} \in \CAlg(\Sp_G)$ and consider $\cat C\coloneqq \mathbb{E}\text{-Mod}_{\Sp_G}$. We have a commutative diagram
	\[\begin{tikzcd}
		\Sp \ar[r,"\triv_G"] \ar[d] & \Sp_G \ar[d] \\
		 \unitation{(\mathbb{E}\text{-Mod}_{\Sp_G})} \ar[r,hook] & \mathbb{E}\text{-Mod}_{\Sp_G}.
	\end{tikzcd}\]
	Since the bottom functor is fully faithful, it follows that $\eend_{\unitation{\cat C}}\hspace{-0.3ex}(\unit) \cong \lambda^G(\mathbb{E})$. Hence we have an equivalence
	\[
		\unitation{(\mathbb{E}\text{-Mod}_{\Sp_G})} \cong \lambda^G(\mathbb{E})\text{-Mod}_{\Sp}
	\]
	by \cref{rem:spectra}.
\end{Rem}

\begin{Exa}
	Let $\mathbb{E} \in \CAlg(\Sp)$. We may construct its associated Borel equivariant spectrum $b_G(\mathbb{E}) \coloneqq F(EG_+,\triv_G(\mathbb{E}))\in \CAlg(\Sp_G)$ and consider $b_G(\mathbb{E})\text{-Mod}_{\Sp_G}$. We have
	\[
		(b_G(\mathbb{E})\text{-Mod}_{\Sp_G})_{\langle \unit \rangle} \cong F(BG_+,\mathbb{E})\text{-Mod}_{\Sp}
	\]
	where the ring structure on $\lambda^G(b_G(\mathbb{E})) \simeq F(BG_+,\mathbb{E})$ arises from the ring structure on $\mathbb{E}$ together with the coring structure on $BG_+$ provided by the diagonal map.
\end{Exa}

\begin{Exa}\label{exa:DRep-monadic}
	If $R$ is a regular noetherian ring of finite Krull dimension then there is a symmetric monoidal equivalence $\Rep(G;R) \cong b_G(\HR)\text{-Mod}_{\Sp_G}$ explained in \cite[Theorem~3.7]{barthel2021rep1} which depends on \cite[Theorem~A.4]{Treumann15pp}. Hence
	\[
		\unitation{\DRep(G,R)} \cong \Ho(F(BG_+,\HR)\text{-Mod}_{\Sp}).
	\]
\end{Exa}

\begin{Exa}\label{exa:HRG}
	Let $G$ be a finite group and let $R$ be a commutative ring. We define $\HRG \coloneqq \triv_G(\HR) \in \CAlg(\Sp_G)$ so that $\Der(\HRG) \coloneqq \Ho(\HRG\text{-Mod}_{\Sp_G})$ is the category of derived $G$-Mackey functors with $R$-linear coefficients; see \cite{PatchkoriaSandersWimmer22} and \cite[Part~V]{BarthelHeardSanders23a}. It is rigidly-compactly generated by the images of the generators~$G/H_+$ of $\SHG$. We have an equivalence
	\[
		\unitation{\Der(\HRG)} \cong \Ho(\HR \wedge \lambda^G(\Sphere_G)\text{-Mod}_{\Sp}).
	\]
\end{Exa}

\begin{Rem}
	We will study these equivariant examples in more detail in \cref{sec:equivariant} and \cref{sec:locally-unigenic}.
\end{Rem}

\begin{Exa}\label{exa:noncommutative-motives}
	Let $k$ be a commutative ring and let $\cat T' \coloneqq \Mot_k^a$ and $\cat S'\coloneqq \Mot_k^\ell$ denote the tt-categories of noncommutative motives over $k$ for additive invariants and localizing invariants, respectively. Although $\cat T'$ is compactly generated, the category $\cat S'$ is not known to be compactly generated. Nevertheless, in either case, the localizing subcategory generated by the dualizable objects is a rigidly-compactly generated tt-category: $\cat T\coloneqq \concentration{\cat T'}{\cat T'^d} = \Loc(\cat T'^d)$ and $\cat S\coloneqq\concentration{\cat S'}{\cat S'^d} = \Loc(\cat S'^d)$. For further details, see \cite{CisinskiTabuada12,DellAmbrogioTabuada12,Tabuada08}.
\end{Exa}

\begin{Exa}
	Let $X$ be a quasi-compact and quasi-separated scheme. The derived category $\Der(X)\coloneqq \Derqc(X)$ of quasi-coherent complexes of sheaves (i.e.~complexes of $\OX$-modules whose cohomology sheaves are quasi-coherent) is rigidly-compactly generated. We'll study these examples in \cref{sec:alg-geom} and \cref{sec:affinization}.
\end{Exa}

\begin{Rem}
	We will also use the ``concentration'' and ``unitation'' terminology in the context of a small rigid tt-category $\cat K$, namely, if $\cat G \subseteq \cat K$ is a set of objects such that $\thick\langle \cat G\rangle$ is a rigid tensor-subcategory, we will write $\concentration{\cat K}{\cat G} \coloneqq \thick\langle \cat G \rangle$ and call it the \emph{concentration} of $\cat K$ at $\cat G$. Note that if $\cat K=\cTc$ then $\concentration{\cat K}{\cat G} = (\concentration{\cT}{\cat G})^c$.
\end{Rem}

\begin{Rem}
	Suppose $F:\cat K \to \cat L$ is a tensor-triangulated functor. If $\thick\langle \cat G \rangle \subseteq \cat K$ is a rigid tt-subcategory then $\thick\langle F(\cat G)\rangle$ is a rigid tt-subcategory of $\cat L$. Showing that it is a tt-subcategory of $\cat L^d$ is a routine thick subcategory argument. Moreover,  $\cat G^{\vee} \subseteq \thick\langle \cat G\rangle \subseteq F^{-1}(\thick\langle F(\cat G)\rangle)$ implies that $F(\cat G)^{\vee} = F(\cat G^\vee) \subseteq \thick\langle F(\cat G)\rangle$ so that $\thick\langle F(\cat G)\rangle$ is rigid.
\end{Rem}

\begin{Rem}
	It follows that a geometric functor $f^*:\cT \to \cS$ (\cref{hyp:geometric-functor}) induces a geometric functor $\concentration{\cT}{\cat G} \to \concentration{\cS}{f^*(\cat G)}$ of concentrations whose right adjoint is conservative (\cref{rem:equiv-cons}). For example, the right adjoint of $\smash{\unitation{\cT} \to \unitation{\cS}}$ is conservative.
\end{Rem}

\begin{Exa}
	If $\cT$ is unigenic then any geometric functor $f^*:\cT\to\cS$ factors as 
	\[
		\cT \to \unitation{\cS} \hookrightarrow \cS
	\]
	and the right adjoint of the first functor is conservative. In particular, $\cS$ is unigenic if and only if $f_*:\cS \to \cT$ is conservative.
\end{Exa}

\begin{Exa}\label{exa:finite-loc-is-unigenic}
	Any finite localization $\cT|_U$ of a unigenic category $\cT$ is unigenic.
\end{Exa}

\begin{Exa}\label{exa:unigenic-fully-faithful}
	If $\cT$ is unigenic then any fully faithful geometric functor $\cT\to \cS$ induces an equivalence $\cT \xrightarrow{\sim} \unitation{\cS}$ by \cref{rem:equiv-cons}.
\end{Exa}

\begin{Rem}\label{rem:geom-between-unigenic}
	A geometric functor $f^*:\cT \to \cS$ between unigenic categories is an equivalence if and only if the canonical map $\unitT\to f_*\unitS$ is an isomorphism if and only if $\End_{\cT}^*(\unitT) \to \End_{\cS}^*(\unitS)$ is an isomorphism. The first equivalence follows from \cref{exa:unigenic-fully-faithful} and \cref{prop:fully-faithful-chars}. On the other hand, since $\cT$ is unigenic, the map $u:\unitT\to f_*\unitS$ is an isomorphism if and only if post-composition by $u$ gives a bijection $\cT(\Sigma^n \unitT,\unitT) \to \cT(\Sigma^n \unitT,f_*\unitS)$ for all $n \in \bbZ$ and one readily checks that this map coincides with $\cT(\Sigma^n\unitT,\unitT) \to \cS(f^*(\Sigma^n\unitT),f^*(\unitT)) \cong \cS(\Sigma^n\unitS,\unitS)$.
\end{Rem}

\begin{Rem}\label{rem:unitation-localization}
	If $U\subseteq \Spcunit{\cT}$ is the complement of a Thomason subset then we have a commutative diagram
	\[\begin{tikzcd}
		\unitation{\cT} \ar[r,hook] \ar[d] & \cT \ar[d] \\
		(\unitation{\cT})|_U \ar[r,hook] & \cT|_{\varphi^{-1}(U)}
	\end{tikzcd}\]
	by \cref{rem:corestriction-is-fully-faithful}. Since the bottom-left category is unigenic (\cref{exa:finite-loc-is-unigenic}) the fully faithful bottom functor induces an equivalence: $(\unitation{\cT})|_U \xrightarrow{\sim} \unitation{(\cT|_{\varphi^{-1}(U)})}$.
\end{Rem}

\begin{Exa}\label{exa:finite-product}
	Suppose $\cT = \cT_1 \times \cdots \times \cT_n$ is a finite product of rigidly-compactly generated tt-categories with the pointwise tensor-triangulated structure. Then $\cT$ is itself rigidly-compactly generated. Moreover, $\cT$ is unigenic if and only if each~$\cT_i$ is unigenic. Indeed, the projection $\pi_i^*:\cT \to \cT_i$ is a finite localization whose (fully faithful) right adjoint $(\pi_i)_*$ sends $x \in \cT_i$ to $(0,\ldots,x,\ldots,0)$. Hence $\cT$ unigenic implies that each $\cT_i$ is unigenic. On the other hand, observe that for every $t \in \cT$ we have $t = \oplus_{i=1}^n (\pi_i)_*\pi_i^*(t) = \oplus_{i=1}^n \unit_i \otimes t$ where $\unit_i = (0,\ldots,\unit,\ldots,0)$. Note that $(\pi_i)_*(t) \in \Loc\langle\unit\rangle$ in $\cT_i$ implies $\unit_i \otimes t \in \Loc\langle\unit_i\rangle$ in $\cT$ and $\unit_i$ is a direct summand of the unit $\unit$ of $\cT$. Thus, $t \in \Loc\langle\unit\rangle$ in $\cT$ and we conclude that $\cT$ is unigenic.
\end{Exa}

\begin{Rem}
	An infinite product of rigidly-compactly generated tt-categories need not be rigidly-compactly generated; see \cite[Example~2.6]{Gomez24a} and \cite[Remark~1.6]{BarthelCastellanaHeardSanders24}.
\end{Rem}

\begin{Exa}\label{exa:Z2-graded-product}
	Let $\cT$ be a nonzero rigidly-compactly generated tt-category and let $\cT \times_{\bbZ/2} \cT$ be the product triangulated category with the $\bbZ/2$-graded $\otimes$-structure discussed in \cite[Example~4.23]{Sanders22}. It is rigidly-compactly generated and the inclusion $x\mapsto (x,0)$ into the first factor provides a fully faithful geometric functor $\cT \hookrightarrow \cT \times_{\bbZ/2} \cT$ whose right adjoint is not conservative. This category $\cT \times_{\bbZ/2} \cT$ is never unigenic. Indeed, the unit $(\unit,0)$ only generates objects concentrated in the first factor. Note that if $\cat T$ is unigenic, then the unitation $\unitation{(\cT \times_{\bbZ/2} \cT)}$ is $\cT$ sitting inside as the first factor.
\end{Exa}

\begin{Exa}
	Let $\cat F$ be any tt-field in the sense of \cite[Definition~1.1]{BalmerKrauseStevenson19}. Then $\cat F \times_{\bbZ/2} \cat F$ is again a tt-field. This shows that a tt-field need not be unigenic.
\end{Exa}

\begin{Rem}
	Any concentration $\Tconc{\cat G}$ of a tt-field $\cT$ is itself a tt-field by \cite[Proposition~5.22]{BalmerKrauseStevenson19}. Also, every geometric functor defined on a tt-field is faithful; see \cite[Corollary~5.6]{BalmerKrauseStevenson19}.
\end{Rem}

\section{Motivic examples}\label{sec:motivic}

We will now briefly illustrate our results for some motivic categories of interest.

\begin{Exa}
	Gallauer \cite{Gallauer19} has given a beautiful computation of the spectrum of the category $\DTM(\overbar{\bbQ},\bbZ)$ of Tate motives over $\overbar{\bbQ}$. This space is noetherian and can be regarded as a two-dimensional version of $\Spec(\bbZ)$ in which the layer of closed points corresponds to mod-$p$ motivic cohomology, the middle layer corresponds to mod-$p$ \'{e}tale cohomology and the generic point corresponds to rational motivic cohomology. We have the fully faithful inclusion $\DTM(\overbar{\bbQ},\bbZ) \hookrightarrow \DM(\overbar{\bbQ},\bbZ)$ into the derived category of all motives over $\overbar{\bbQ}$. By \cref{thm:main-thmb}, the map
	\[
		\Spc(\DM(\overbar{\bbQ},\bbZ)^c) \to \Spc(\DTM(\overbar{\bbQ},\bbZ)^c)
	\]
	is a strong topological quotient map with connected fibers. The problem of determining the spectrum of the category of all motives over $\overbar{\bbQ}$ boils down to understanding the connected fibers of this map.
`\end{Exa}

\begin{Exa}
	The spectrum of the category $\DATM(\bbR;\bbZ)$ of Artin--Tate motives over the real numbers (and other real-closed fields) has been studied in \cite{BalmerGallauer22a}. We may observe that the two maps
	\[\begin{tikzcd}[row sep=small,column sep=small]
		\Spc(\DATM(\bbR;\bbZ/2)^c)\ar[d]\ar[r] &\Spc(\DTM(\bbR;\bbZ/2)^c)\\
		\Spc(\DAM(\bbR;\bbZ/2)^c) & 
	\end{tikzcd}\]
	described in \cite[Remark~12.5]{BalmerGallauer22b} have connected fibers as \cref{thm:main-thmb} predicts. They look as follows:
	
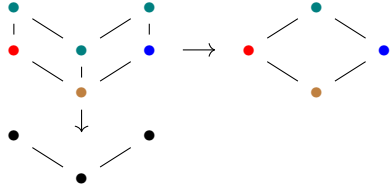
\[\begin{tikzcd}[row sep=tiny,column sep=small]
		{\textcolor{teal}\bullet} && {\textcolor{teal}\bullet}&&&{\textcolor{teal}\bullet}\\
		{\textcolor{red}\bullet}\ar[u,-] &{\textcolor{teal}\bullet}\ar[ul,-]\ar[ur,-]&{\textcolor{blue}\bullet} \ar[u,-]\ar[rr,shorten=2mm]&&{\textcolor{red}\bullet}\ar[ur,-]&&{\textcolor{blue}\bullet}\ar[ul,-] \\
										 &{\textcolor{brown}\bullet}\ar[d]\ar[ul,-] \ar[u,-] \ar[ur,-]  &&&& {\textcolor{brown}\bullet}\ar[ul,-]\ar[ur,-] \\
		\bullet & \hphantom{\bullet}& \bullet \\
				& \bullet \ar[ur,-]\ar[ul,-] &
	\end{tikzcd}\]
\end{Exa}

\begin{Exa}\label{exa:vishik}
	We can also consider the map induced by the fully faithful inclusion $\DATM(\bbR;\bbZ/2) \hookrightarrow \DM(\bbR;\bbZ/2)$ into the category of all motives over $\bbR$:
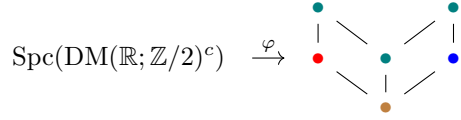
	\[\begin{tikzcd}[row sep=tiny,column sep=small]
		&& {\textcolor{teal}\bullet} && {\textcolor{teal}\bullet}\\
		\Spc(\DM(\bbR;\bbZ/2)^c)
		\ar[rr,shorten=2mm,"\varphi"]
		&& {\textcolor{red}\bullet}\ar[u,-] &{\textcolor{teal}\bullet}\ar[ul,-]\ar[ur,-]&{\textcolor{blue}\bullet} \ar[u,-]& \\
		&&&{\textcolor{brown}\bullet}\ar[ul,-] \ar[u,-] \ar[ur,-]  &
	\end{tikzcd}\]
	In a recent breakthrough, Vishik \cite{Vishik24} has increased our knowledge of the mysterious left-hand space and shown that it is considerably more complicated than the right-hand side. This result builds on \cite{Vishik19,Vishik22,Vishik23} and ultimately depends on proving that --- over so-called ``flexible fields'' --- isotropic Chow groups coincide with Chow groups for numerical equivalence. This leads to the construction of a family of ``isotropic'' primes which provide $2^{2^{\aleph_0}}$ distinct points in $\Spc(\DM(\bbR;\bbZ/2)^c)$. These isotropic points all lie over the top two green points in the figure above. Since~$\varphi$ is surjective, we know that there are more primes that have yet to be discovered. Moreover, by \cref{thm:main-thmb}, the above map is a strong topological quotient map whose fibers are connected. In fact, Vishik proved that there are no inclusions among the isotropic primes. There is tension between this fact and the fact that the fibers are connected, and this tension provides nontrivial information about the space. For example, since the fiber of over the top-left green point is connected and yet the specialization order among the known points of the fiber is trivial, either (a) there exist additional points in the fiber which make it connected, or (b) the constructible topology plays a nontrivial role in the fiber's spectral topology.
\end{Exa}

\begin{Exa}
	The unitation of the motivic stable homotopy category $\SH(\bbC)$ is equivalent to the classical stable homotopy category $\SH$. This follows from a theorem of Levine \cite[Theorem~1]{Levine14} which establishes that the canonical tt-functor $\SH \to \SH(\bbC)$ is fully faithful. Hence it induces an equivalence $\SH \to \unitation{\SH(\bbC)}$ by \cref{exa:unigenic-fully-faithful}. Therefore, the two maps
	\[
		\Spc(\SH(\bbC)^c) \to \Spc(\SHcell(\bbC)^c) \to \Spc(\SH^c)
	\]
	have connected fibers.
\end{Exa}

\begin{Rem}
	Further motivic consequences of our results will be given in \cref{exa:motivic-comp,exa:motivic-comp2,exa:noncommutative-comp} below.
\end{Rem}

\section{Applications to the graded comparison map}\label{sec:graded}

A fruitful approach to studying the Balmer spectrum of a given category $\cat K$ is via the (graded or ungraded) comparison map
\[
	\rho_{\cat K}: \SpcK \to \SpecRK
\]
constructed in \cite{Balmer10b} from the Balmer spectrum to the Zariski spectrum of the (graded or ungraded) endomorphism ring of the unit. This map factors through the unitation:
\begin{equation}\label{eq:factorization}
	\begin{aligned}
	\begin{tikzcd}[column sep=large]
		\Spc(\cat K) \ar[rr,bend right=15,"\rho_{\cat K}"'] \ar[r,"\varphi"] & \Spc(\cat K_{\langle \unit \rangle}) \ar[r,"\rho_{{\cat K_{\hspace{-0.3ex}\langle \unit \rangle}}}"] & \SpecRK. 
	\end{tikzcd}
	\end{aligned}
\end{equation}
In a number of interesting examples, the (graded or ungraded) comparison map is a homeomorphism. For such categories, we immediately obtain:

\begin{Cor}\label{cor:injective-comp}
	Let $\cT$ be a rigidly-compactly generated tt-category. If the graded or ungraded comparison map $\rho:\Spc(\cTc)\to \Spc(R_{\cT})$ is injective, then unitation $\unitation{\cT} \hookrightarrow \cT$ induces a homeomorphism $\Spc(\cTc) \xrightarrow{\simeq} \Spc(\unitation{\cT}^c)$.
\end{Cor}

\begin{proof}
	If either comparison map is injective, then it follows from the commutative diagram \eqref{eq:factorization} that the map $\varphi$ induced by unitation is injective and hence a homeomorphism by \cref{thm:faithful-are-quotient} and \cref{cor:spectral-homeo}.
\end{proof}

\begin{Rem}
	We will see some explicit examples of this phenomenon in~\cref{sec:equivariant}. For the remainder of this section, we will go the other way around and use information about the unitation to gain information about the comparison map. More precisely, we are able to strengthen the surjectivity results of Balmer \cite[Section~7]{Balmer10b} and of Lau \cite[Section~2]{Lau23}.
\end{Rem}

\begin{Rem}\label{rem:algebraic}
	Recall that for each prime ideal $\mathfrak p \in \Spec(R)$ of the graded or ungraded endomorphism ring, we have the associated algebraic localization $(-)_{\mathfrak p}:\cT \to \cT_{\mathfrak p}$. It is the finite localization associated to the Thomason subset $\rho^{-1}(\gen(\frakp)^\cc) = \bigcup_{s \not\in \mathfrak p} \supp(\cone(s))$; see \cite[Example~1.34]{BarthelHeardSanders23a} and \cite[Section~5]{Balmer10b}. It has the effect of localizing the (graded) $R$-module of (graded) morphisms
	\[
		\cT_{\frakp}(x_{\frakp},\Sigma^{(\bullet)}y_{\frakp}) \cong \cT(x,\Sigma^{(\bullet)}y)_{\frakp}
	\]
	for any compact $x \in \cTc$ and arbitrary $y \in \cT$; see \cite[Theorem~3.3.7]{HoveyPalmieriStrickland97} or \cite[Proposition~6.5]{Zou23bpp}.
\end{Rem}

\begin{Rem}\label{rem:algebraic-commute}
	Consider the map $\varphi:\SpcT\to\Spcunit{\cat T}$ given by the unitation and let $\frakp\in\Spec(R)$. It follows from the factorization~\eqref{eq:factorization} that 
	\[
		\varphi^{-1}(\rho_{\unitation{\cT}}^{-1}(\gen(\frakp)^\cc)) = \rho_{\cT}^{-1}(\gen(\frakp)^\cc).
	\]
	It then follows from \cref{rem:unitation-localization} that algebraic localization and unitation commute:
	\[\begin{tikzcd}[column sep=small]
		\unitation{\cT}\ar[d] \ar[rr,hook] && \cT\hphantom{.} \ar[d,shift right=0.3ex]\\
		(\unitation{\cT})_{\frakp} \ar[r,"\sim"] &\unitation{(\cT_{\frak p})} \ar[r,hook] & \cT_{\frakp}.
	\end{tikzcd}\]
\end{Rem}

\begin{Prop}\label{prop:monogeniccoherent}
	Let $\cat K$ be an essentially small rigid tt-category. Suppose that $\cat K=\thick\langle \unit\rangle$ and that the graded ring $\RK\coloneqq \End_{\cat K}^\bullet(\unit)$ is coherent. Then the graded comparison map $\rho_{\cat K} : \Spc(\cat K) \to \Spec(\RK)$ is a closed quotient map. Moreover, 
	\begin{equation}\label{eq:VX}
		\rho_{\cat K}(\supp(x)) = \SET{\mathfrak p \in \Spec(\RK)}{x_{\mathfrak p} \neq 0 \text{ in } \cat K_{\mathfrak p}}
	\end{equation}
	for each $x \in \cat K$.
\end{Prop}

\begin{proof}
	First we explain how the equality \eqref{eq:VX} will imply that $\rho_{\cat K}$ is a closed quotient map. The right-hand side of \eqref{eq:VX} is denoted $V(x)$ in \cite[Section~2]{Lau23} and we will use this notation in the proof. Note that $V(x)$ is always specialization closed, since algebraic localizations nest. Since $\supp(x)$ is proconstructible and $\rho_{\cat K}$ is a spectral map, the image $\rho_{\cat K}(\supp(x))$ is proconstructible (\cite[Corollary~1.3.23]{DickmannSchwartzTressl19}). Hence, if \eqref{eq:VX} holds then $\rho_{\cat K}(\supp(x))$ is specialization closed and hence closed (\cite[Theorem~1.5.4]{DickmannSchwartzTressl19}). In summary, if the equality \eqref{eq:VX} holds for all $x \in \cat K$, then $\rho_{\cat K}$ sends Thomason closed subsets to closed subsets. This implies $\rho_{\cat K}$ is a closed map by \cite[Theorem~5.3.3]{DickmannSchwartzTressl19}. Finally, $V(\unit)=\Spec(\RK)$, so the $x=\unit$ case of \eqref{eq:VX} implies that $\rho_{\cat K}$ is surjective. Thus, $\rho_{\cat K}$ is a closed surjective map, and hence is a closed quotient map by \cite[Corollary~6.4.14]{DickmannSchwartzTressl19}.

	With the above in hand, it remains to establish the equality \eqref{eq:VX}. The inclusion $\rho_{\cat K}(\supp(x)) \subseteq V(x)$ always holds, without any hypotheses on $\cat K$, by \cite[Proposition~2.7]{Lau23}. We claim that under our hypotheses, this is an equality. To this end, consider $\mathfrak p \in V(x)$. We will show that there exists $\cat P \in \supp(x)$ such that $\rho_{\cat K}(\cat P)=\mathfrak p$. As explained in the proof of \cite[Proposition~2.7]{Lau23}, we can use algebraic localization to reduce to the case where $\RK$ is local and $\mathfrak p=\mathfrak m$ is the unique homogeneous maximal ideal. Moreover, by that argument, it suffices to prove that 
	\[
		x^{\otimes n} \otimes \cone(f_1) \otimes \cdots \otimes \cone(f_r) \neq 0
	\]
	for any $n \ge 0$ and homogeneous elements $f_1,\ldots,f_r \in \mathfrak m$. Since $\cat K$ is rigid, it is enough to prove this for $n=0$ and $n=1$. Thus, it suffices to prove that for any $a \in \cat K$ and homogeneous $f \in \mathfrak m$, if $a \neq 0$ then $a \otimes \cone(f) \neq 0$. By \cite[Proposition~7.5]{Balmer10b}, if $\Hom_{\cat K}^\bullet(\unit,a)$ is nonzero and coherent as a graded $\RK$-module then $\Hom_{\cat K}^\bullet(\unit,a\otimes\cone(f))$ is nonzero, and hence $a\otimes \cone(f)\neq 0$. Since $\cat K$ is unigenic by hypothesis, $a \neq 0$ implies $\Hom_{\cat K}^\bullet(\unit,a) \neq 0$. In general, for a tt-category~$\cat K$, the collection of objects $a \in \cat K$ such that the graded $\RK$-module $\Hom_{\cat K}^\bullet(\unit,a)$ is coherent forms a thick subcategory. Thus, if the graded ring $\RK$ is coherent, then $\Hom_{\cat K}^\bullet(\unit,a)$ is a coherent $\RK$-module for every  $a \in \thick\langle \unit\rangle = \cat K$. This observation completes the proof.
\end{proof}

\begin{Thm}\label{thm:coherent-spectral}
	Let $\cT$ be a rigidly-compactly generated tt-category. Suppose $\RT\coloneqq \End_{\cT}^\bullet(\unit)$ is coherent (e.g., noetherian). Then the graded comparison map
	\begin{equation}\label{eq:graded-comparison}
		\rho : \Spc(\cTc) \to \Spec(\RT)
	\end{equation}
	is a strong spectral quotient map. In particular, $\rho$ is a homeomorphism if and only if it is injective.
\end{Thm}

\begin{proof}
	Consider $\unitation{\cT}^c= \thick\langle \unit \rangle \subseteq \cTc$. We have the factorization
	\begin{equation}\label{eq:graded-factorization}\begin{tikzcd}
		\Spc(\cTc) \ar[rr,bend left=20,"\rho_{\cT}"] \ar[r,"\varphi"'] & \Spc(\unitation{\cT}^c) \ar[r,"\rho_{\cT_{\hspace{-0.35ex}\langle\unit\rangle}}"'] & \Spec(\RT).
	\end{tikzcd}\end{equation}
	By \cref{prop:monogeniccoherent}, $\rho_{\cT_{\hspace{-0.3ex}\langle\unit\rangle}}$ is a closed quotient map and, by \cref{thm:faithful-are-quotient}, $\varphi$ is a strong spectral quotient map. Since every closed quotient map is a strong spectral quotient map (\cref{exa:closed-quotient-is-strong}), it follows that the composite $\rho_{\cT}$ is also a strong spectral quotient. Finally, an injective spectral quotient map is a homeomorphism by \cref{cor:spectral-homeo}.
\end{proof}

\begin{Rem}
	Except for the fact that we work with a big tt-category $\cT$, \cref{thm:coherent-spectral} enhances the surjectivity theorem of Balmer~ \cite[Theorem~7.3]{Balmer10b}. It also generalizes results in \cite[Section~2]{Lau23} which consider the case of an End-finite (\aka noetherian) category. By passing through the unitation, we are able to show that one doesn't need to assume the category is End-finite, and that it is enough just to make a hypothesis on the endomorphism ring.
\end{Rem}

\begin{Rem}\label{rem:graded-can-be-disconnected}
	The map \eqref{eq:graded-comparison} in \cref{thm:coherent-spectral} need not have connected fibers (even when both spaces are finite). A counter-example is provided in \cref{exa:graded-noetherian-not-connected} below. Thus, the theorem should be contrasted with \cref{thm:main-thmb}.
\end{Rem}

\section{Connective categories and weight structures}\label{sec:weight}

Our next goal is to prove a result about the ungraded comparison map of a connective tt-category in the spirit of \cref{thm:coherent-spectral} but whose conclusions are considerably stronger; see \cref{cor:ungraded-comparison} below. For this, we will have need of Bondarko's theory of weight structures.

\begin{Def}\label{def:connective}
	We say that a tt-category $\cat K$ is \emph{connective} if $\pi_i(\unit) \coloneqq \End_{\cat K}(\unit,\Sigma^{-i}\unit)$ is zero for all $i<0$.
\end{Def}

\begin{Exa}
	The homotopy category of spectra $\SH=\Ho(\Sp)$ is connective.
\end{Exa}

\begin{Not}
	For a collection $\cat E$ of objects in a triangulated category, we write $\add(\cat E)$ for the full additive subcategory consisting of finite direct sums of objects in~$\cat E$ and write $\smd(\cat E)$ for the full subcategory of direct summands of objects in~$\add(\cat E)$.
\end{Not}

\begin{Thm}[Bondarko]
	Let $\cT$ be a triangulated category. If $\cat G \subset \cT$ is a collection of objects which is connective in the sense that ${\Hom_{\cT}(a,\Sigma^n b) = 0}$ for all $a,b\in \cat G$ and $n >0$, then there is a unique (bounded) weight structure on $\cat K \coloneqq \thick\langle\cat G\rangle$ whose heart contains~$\cat G$. Moreover, $\cat K^{\heartsuit} = \smd(\cat G)$.
\end{Thm}

\begin{proof}
	This follows from \cite[Corollary~2.1.2]{BondarkoSosnilo18}; cf.~\cite[Theorem~4.3.2]{Bondarko10}.
\end{proof}

In the presence of an underlying model, a weight structure provides a well-behaved weight complex functor:

\begin{Thm}[Bondarko--Sosnilo--Aoki]\label{thm:tensor-weight}
	Let $\cat K$ be a tensor-triangulated category equipped with a bounded weight structure such that $\cat K_{w \ge 0}$ and $\cat K_{w \le 0}$ are both closed under the $\otimes$-product. If $\cat K$ is the homotopy category of an underlying symmetric monoidal stable $\infty$-category then there is a conservative tensor-triangulated functor
	\begin{equation}\label{eq:weight}
		W_{\cat K} : \cat K \to K^b(\cat K^{\heartsuit})
	\end{equation}
	which is the identity on the heart. It is the unique weight-exact tensor-triangulated functor (up to isomorphism) which arises from an exact symmetric monoidal functor of underlying $\infty$-categories whose restriction to the heart is equivalent to the identity.
\end{Thm}

\begin{proof}
	This is mostly established by \cite[Theorem~3.3.1]{Bondarko10}, \cite[Corollary~3.5]{Sosnilo19} and \cite[Corollary~4.5]{Aoki20}. We briefly recall the construction for the convenience of the reader and to explain the uniqueness statement. Let $\cat C$ and $\cat D$ be stable $\infty$-categories equipped with bounded weight structures and let $g:\cat C^{\heartsuit}\to \cat D^{\heartsuit}$ be an additive functor of additive $\infty$-categories. Sosnilo's uniqueness theorem \cite[Corollary~3.5]{Sosnilo19} establishes that there is an essentially unique weight-exact functor $G:\cat C \to \cat D$ which extends $g$. Now suppose that $\cat C$ and~$\cat D$ have symmetric monoidal structures $\cat C^{\otimes}$ and $\cat D^{\otimes}$ which are compatible with the weight structure. If~$g$ has a symmetric monoidal refinement $g^{\otimes}:(\cat C^{\heartsuit})^{\otimes} \to (\cat D^{\heartsuit})^{\otimes}$ then Aoki \cite[Corollary~4.5]{Aoki20} constructs a symmetric monoidal refinement $G^{\otimes}:\cat C^{\otimes} \to \cat D^{\otimes}$ of the corresponding weight-exact functor which restricts to $g^{\otimes}$. Working carefully through the construction, one sees that $G^{\otimes}$ is in fact the essentially unique refinement of $G$ which restricts to~$g^{\otimes}$.

	Now let $\cat K\coloneqq\Ho(\cat C)$. For the additive category $\cat A\coloneqq \cat K^{\heartsuit}=\Ho(\cat C^{\heartsuit})$, we have the dg-nerve $\cat D\coloneqq \dgNerve(\Ch^b(\cat A))$ with $\Ho(\cat D)=K^b(\cat A)$ and $\cat D^{\heartsuit} = \Nerve(\cat A)$. It follows from the above discussion that there is an essentially unique weight-exact symmetric monoidal functor 
	\begin{equation}\label{eq:weight-nerve}
		\cat C\to \dgNerve(\Ch^b(\cat K^{\heartsuit}))
	\end{equation}
	whose restriction to the hearts is the unit 
	\[
		\eta : \cat C^{\heartsuit} \to \Nerve(\Ho(\cat C^{\heartsuit}))=\Nerve(\cat K^\heartsuit)
	\]
	of the $\Ho \dashv \Nerve$ adjunction. By definition, the weight complex \eqref{eq:weight} is the functor that \eqref{eq:weight-nerve} induces on homotopy categories: ${\cat K=\Ho(\cat C)\to \Ho(\cat D)=K^b(\cat K^{\heartsuit})}$. It is conservative by \cite[Corollary~3.5]{Sosnilo19} and \cite[Theorem~3.3.1]{Bondarko10}. Finally, one can readily check that if $g:\cat C^{\heartsuit} \to \Nerve(\Ho(\cat C^{\heartsuit}))$ is a functor such that $\Ho(g)\simeq \Ho(\eta) =\id$ then $g \simeq \eta$.
\end{proof}

\begin{Exa}\label{exa:weight-R}
	Let $R$ be a commutative ring. The derived category of perfect complexes $\Der(R)^c$ has a bounded weight structure whose associated weight complex functor $W:\Der(R)^c \to K^b(\proj(R))\cong \Der(R)^c$ is equivalent to the identity.
\end{Exa}

\begin{Exa}\label{exa:SHG-weight}
	Let $G$ be a finite group and consider the equivariant stable homotopy category $\SHG$. The generators $\SET{G/H_+}{H \le G}$ form a connective collection of objects since $\pi_{n}^H(G/K_+) = 0$ for all $n <0$ and $H,K \le G$. Hence $\SHG^c$ has a unique weight structure whose heart is $\smd(G/H_+ \mid H \le G)$. This weight structure has been studied in \cite{Bondarko21,Bondarko24} and will be utilized in the proof of \cref{prop:DHZG-key-prop}.
\end{Exa}

\begin{Exa}\label{exa:weight-connective}
	Let $\cat K$ be an essentially small and idempotent-complete tt-category with $R_{\cat K} \coloneqq \End_{\cat K}(\unit)$. If $\cat K=\thick\langle\unit\rangle$ is unigenic and connective then $\cat K$ admits a bounded weight structure whose heart $\cat K^{\heartsuit}=\smd(\unit)$ is equivalent as an additive tensor category to the category $\proj(R_{\cat K})$ of finitely generated projective $R_{\cat K}$-modules. Moreover, since $\cat K_{w\ge 0}$ and $\cat K_{w\le 0}$ are the subcategories of $\cat K$ generated by $\unit$ using direct summands, extensions and either positive or negative shifts, it follows that $\cat K_{w \ge 0}$ and $\cat K_{w \le 0}$ are each closed under the $\otimes$-product, i.e., the weight structure is compatible with the $\otimes$-product in the sense of \cref{thm:tensor-weight}. Thus, if $\cat K$ is the homotopy category of an underlying symmetric monoidal stable $\infty$-category then the weight complex provides a conservative tensor triangulated functor 
	\begin{equation}\label{eq:weight-complex}
		W:\cat K \to K^b(\cat K^{\heartsuit})\cong \Der(R_{\cat K})^c
	\end{equation}
	which is the identity on the heart. Moreover, if we have two such categories $\cat K$ and~$\cat L$ then any tensor triangulated functor $F:\cat K\to \cat L$ is automatically weight-exact. Assuming this functor arises from a symmetric monoidal exact functor of underlying $\infty$-categories, we have a commutative diagram
	\begin{equation}\label{eq:weight-exact}
	\begin{tikzcd}
		\cat K\ar[d,"F"'] \ar[r,"W_{\cat K}"] & K^b(\cat K^{\heartsuit})\hphantom{.}\ar[d,"K^b(F^{\heartsuit})"] \\
		\cat L \ar[r,"W_{\cat L}"] & K^b(\cat L^{\heartsuit}).
	\end{tikzcd}
	\end{equation}
\end{Exa}

\begin{Rem}\label{rem:heart-base-change}
	The right-hand functor in \eqref{eq:weight-exact} is induced by the additive functor $F^{\heartsuit} \coloneqq F|_{\cat K^{\heartsuit}}:\cat K^{\heartsuit} \to \cat L^{\heartsuit}$. Under the identifications 
	\[
		\cat K^{\heartsuit} \cong \proj(\RK) \qquad \text{and} \qquad \cat L^{\heartsuit} \cong \proj(R_{\cat L})
	\]
	this is the functor given by base-change along the ring homomorphism ${\RK \to R_{\cat L}}$ induced by $F$. Applying \cref{thm:tensor-weight} to $K^b(\cat K^{\heartsuit})$, we conclude that the functor ${K^b(\cat K^{\heartsuit})\to K^b(\cat L^{\heartsuit})}$ coincides with the base-change functor ${\Der(\RK)^c \to \Der(R_{\cat L})^c}$.
\end{Rem}

\begin{Prop}\label{prop:R-local-K-local}
	Let $\cat K$ be an essentially small idempotent-complete tt-category which is the homotopy category of a symmetric monoidal stable $\infty$-category. Suppose that $\cat K$ is connective and unigenic. If $R_{\cat K}$ is a local ring then $\cat K$ is a local tt-category.
\end{Prop}

\begin{proof}
	Since $\cat K$ is connective and unigenic, it admits the bounded weight structure of \cref{exa:weight-connective}. The associated weight complex $W:\cat K\to\Der(\RK)^c$ is a conservative tt-functor to a local tt-category and it follows that $\cat K$ itself is local.
\end{proof}

\begin{Rem}
	There are plenty of examples of non-local tt-categories $\cat K$ whose endomorphism ring $R_{\cat K}$ is local. A connective example is the derived category of the projective line $\Der(\Pone)$ and a unigenic example is described in \cref{exa:3-point-scheme}. Thus \cref{prop:R-local-K-local} is false if either of these hypotheses on $\cat K$ is removed.
\end{Rem}

We end this section by providing an alternative perspective on the weight complexes of \cref{exa:weight-connective}.

\begin{Rem}
	Recall that the $\infty$-category of spectra $\Sp$ is endowed with a standard t-structure defined in terms of the vanishing of homotopy groups. The truncation \[\tau_{\le 0} : \Sp_{\ge 0} \to \Sp^{\heartsuit}\cong \NAb\] is a symmetric monoidal localization whose right adjoint sends an abelian group $A$ to its Eilenberg--MacLane spectrum $\rmH \hspace{-0.2em}A$. Thus, for each connective ring spectrum $\mathbb{E} \in \CAlg(\Sp_{\ge 0})$, we have a homomorphism $\mathbb{E} \to \tau_{\le 0} \mathbb{E} = \rmH(\pi_0(\mathbb{E}))$ in $\CAlg(\Sp)$ and hence a geometric functor
	\begin{equation}\label{eq:E-base-change}
		\Der(\mathbb{E})\to\Der(\rmH(\pi_0(\mathbb{E})))\cong\Der(\pi_0(\mathbb{E}))
	\end{equation}
	given by base-change.
\end{Rem}

\begin{Exa}\label{exa:connective-weight}
	If $\cT=\Ho(\cat C)$ is connective and unigenic then $\eend_{\cat C}(\unit) \in \Sp_{\ge 0}$ is a connective ring spectrum and we have a geometric functor
	\begin{equation*}\label{eq:connective-weight}
		\cT\cong \Der(\eend_{\cat C}(\unit))\to \Der(R_{\cT})
	\end{equation*}
	given by base-change along
	\begin{equation*}\label{eq:connective-base-change}
		\eend_{\cat C}(\unit) \to \tau_{\le 0}(\eend_{\cat C}(\unit))=\rmH(\pi_0(\eend_{\cat C}(\unit)))=\rmH(R_{\cT})
	\end{equation*}
	and \cref{rem:spectra}.
\end{Exa}

\begin{Prop}\label{prop:weight-truncation}
	If $\cT=\Ho(\cat C)$ is both connective and unigenic, then the weight complex functor $W:\cTc\to \Der(R_{\cT})^c$ of \cref{exa:weight-connective} coincides as a tt-functor with the restriction to compact objects of the geometric functor $\cT\cong\Der(\eend_{\cat C}(\unit))\to \Der(R_{\cT})$ of~\cref{exa:connective-weight}.	
\end{Prop}

\begin{proof}
	The functor $\cat C \cong \eend_{\cat C}(\unit)\text{-Mod}_{\Sp} \to \tau_{\le 0}(\eend_{\cat C}(\unit))\text{-Mod}_{\Sp}$ is a weight-exact functor whose restriction to the heart is the unit $\cat C^{\heartsuit} \to \Nerve(\Ho(\cat C^{\heartsuit}))$ of the $\Ho \dashv \Nerve$ adjunction. Its restriction to compact objects is thus equivalent to the weight complex functor by \cref{thm:tensor-weight}.
\end{proof}

\begin{Exa}
	From this perspective, the conservativity of the weight complex functors of \cref{exa:weight-connective} amounts to the statement that for any connective ring spectrum $\bbE \in \CAlg(\Sp_{\ge 0})$, the induced functor $\Der(\bbE) \to \Der(\pi_0(\bbE))$ is conservative on compact objects. For $\bbE=\Sphere$ this is a straightforward consequence of the Hurewicz theorem.
\end{Exa}

\begin{Exa}
	Let $\mathbb{E} \in \CAlg(\Sp_{\ge 0})$ be a connective ring spectrum. The restriction of \eqref{eq:E-base-change} to compact objects
	\[
		\Der(\mathbb{E})^c \to \Der(\pi_0(\mathbb{E}))^c\cong K^b(\proj(\pi_0(\mathbb{E})))
	\]
	can be identified with the weight complex functor
	\[
		\Der(\mathbb{E})^c \to K^b(\Der(\mathbb{E})^{\heartsuit}) \cong K^b(\proj(\pi_0(\mathbb{E})))
	\]
	associated to the canonical bounded weight structure on $\Der(\mathbb{E})^c$.
\end{Exa}

\begin{Exa}
	We may apply the above to $\SH=\Ho(\Sp)$, i.e., to $\mathbb{E}=\Sphere$. The restriction of $\SH \to \Der(\HZ)$ to compact objects
	\[
		\SH^c \to \Der(\HZ)^c
	\]
	can be identified with the weight complex functor associated to the canonical bounded weight structure on $\SH^c$.
\end{Exa}

\section{Applications to the ungraded comparison map}\label{sec:ungraded}

We will now apply the above results to the ungraded comparison map of a connective category.

\begin{Ter}
	Recall that a topological space $X$ is \emph{local} if it is nonempty and any open cover is trivial: $X=\bigcup_{i \in I} U_i$ implies $X=U_i$ for some $i \in I$. Note that this is considerably stronger than being connected. A spectral space $X$ is local if and only if it has a unique closed point. (This follows from the fact that every closed subset of a spectral space contains a closed point, see \cite[Proposition~4.1.2]{DickmannSchwartzTressl19}). We'll say that a spectral subspace $Z \subseteq X$ has a \emph{unique relatively closed point} if~$Z$ has a unique closed point when regarded as a spectral space in its own right or, equivalently, if $Z$ is local in the induced topology.
\end{Ter}

We isolate the following purely topological observation:

\begin{Lem}\label{lem:local-algebraic}
	Let $g:X \to Y$ be a spectral map of spectral spaces and let $f:Y\to X$ be a function such that $g\circ f=\id_Y$. Suppose that
	\begin{enumerate}
		\item[(i)] $f$ preserves specializations, equivalently, the preimage $f^{-1}(Z)$ of a (specialization) closed set $Z \subseteq X$ is specialization closed; and
		\item[(ii)] $x \rightsquigarrow f(g(x))$ for each $x \in X$.
	\end{enumerate}
	Then the following hold:
	\begin{enumerate}
		\item $g$ is a closed quotient map.
		\item $f$ is a topological embedding.
		\item $g(Z) = f^{-1}(Z)$ for every specialization closed set $Z \subseteq X$.
	\end{enumerate}
\end{Lem}

\begin{proof}
	$(a)$: A surjective spectral map is a closed quotient map as soon as it is a closed map (\cite[Corollary~6.4.14]{DickmannSchwartzTressl19}). Moreover, to establish that $g$ is a closed map it suffices to establish that $g(\overbar{\{x\}})$ is closed for each $x \in X$ (\cite[Theorem~5.3.3]{DickmannSchwartzTressl19}). Finally, since $g(\overbar{\{x\}})$ is proconstructible, it suffices to prove that $g(\overbar{\{x\}})$ is closed under specialization. To this end, suppose $x \rightsquigarrow x'$ and $g(x') \rightsquigarrow y$. Then $f(g(x'))\rightsquigarrow f(y)$ by hypothesis $(i)$ and $x' \rightsquigarrow f(g(x'))$ by hypothesis $(ii)$. Thus $x \rightsquigarrow f(y)$ and hence $y=g(f(y)) \in g(\overbar{\{x\}})$ as desired. This proves part $(a)$.

	$(c)$: We always have $f^{-1}(Z) \subseteq g(Z)$ for any subset $Z \subseteq X$ simply from the fact that $g\circ f=\id_Y$. On the other hand, hypothesis $(ii)$ implies that $f^{-1}(Z) \subseteq g(Z)$ if~$Z$ is specialization closed.

	$(b)$: Note that $(a)$ and $(c)$ together imply that $f$ is continuous and it is standard that a split monomorphism in the category of topological spaces is an embedding.
\end{proof}

\begin{Rem}\label{rem:alt-local-hyp}
	The hypotheses of the previous lemma can be reformulated as follows: Let $g:X \to Y$ be a spectral map with the property that each fiber $g^{-1}(\{y\})$ is a local space (i.e.~contains a unique relatively closed point) and that, moreover, the map $y \mapsto f(y) \in g^{-1}(\{y\})$ which sends $y$ to the unique relatively closed point of the fiber $g^{-1}(\{y\})$ is specialization-preserving.
\end{Rem}

\begin{Prop}\label{prop:alg-local-is-local}
	Let $\cat K$ be an essentially small tt-category with the property that the algebraic localization $\cat K_{\frak p}$ is local for each $\frak p \in \Spec(R_{\cat K})$. Then
	\begin{enumerate}
		\item The comparison map $\rho_{\cat K} : \Spc(\cat K) \to \Spec(R_{\cat K})$ is a closed quotient map whose fibers are local.
		\item The map  $f:\Spec(\RK)\to\Spc(\cat K)$ given by $\frakp \mapsto \Ker(\cat K \to \cat K_{\frakp})$ is a well-defined continuous section to $\rho_{\cat K}$ and hence is a topological embedding.
		\item For each $a \in \cat K$, we have
			\[
				\rho_{\cat K}(\supp(a)) = f^{-1}(\supp(a))=\SET{\frak p \in \Spec(R)}{x_{\frak p} \neq 0 \text{ in } \cat K_{\frak p}}.
			\]
	\end{enumerate}
\end{Prop}

\begin{proof}
	Recall from \cite[Corollary~5.6]{Balmer10b} that an algebraic localization provides a cartesian square on Balmer spectra. From this, one sees that the hypothesis that $\cat K_{\mathfrak p}$ is local is equivalent to the assertion that the fiber $\rho_{\cat K}^{-1}(\{\mathfrak p\})$ has a unique minimum point for inclusion, namely the \emph{prime} tt-ideal $\Ker(\cat K\to \cat K_{\mathfrak p})$. The function sending~$\mathfrak p$ to this unique minimum point in $\rho_{\cat K}^{-1}(\{\mathfrak p\})$ is denoted $f$ in part~$(b)$. Algebraic localizations nest: If $\mathfrak p \rightsquigarrow \mathfrak q$ then the $\mathfrak p$-localization factors through the $\mathfrak q$-localization: $\cat K \to \cat K_{\mathfrak q} \to \cat K_{\mathfrak p}$, i.e., $f(\mathfrak p) \rightsquigarrow f(\mathfrak q)$. In other words, the function $f$ preserves specializations. Note also that it follows from the definition of $\rho_{\cat K}$ that the tt-localization $\cat K \to \cat K/\cat P$ factors through the algebraic localization: $\cat K \to \cat K_{\rho(\cat P)} \to \cat K/\cat P$. In other words, $\cat P \rightsquigarrow f(\rho(\cat P))$. Thus, we are in the situation of \cref{lem:local-algebraic} with $g\coloneqq\rho_{\cat K}$ and the conclusions follow.
\end{proof}

\begin{Thm}\label{thm:connective-unigenic}
	Let $\cat K$ be an essentially small idempotent-complete tt-category which is unigenic and connective and let $R_{\cat K} = \End_{\cat K}(\unit)$. Assume that $\cat K$ is the homotopy category of an underlying symmetric monoidal stable $\infty$-category. Let~$W_{\cat K}:\cat K \to \Der(\RK)^c$ denote the weight complex functor for the weight structure of \cref{exa:weight-connective} and let $\omega_{\cat K}$ denoted the induced spectral map
	\[
		\Spec(\RK) \cong \Spc(\Der(\RK)^c) \xrightarrow{\Spc(W_{\cat K})} \Spc(\cat K).
	\]
	Then the following statements hold:
	\begin{enumerate}
		\item Each fiber $\rho_{\cat K}^{-1}(\{\frakp\})$ is a local space whose unique relatively closed point is the prime tt-ideal $\Ker(\cat K \to \cat K_{\frakp}) \in \Spc(\cat K)$.
		\item The function $\omega_{\cat K} :\Spec(\RK) \to \Spc(\cat K)$ is given by $\frakp \mapsto \Ker(\cat K\to \cat K_{\frakp})$.
		\item For any $a \in \cat K$ we have,
			\[
				\rho_{\cat K}(\supp(a)) = \omega_{\cat K}^{-1}(\supp(a)) = \SET{\mathfrak p\in \Spec(R)}{x_{\mathfrak p} \neq 0 \text{ in } \cat K_{\mathfrak p}}.
			\]
		\item The comparison map $\rho_{\cat K} : \Spc(\cat K) \to \Spec(R_{\cat K})$ is a closed quotient map.
		\item The map $\omega_{\cat K}:\Spec(R_{\cat K})\to \Spc(\cat K)$ is a topological embedding whose image is the subspace of $\Spc(\cat K)$ consisting of the unique relatively closed points of the fibers of $\rho_{\cat K}$.
	\end{enumerate}
\end{Thm}

\begin{proof}
	For any $\frakp \in \Spec(\RK)$, the algebraic localization $\cat K_{\frakp}$ is again unigenic (\cref{exa:finite-loc-is-unigenic}) and connective (\cref{rem:algebraic}). Therefore \cref{prop:R-local-K-local} establishes that the tt-category $\cat K_{\frakp}$ is local. This implies $(a)$ and moreover that \cref{prop:alg-local-is-local} applies. Note that $(b)$ asserts that the map $\omega_{\cat K}$ coincides with the map denoted $f$ in \cref{prop:alg-local-is-local}. Thus, once we establish $(b)$ the rest follows from \cref{prop:alg-local-is-local}.

	Recall from \cref{exa:weight-connective} that the localization $\cat K \to \cat K_{\frakp}$ is automatically weight-exact and hence we have a commutative diagram
	\begin{equation}\label{eq:omega-rho}\begin{tikzcd}
		\Spec(\RK) \ar[r,"\omega_{\cat K}"]  & \Spc(\cat K) \ar[r,"\rho_{\cat K}"] & \Spec(\RK) \\
		\Spec((\RK)_{\frakp})\ar[u,hook] \ar[r,"\omega_{\cat K_{\frak p}}"] & \Spc(\cat K_{\frakp}) \ar[u,hook] \ar[r,"\rho_{\cat K_{\frakp}}"] & \Spec((\RK)_{\frakp}).\ar[u,hook]
	\end{tikzcd}\end{equation}
	The conservativity of the weight complex 
	implies that $\omega_{\cat K_{\frakp}}$ preserves the unique closed points. Moreover, the middle upwards arrow maps the unique closed point of $\Spc(\cat K_{\frakp})$ to $\Ker(\cat K \to \cat K_{\frakp}) \in \Spc(\cat K)$. The commutativity of the diagram then establishes $(b)$. This completes the proof. We note in passing that since $\rho_{\cat K}\circ \omega_{\cat K}$ is the identity map and the right-hand square in~\eqref{eq:omega-rho} is cartesian, it follows that the left-hand square is cartesian, too.
\end{proof}

\begin{Rem}\label{rem:linear-embedding}
	Thus, in the spectrum of a connective unigenic category $\cat K$ there is an embedded copy of $\Spec(R_{\cat K})$ which is provided by the weight complex functor for the canonical weight structure on $\cat K$. The stable homotopy category $\cat K=\SH^c$ provides an informative example of this phenomenon. Indeed, one of the motivations for \cite{PatchkoriaSandersWimmer22} was the observation that, on Balmer spectra, the canonical functor $\SH\to\Der(\HZ)\cong\Der(\bbZ)$ embeds a copy of $\Spec(\bbZ)$ into $\Spc(\SH^c)$ covering the top and bottom layers of points:

	
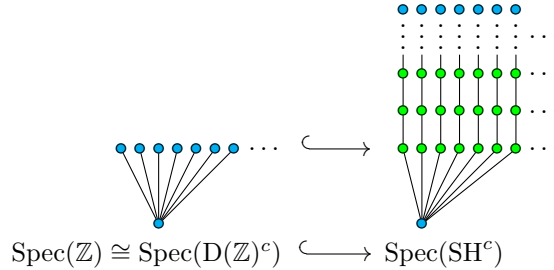
\vbox{
	\begin{equation*}
	\xy
	{\ar@{^{(}->} (0,-5)*{};(8,-5)*{}};
	{\ar@{-} (-20,-15)*{};(-25,-5)*{}};
	{\ar@{-} (-20,-15)*{};(-22.5,-5)*{}};
	{\ar@{-} (-20,-15)*{};(-20,-5)*{}};
	{\ar@{-} (-20,-15)*{};(-17.5,-5)*{}};
	{\ar@{-} (-20,-15)*{};(-15,-5)*{}};
	{\ar@{-} (-20,-15)*{};(-12.5,-5)*{}};
	{\ar@{-} (-20,-15)*{};(-10,-5)*{}};
	(-25,-5)*{\color{cyan}\bullet};
	(-25,-5)*{\circ};
	(-22.5,-5)*{\color{cyan}\bullet};
	(-22.5,-5)*{\circ};
	(-20,-5)*{\color{cyan}\bullet};
	(-20,-5)*{\circ};
	(-17.5,-5)*{\color{cyan}\bullet};
	(-17.5,-5)*{\circ};
	(-15,-5)*{\color{cyan}\bullet};
	(-15,-5)*{\circ};
	(-12.5,-5)*{\color{cyan}\bullet};
	(-12.5,-5)*{\circ};
	(-10,-5)*{\color{cyan}\bullet};
	(-10,-5)*{\circ};
	(-6,-5)*{\hdots};
	(-20,-15)*{\color{cyan}\bullet};
	(-20,-15)*{\circ};
	{\ar@{-} (15,-15)*{};(12.5,-5)*{}};
	{\ar@{-} (15,-15)*{};(15,-5)*{}};
	{\ar@{-} (15,-15)*{};(17.5,-5)*{}};
	{\ar@{-} (15,-15)*{};(20,-5)*{}};
	{\ar@{-} (15,-15)*{};(22.5,-5)*{}};
	{\ar@{-} (15,-15)*{};(25,-5)*{}};
	{\ar@{-} (15,-15)*{};(27.5,-5)*{}};
	{\ar@{-} (20,-5)*{};(20,7.5)*{}};
	{\ar@{-} (12.5,-5)*{};(12.5,7.5)*{}};
	{\ar@{-} (15,-5)*{};(15,7.5)*{}};
	{\ar@{-} (17.5,-5)*{};(17.5,7.5)*{}};
	  (12.5,-5)*{\color{green}\bullet};
	  (12.5,-5)*{\circ};
	  (12.5,0)*{\color{green}\bullet};
	  (12.5,0)*{\circ};
	  (12.5,5)*{\color{green}\bullet};
	  (12.5,5)*{\circ};
	  (12.5,11)*{\vdots};
	  (12.5,13.5)*{\color{cyan}\bullet};
	  (12.5,13.5)*{\circ};
	%
	  (15,-5)*{\color{green}\bullet};
	  (15,-5)*{\circ};
	  (15,0)*{\color{green}\bullet};
	  (15,0)*{\circ};
	  (15,5)*{\color{green}\bullet};
	  (15,5)*{\circ};
	  (15,11)*{\vdots};
	  (15,13.5)*{\color{cyan}\bullet};
	  (15,13.5)*{\circ};
	%
	  (17.5,-5)*{\color{green}\bullet};
	  (17.5,-5)*{\circ};
	  (17.5,0)*{\color{green}\bullet};
	  (17.5,0)*{\circ};
	  (17.5,5)*{\color{green}\bullet};
	  (17.5,5)*{\circ};
	  (17.5,11)*{\vdots};
	  (17.5,13.5)*{\color{cyan}\bullet};
	  (17.5,13.5)*{\circ};
	%
	 (20,-5)*{\color{green}\bullet};
	 (20,-5)*{\circ};
	 (20,0)*{\color{green}\bullet};
	 (20,0)*{\circ};
	  (20,5)*{\color{green}\bullet};
	  (20,5)*{\circ};
	  (20,11)*{\vdots};
	  (20,13.5)*{\color{cyan}\bullet};
	  (20,13.5)*{\circ};
	%
	{\ar@{-} (22.5,-5)*{};(22.5,7.5)*{}};
	 (22.5,-5)*{\color{green}\bullet};
	 (22.5,-5)*{\circ};
	 (22.5,0)*{\color{green}\bullet};
	 (22.5,0)*{\circ};
	  (22.5,5)*{\color{green}\bullet};
	  (22.5,5)*{\circ};
	  (22.5,11)*{\vdots};
	  (22.5,13.5)*{\color{cyan}\bullet};
	  (22.5,13.5)*{\circ};
	%
	{\ar@{-} (25,-5)*{};(25,7.5)*{}};
		 (25,-5)*{\color{green}\bullet};
		 (25,-5)*{\circ};
		 (25,0)*{\color{green}\bullet};
		 (25,0)*{\circ};
	  (25,5)*{\color{green}\bullet};
	  (25,5)*{\circ};
	  (25,11)*{\vdots};
	  (25,13.5)*{\color{cyan}\bullet};
	  (25,13.5)*{\circ};
	%
	{\ar@{-} (27.5,-5)*{};(27.5,7.5)*{}};
	(27.5,-5)*{\color{green}\bullet};
	(27.5,-5)*{\circ};
	(27.5,0)*{\color{green}\bullet};
	(27.5,0)*{\circ};
	  (27.5,5)*{\color{green}\bullet};
	  (27.5,5)*{\circ};
	  (27.5,11)*{\vdots};
	  (27.5,13.5)*{\color{cyan}\bullet};
	  (27.5,13.5)*{\circ};
	%
	  (31.25,-5)*{\hdots};
	  (31.25,0)*{\hdots};
	  (31.25,5)*{\hdots};
	  (31.25,10)*{\hdots};
	%
	%
	%
	  (15,-15)*{\color{cyan}\bullet};
	  (15,-15)*{\circ};
	%
	\endxy
	\end{equation*}
	\[\hspace{-6em}\xymatrix@1{ \Spec(\bbZ) \cong \Spec(\Der(\bbZ)^c)\;\;\; \ar@{^{(}->}[r] & \,\;\Spec(\SH^c)}\]
}

	\smallskip
	\noindent Observe that it embeds $\Spec(\bbZ)$ precisely onto the subspace consisting of the unique relatively closed points of the fibers of the comparison map. \Cref{thm:connective-unigenic} provides a general explanation for this phenomenon.
\end{Rem}

\begin{Rem}
	Observe that the statement of \cref{thm:connective-unigenic} is considerably stronger than the result from \cref{prop:monogeniccoherent} for the graded comparison map of a unigenic category whose graded endomorphism ring is coherent. It is natural to ask whether the latter has an analogous strengthening. This is false. Consider the derived category of the projective line. The graded endomorphism ring is coherent and concentrated in degree zero where it is a single point. But the fiber, which is $\Pone$, does not have a unique closed point.
\end{Rem}

\begin{Cor}\label{cor:ungraded-comparison}
	Let $\cT$ be a connective rigidly-compactly generated tt-category which is the homotopy category of an underlying symmetric monoidal stable $\infty$-category. The ungraded comparison map
	\begin{equation}\label{eq:rho-ungraded-comparison}
		\rho : \SpcT \to \Spec(\End_{\cT}(\unit))
	\end{equation}
	is a strong spectral quotient map whose fibers are connected. In particular, $\rho$ is a homeomorphism if and only if it is injective.
\end{Cor}

\begin{proof}
	Let $\RT \coloneqq \End_{\cT}(\unit)$. As in the proof of \cref{thm:coherent-spectral}, we have a factorization of the ungraded comparison map as
	\begin{equation}\label{eq:ungraded-factorization}\begin{tikzcd}
		\Spc(\cTc) \ar[rr,bend left=20,"\rho"] \ar[r,"\varphi"'] & \Spc(\unitation{\cT}^c) \ar[r,"\rho"'] & \Spec(\RT).
	\end{tikzcd}\end{equation}
	The first map is a strong spectral quotient map (\cref{thm:faithful-are-quotient}) while the second is a closed quotient map (\cref{thm:connective-unigenic}). Since closed quotient maps are strong spectral quotient maps (\cref{exa:closed-quotient-is-strong}), it follows that the composite is a strong spectral quotient map. The last statement then follows from \cref{cor:spectral-homeo}.

	We now establish that the fibers are connected. Recall from \cref{rem:algebraic-commute} that algebraic localization ``commutes'' with unitation. It follows that we have a commutative diagram 
	\[\begin{tikzcd}
		\Spc(\cTc)  \ar[r,"\varphi"] & \Spc(\unitation{\cT}^c) \ar[r,"\rho"] & \Spec(\RT) \\
		\Spc(\cTc_{\frakp}) \ar[u,hook] \ar[r,"\varphi"] & \Spc(\unitation{(\cT_{\frakp})}^c) \ar[u,hook]  \ar[r,"\rho"]& \Spec((\RT)_{\frakp}).  \ar[u,hook]
	\end{tikzcd}\]
	Since the outer square is cartesian \cite[Corollary~5.6]{Balmer10b}, we are reduced to showing that $\varphi^{-1}(\rho^{-1}(\{\frakm\}))$ is connected where $\frakm \in \Spec((\RT)_{\frakp})$ is the unique closed point. Note that $\unitation{(\cT_{\frakp})}\cong (\unitation{\cT})_{\frakp}$ is an algebraic localization of a connective category and hence is still connective (\cref{rem:algebraic}). Thus, \cref{prop:R-local-K-local} establishes that~$\unitation{(\cT_{\frakp})}$ is a local tt-category. Since the bottom $\varphi$ is induced by the fully faithful functor $\smash{\unitation{(\cT_{\frakp})}} \hookrightarrow \cT_{\frakp}$, the fiber $\varphi^{-1}(\rho^{-1}(\{\frakm\}))$ is connected by \cref{prop:connected-preimage}.
\end{proof}

\begin{Rem}\label{rem:connective-noetherian}
	If $\Spec(\End_{\cT}(\unit))$ is noetherian then the strong spectral quotient of \cref{cor:ungraded-comparison} is a strong topological quotient map by \cref{prop:strong-spec-topo}.
\end{Rem}

\begin{Exa}\label{exa:motivic-comp}
	The derived category of motives $\DM(k;R)$ is connective, since the motivic cohomology groups $\End(\unit,\Sigma^n \unit) = H^{n,0}(k;R)$ vanish for $n>0$; see \cite[Theorem~3.6]{MazzaVoevodskyWeibel06}. Therefore, by \cref{cor:ungraded-comparison}, the comparison map ${\rho:\Spc(\DM(k;R)^c)\to \Spec(R)}$ is a strong spectral quotient map whose fibers are connected. Moreover, this statement is also true for the comparison map of each of the concentrations $\DTM(k;R)$, $\DAM(k;R)$ and $\DATM(k;R)$; see \mbox{\cref{exa:DM,exa:artin-motives,exa:DATM}}.
\end{Exa}

\begin{Exa}\label{exa:motivic-comp2}
	Let $k$ be a field of characteristic zero. The motivic stable homotopy category $\cat T=\SH(k)$ is connective (see \cite{Morel05}) and its ungraded endomorphism ring is isomorphic to the Grothendieck--Witt ring $\GW(k)$ by \cite[Theorem~6.2.1]{Morel04a}. Moreover, the Zariski spectrum of $\GW(k)$ is known; see \cite[Remark~10.2]{Balmer10b}. By \cref{cor:ungraded-comparison}, the comparison map $\rho:\Spc(\SH(k)^c) \to \Spec(\GW(k))$ is a strong spectral quotient map whose fibers are connected.
\end{Exa}

\begin{Exa}\label{exa:noncommutative-comp}
	Let $k$ be a commutative ring and let $\cT$ denote the concentration of~$\Mot_k^a$ at the dualizable noncommutative motives; see \cref{exa:noncommutative-motives}. The graded endomorphism ring of the unit is the algebraic $K$-theory of the ring $k$ by \cite[Theorem~15.10]{Tabuada08}. More precisely
	\[
		\pi_n(\unit) = \Hom_{\cT}(\Sigma^n\unit,\unit)=
		\begin{cases}
		K_n(k) & \text{for } n \ge 0\\
		0 & \text{for } n<0.
		\end{cases}
	\]
	In particular, the tt-category $\cT$ is connective. By \cref{cor:ungraded-comparison}, the comparison map $\rho:\Spc(\cT^c)\to \Spec(K_0(k))$ is a strong spectral quotient map with connected fibers. The unitation of $\cT$ has been studied in \cite{DellAmbrogioTabuada12}. They prove that if $k$ is a finite field or the closure of a finite field, then $\rho:\Spcunit{\cT}\to\Spec(K_0(k))=\Spec(\bbZ)$ is a bijection, and hence a homeomorphism by \cref{cor:spectral-homeo}. Moreover, for a field $k$ of characteristic zero, they also consider the concentration $\concentration{\cT}{\cat G}$ generated by $\unit$ and the noncommutative motive of $k[t]$; see \cite[Section~5]{DellAmbrogioTabuada12}.
\end{Exa}

\begin{Exa}
	Let $\cT=\SHG$ or $\cT=\DHZG$ for any finite group $G$. By \cref{cor:ungraded-comparison}, the comparison map $\rho:\Spc(\cTc)\to \Spec(A(G))$ is a topological quotient map whose fibers are connected. The fact that it is a quotient map (for $\cT=\DHZG$) was observed by direct methods in \cite[Corollary~2.42]{PatchkoriaSandersWimmer22}.
\end{Exa}

\begin{Rem}
	The stronger consequences of \cref{thm:connective-unigenic}, namely that the fibers are local (rather than just connected), need not be true if $\cT$ is not unigenic; see \cref{rem:exceptional} below.
\end{Rem}

\begin{Rem}
	The fibers of the ungraded and graded comparison maps can fail to be connected if the category is not connective, even if it is unigenic. See \cref{exa:graded-noetherian-not-connected}.
\end{Rem}

\begin{Rem}
	It remains an open question in general, whether the comparison map of a tt-category is a homeomorphism if and only if it is a bijection. This is known to be true under some hypotheses; see \cite[Corollary~2.5]{Lau23} and \cite{DellAmbrogioStanley16erratum}. Our \cref{thm:coherent-spectral} and \cref{cor:ungraded-comparison} give further sufficient conditions for the question to have an affirmative answer.
\end{Rem}

\section{Unitation in equivariant higher algebra}\label{sec:equivariant}

We now turn to explicit examples arising in equivariant higher algebra starting with the modular representation theory of finite groups.

\begin{Prop}\label{prop:p-group}
	Let $G$ be a finite group and let $k$ be a field of characteristic $p>0$. The derived category of representations $\DRep(G,k)\cong K(\Inj(kG))$ is unigenic if and only if $G$ is a $p$-group.
\end{Prop}

\begin{proof}
	Recall that $\DRep(G,k)^c \cong \Dbmod{kG}$. Note that finitely generated $kG$-modules are of finite length. Hence $\Dbmod{kG}$ is generated by the collection of simple $kG$-modules. More precisely, if $S$ is any set of simple $kG$-modules then the full subcategory $\cat A_S \subseteq \mathrm{mod}(kG)$ consisting of those modules whose composition factors lie in $S$ forms an abelian subcategory which is closed under extensions. Hence $\DbSmod{kG}\coloneqq \SET{ t \in \Dbmod{kG}}{H_i(t) \in \cat A_S \text{ for all } i \in \mathbb{Z}}$ is a thick subcategory. Thus, if $\Dbmod{kG}$ is unigenic then, taking $S=\{k\}$, we have that $\Dbmod{kG}=\thick\langle k\rangle \subseteq \Dbkmod{kG}$ which implies that $k$ is the only simple module. In summary, $\Dbmod{kG}$ is unigenic if and only if $k$ is the only simple module. Now recall that the Jacobson radical $J(G) \coloneqq \mathrm{rad}(kG)$ is the largest nilpotent ideal \cite[Proposition~3.2]{Weintraub03} and $J(G) \subseteq I(G) \coloneqq \Ker(kG \to k)$. Moreover, $J(G)$ is the intersection of the annihilators of all the simple modules \cite[Proposition~2.2]{Weintraub03}. Hence, the trivial representation $k$ is the unique simple module if and only if $J(G)=I(G)$ if and only if the augmentation ideal $I(G)$ is nilpotent if and only if $G$ is a finite $p$-group \cite[Chapter 3, Lemma 1.6, page~70]{Passman77}.
\end{proof}

\begin{Exa}\label{exa:modular-unitation}
	Let $G$ be a finite group and let $k$ be a field of positive characteristic. As explained in \cref{exa:DRep-monadic}, the unitation of $\DRep(G;k)\cong K(\Inj(kG))$ is equivalent as a tt-category to the derived category of the ring spectrum $F(BG_+,\Hk)\in \CAlg(\Sp)$. Recall that this is equivalent to the derived category of the dg-algebra $C^*(BG;k)$ of $k$-valued cochains on the classifying space $BG$:
	\[
		\DRep(G;k)_{\langle \unit \rangle} \cong \Der(F(BG_+,\Hk))\cong \Derdg(C^*(BG;k)).
	\]
	See \cite{BarthelCastellanaHeardValenzuela22} and \cite{BensonKrause08}. On Balmer spectra, the unitation of $\DRep(G,k)$ induces a homeomorphism
	\[
		\Spc(\Dbmod{kG}) \xrightarrow{\cong} \Spc(\unitation{\Dbmod{kG}}).
	\]
	This is an instance of \cref{cor:injective-comp} since the graded comparison map is a homeomorphism; see \cite[Proposition~8.5]{Balmer10b}.
\end{Exa}

\begin{Rem}\label{rem:not-p-group-ff}
	Note that if $G$ is not a $p$-group then the unitation of $\DRep(G,\Fp)$ is an example of a fully faithful geometric functor, which is not an equivalence, but which induces a homeomorphism on Balmer spectra.
\end{Rem}

\begin{Exa}\label{exa:stmod-unitation}
	Let $G$ be a finite group and let $k$ be a field of positive characteristic. The unitation of $\StMod(kG)$ induces a homeomorphism on Balmer spectra
	\[
		\Spc(\stmod(kG)) \xrightarrow{\cong} \Spc(\unitation{\stmod(kG)}).
	\]
	Again this is an instance of \cref{cor:injective-comp} since the graded comparison map is injective; see the proof of \cite[Proposition~8.5]{Balmer10b}.
\end{Exa}

\begin{Rem}\label{rem:unitation-perm}
	Let $G$ be a finite group and let $k$ be a field of positive characteristic. We have already observed in \cref{exa:DPerm-inflation} that the unitation of $\DPerm(G;k)$ is equivalent to $\Der(k)$. From a different perspective, 
	\[
		\DPerm(G;k) \cong \Ho(\Hkbar\text{-Mod}_{\Sp_G})
	\]
	where $\Hkbar$ denotes the Eilenberg--MacLane $G$-spectrum associated to the constant \mbox{$G$-Mackey} functor $\underline{k}$. By \cref{rem:equiv-monadic}, we have
	\[
		\unitation{(\Hkbar\text{-Mod}_{\Sp_G})} \cong \lambda^G( \Hkbar)\text{-Mod}_{\Sp}
	\]
	and it is immediate from the definitions that $\lambda^G(\Hkbar) \cong \Hk$.
\end{Rem}

\begin{Prop}\label{prop:DPerm-unigenic}
	Let $G$ be a finite group and let $k$ be a field of characteristic~$p>0$. The derived category of permutation modules $\DPerm(G;k)$ is unigenic if and only if~$G$ is the trivial group.
\end{Prop}

\begin{proof}
	Recall that we have a localization $\DPerm(G;k) \twoheadrightarrow \DRep(G;k)$ and that $\Der(k) \cong \unitation{\DPerm(G;k)}$. Thus, if $\DPerm(G;k)$ is unigenic then we would have a localization $\Der(k) \twoheadrightarrow \DRep(G;k)$. Since any exact functor on $\Der(k)$ is faithful, we would thus have $\Der(k) \cong \DRep(G;k)$ which is false unless $G=1$ since any nontrivial finite group has nontrivial cohomology.
\end{proof}

\begin{Rem}
	More generally, consider $\DPerm(G;R)$ and $\DRep(G;R)$ for any noetherian ring $R$. Recall from \cref{exa:DRepR} that we have a factorization
	\[
		\DPerm(G;R) \twoheadrightarrow \concentration{\DRep(G;R)}{\cat G} \hookrightarrow \DRep(G;R)
	\]
	of the canonical functor $\DPerm(G;R) \to \DRep(G;R)$ and that the inclusion $\concentration{\DRep(G;R)}{\cat G} \hookrightarrow \DRep(G;R)$ can be strict if $R$ is non-regular. Since the graded comparison map for $\DRep(G;R)$ is a homeomorphism by \cite[Theorem~1.3]{Lau23}, it follows as above that the concentration $\concentration{\DRep(G;R)}{\cat G} \hookrightarrow \DRep(G;R)$ induces a homeomorphism on Balmer spectra. We conclude that although the canonical functor $\DPerm(G;R) \to \DRep(G;R)$ is not always a localization it nevertheless always induces an embedding
	\[
		\Spc(\DRep(G;R)^c) \hookrightarrow \Spc(\DPerm(G;R)^c).
	\]
\end{Rem}

\begin{Rem}
	We now turn to the equivariant stable homotopy category and related categories of spectral Mackey functors. We will take for granted some familiarity with the notation and results of \cite{BalmerSanders17} and \cite{PatchkoriaSandersWimmer22}. Our first goal is to prove that the unitation of the category of derived Mackey functors identifies, on Balmer spectra, with the spectrum of the Burnside ring. We will reduce this theorem to the following proposition:
\end{Rem}

\begin{Prop}\label{prop:DHZG-key-prop}
	Let $G$ be a finite $p$-group and let
	\[
		\varphi:\Spc(\DHZG^c) \to \Spc(\unitation{\DHZG}^c)
	\]
	be the map induced by unitation. Then $\varphi(\cat P(H,p)) = \varphi(\cat P(K,p))$ for all $H,K \le G$.
\end{Prop}

\begin{proof}
	Because unitation commutes with algebraic localization (\cref{rem:algebraic-commute}), we are reduced to the analogous statement for $\DHZG_{(p)}\cong\DHZpG$. In the spectrum of the category of derived Mackey functors we have $\cat P(K,p) \subseteq \cat P(H,p)$ if and only if~$K$ is $G$-conjugate to a $p$-subnormal subgroup of $H$. Thus, in the case where $G$ is a $p$-group we have $\cat P(1,p) \subseteq \cat P(H,p) \subseteq \cat P(G,p)$ for any $H \le G$. Hence $\varphi(\cat P(1,p))\subseteq \varphi(\cat P(H,p))\subseteq \varphi(\cat P(G,p))$. It thus suffices to prove that $\varphi(\cat P(G,p))\subseteq\varphi(\cat P(1,p))$. Recall that the prime $\cat P(G,p)$ is the kernel of the geometric fixed point functor
	\[
		\Phi^G : \DHZpG^c \to \Der(\HZp)^c.
	\]
	We claim that $\varphi(\cat P(G,p)) = \thick\langle\unit\rangle \cap \Ker(\Phi^G)$ is the zero ideal. Recall from \cref{exa:SHG-weight} that $\SHG^c$ has a bounded weight structure generated by the orbits~$G/H_+$. In a similar fashion, $\DHZpG^c$ admits a weight structure generated by the (images of) the orbits. In particular, its unitation $\unitation{\DHZpG}^c=\thick\langle \unit\rangle$ admits a weight structure generated by $\unit = G/G_+$ as in \cref{exa:weight-connective}. The heart of this weight structure is the the additive tensor-category $\proj(A(G)_{(p)})$ of finitely generated projective $A(G)_{(p)}$-modules. Since the two functors 
	\begin{equation}\label{eq:weight-composite}
		\unitation{\DHZpG}^c \hookrightarrow \DHZpG^c \xrightarrow{\Phi^G} \Der(\HZp)^c
	\end{equation}
	are weight-exact, we have a commutative diagram
	\[\begin{tikzcd}
		\unitation{\DHZpG}^c \ar[r] \ar[d,"W"'] & \Der(\HZp)^c\ar[d,"W"]\\
		K^b(\proj(A(G)_{(p)})) \ar[r] & K^b(\proj(\mathbb{Z}_{(p)}))
	\end{tikzcd}\]
	where the vertical functors are the weight complex functors. By~\cref{rem:heart-base-change}, the bottom functor 
	\[
		\Der(A(G)_{(p)})^c\cong K^b(\proj(A(G)_{(p)})) \to K^b(\proj(\mathbb{Z}_{(p)}))\cong \Der(\mathbb{Z}_{(p)})^c
	\]
	may be identified with extension-of-scalars with respect to the ring homomorphism on endomorphism rings induced by $\Phi^G$. This is the $p$-localization of the ring homomorphism $A(G)\to \mathbb{Z}$ which sends a finite $G$-set $[X]$ to $|X^G|$; see \cite[Remark~3.8]{BalmerSanders17}. Since $G$ is a $p$-group, the $p$-local Burnside ring $A(G)_{(p)}$ has a unique closed point which is hit by the unique closed point of $\mathbb{Z}_{(p)}$; see \cite[Theorem~3.6]{BalmerSanders17} or \cite{Dress69}. It then follows, for example from \cite[Theorem~1.2]{Balmer18}, that the bottom functor is conservative. Since the weight complex functors are conservative (\cref{thm:tensor-weight}), it follows that the top functor~\eqref{eq:weight-composite} is conservative. That is, $\varphi(\cat P(G,p)) =  (0)$, which proves the result.
\end{proof}

\begin{Thm}\label{thm:DHZG-comp}
	Let $G$ be a finite group. The comparison map
	\[
		\rho:\Spcunit{\DHZG} \to \Spec(A(G))
	\]
	is a homeomorphism.
\end{Thm}

\begin{proof}
	Consider the commutative diagram
	\[\begin{tikzcd}
		\Spc(\DHZG^c) \ar[dr,bend left=20,"\rho"] \ar[d,"\varphi"']&\\
		\Spcunit{\DHZG} \ar[r,"\rho"'] & \Spec(A(G)).
	\end{tikzcd}\]
	It suffices to prove that the bottom comparison map is injective since \cref{thm:connective-unigenic} establishes that it is a closed quotient map; cf.~\cref{cor:ungraded-comparison}. Since $\varphi$ is surjective, it suffices to prove that for any $\cat P,\cat Q \in \Spc(\DHZG^c)$, if $\rho(\cat P)=\rho(\cat Q)$ then $\varphi(\cat P)=\varphi(\cat Q)$. By the explicit description of $\rho$ given in \cite[Corollary~2.42]{PatchkoriaSandersWimmer22}, this reduces to the following claim: For any prime $p$ and subgroup $H\le G$, we have $\varphi(\cat P(H,p))=\varphi(\cat P(O^p(H),p))$. Considering the restriction
	\[\begin{tikzcd}
		\DHZG \ar[r,"\res^G_H"] & \DHZH \\
		\DHZG_{\langle\unit\rangle}\ar[u,hook] \ar[r] & \DHZH_{\langle\unit\rangle}\ar[u,hook]
	\end{tikzcd}\]
	and recalling \cite[Lemma~2.12]{PatchkoriaSandersWimmer22}, we can assume without loss of generality that $H=G$. Similarly, for any normal subgroup $N \lenormal G$ we have
	\[\begin{tikzcd}
		\DHZG \ar[r,"\widetilde{\Phi}^N"] & \DHZGN \\
		\DHZG_{\langle \unit \rangle} \ar[u,hook] \ar[r] & \DHZGN_{\langle \unit \rangle}\ar[u,hook].
	\end{tikzcd}\]
	Taking $N=O^p(G)$ and recalling \cite[Lemma~2.15]{PatchkoriaSandersWimmer22} we are reduced to the claim that in the $p$-group $G/O^p(G)$ we have $\varphi(\cat P(1,p))=\varphi(\cat P(G/O^p(G),p))$ which is provided by \cref{prop:DHZG-key-prop}.
\end{proof}

\begin{Rem}\label{rem:AG-gluing}
	The quotient map
	\[\begin{tikzcd} \Spc(\DHZG^c) \ar[dr,bend right=15,"\rho"'] \ar[r,"\varphi"] & \Spc(\DHZG_{\langle\unit\rangle}) \ar[d,"\rho","\cong"'] \\
		& \Spec(A(G))
	\end{tikzcd}\]
	is explicitly described in \cite[Corollary~2.42]{PatchkoriaSandersWimmer22}. In particular, for any two primes $\cat P,\cat Q \in \Spc(\DHZG^c)$, we have $\rho(\cat P)=\rho(\cat Q)$ if and only if $\cat P=\cat P(H,p)$ and $\cat Q=\cat P(K,p)$ for some prime number $p$ and subgroups $H,K\le G$ such that $H \cap K$ is a $p$-subnormal subgroup of $H$ and $K$.
\end{Rem}

\begin{Rem}
	The above results show that for the category of derived Mackey functors unitation effects a nontrivial quotient on the Balmer spectrum. In contrast, we will next show that unitation does not change the Balmer spectrum at finite heights.
\end{Rem}

\begin{Def}\label{def:chromatic-truncation}
	Let $0\le n <\infty$. Then $Y_n \coloneqq \SET{\cat C_{p,h}}{h>n}\subseteq \Spc(\SH^c)$ is a Thomason subset with an associated idempotent triangle $\en \to \Sphere \to \fn \to \Sigma \en$. Consider the canonical functor $\triv_G:\SH\to\SHG$ from \cref{not:triv} and let $\varphi\coloneqq\Spc(\triv_G):\Spc(\SHG^c) \to \Spc(\SH^c)$. We have an induced finite localization on~$\SHG$ with idempotent triangle $\triv_G(\en)\to \Sphere_G \to \triv_G(\fn) \to \Sigma\triv_G(\en)$ which corresponds to the Thomason subset
	\[
		Y_{G,n}\coloneqq\varphi^{-1}(Y_n) = \SET{\cat P(H,p,h)}{h>n}\subseteq \Spc(\SHG^c).
	\]
	We call $\SH_{\le n} \coloneqq \SH\hspace{-.3ex}|_{Y_n^\cc}$ and $\SH_{G,\le n} \coloneqq \SH_G\hspace{-.3ex}|_{Y_{G,n}^\cc}$ the \emph{chromatic truncations} of $\SH$ and $\SHG$ below height $n$. A $p$-localized version of this truncation construction is discussed in \cite[Example 6.15 and Warning 6.16]{AroneBarthelHeardSanders24pp}.
\end{Def}

\begin{Thm}\label{thm:truncation}
	Let $G$ be a finite group and $0 \le n <\infty$. Let $\SH_{G,\le n}$ denote the truncation at chromatic height $\le n$. The unitation induces a homeomorphism
	\[
		\Spc(\SH_{G,\le n}^c) \xrightarrow{\cong} \Spcunit{(\SH_{G,\le n})}.
	\]
\end{Thm}

\begin{proof}
	The functor $\triv_G:\SH\to\SHG$ induces a corresponding functor on truncations $\SH_{\le n} \to \SH_{G,\le n}$ which, since $\SH_{\le n}$ is unigenic, factors as 
	\[
		\SH_{\le n} \to \unitation{(\SH_{G,\le n})} \hookrightarrow \SH_{G,\le n}.
	\]
	Consider the induced maps on Balmer spectra:
	\[
		\Spc(\SH_{G,\le n}^c) \xrightarrow{\varphi} \Spcunit{(\SH_{G,\le n})} \xrightarrow{\psi} \Spc(\SH_{\le n}^c).
	\]
	For any $\cat P \in \Spc(\SH^c_{G,\le n})$, we have the inclusion
	\[
		\varphi^{-1}(\{\varphi(\cat P)\}) \subseteq \varphi^{-1}(\psi^{-1}(\{\psi(\varphi(\cat P))\}).
	\]
	Any such prime $\cat P$ is of the form $\cat P=\cat P(H,\cat C)$ for a subgroup $H \le G$ and nonequivariant prime $\cat C \in \Spc(\SH_{\le n}^c)$. Moreover, $\psi(\varphi(\cat P(H,\cat C))) = \cat C$ by \cite[Corollary~4.6]{BalmerSanders17}. \Cref{thm:main-thmb} implies that the fiber $\varphi^{-1}(\{\varphi(\cat P(H,\cat C))\})$ is a \emph{connected} subset of
	\[
		\varphi^{-1}(\psi^{-1}(\{\psi(\varphi(\cat P(H,\cat C)))\})) = \varphi^{-1}(\psi^{-1}(\{\cat C\})) = \SET{\cat P(K,\cat C)}{K \le G}
	\]
	but the latter space is discrete by \cite[Proposition~8.1]{BalmerSanders17}. We thus conclude that $\varphi^{-1}(\{\varphi(\cat P(H,\cat C))\})=\{\cat P(H,\cat C)\}$. This establishes that $\varphi$ is injective and hence is a homeomorphism since it is a spectral quotient map (\cref{cor:spectral-homeo}).
\end{proof}

Putting \cref{thm:DHZG-comp} and \cref{thm:truncation} together, we obtain:

\begin{Thm}\label{thm:SHG}
	Let $G$ be a finite group. The unitation $(\SH_G)_{\langle \unit \rangle}\hookrightarrow \SH_G$ is, on Balmer spectra, a topological quotient 
	\begin{equation}\label{eq:SHG}
		\varphi:\Spc(\SH_G^c)\to \Spc((\SH_G)_{\langle \unit \rangle}^c)
	\end{equation}
	which identifies those points at chromatic height $\infty$ which become identified in the spectrum of the Burnside ring. More precisely, two points $\cat P$ and $\cat Q$ are identified if and only if $\cat P=\cat P(H,p,\infty)$ and $\cat Q = \cat P(K,p,\infty)$ for a prime $p$ and subgroups $H,K\le G$ such that $H \cap K$ is a $p$-subnormal subgroup of $H$ and $K$.
\end{Thm}

\begin{proof}
	Let $\cat P=\cat P(H,p,n)$ and $\cat Q=\cat P(K,q,m)$ be arbitrary points of $\Spc(\SH_G^c)$. First, suppose $n$ and $m$ are finite, say with $m \le n$. Consider the commutative diagrams
	\begin{equation}\label{eq:SHG-trunc}
	\begin{tikzcd}[column sep=small]
		\SHG \ar[r] & \SH_{G,\le n} \ar[r,shift right=1.8em,phantom,"\mapsto"]& \Spc(\SHG^c)\ar[d] & \Spc(\SH_{G,\le n}^c)\ar[l]\ar[d,"\cong"]\\
		\unitation{(\SHG)} \ar[u,hook] \ar[r] & \unitation{(\SH_{G,\le n})} \ar[u,hook] & 
		\Spcunit{(\SHG)} & \Spcunit{(\SH_{G,\le n})}\ar[l,hook']
	\end{tikzcd}
	\end{equation}
	Recall from \cref{def:chromatic-truncation} that the Thomason subset defining the truncation $\SHG \to \SH_{G,\le n}$ is pulled back from $\SH$ via the functor $\triv_G : \SH\to \SH_G$. Since $\SH$ is unigenic, this functor factors through the unitation: $\SH \to \unitation{(\SHG)} \hookrightarrow \SH_G$. Hence the Thomason subset defining the truncation is pulled back through $\unitation{(\SHG)}$. This implies (\cref{rem:unitation-localization}) that $\unitation{(\SH_{G,\le n})} \cong (\unitation{(\SH_G)})_{\le n}$ and the bottom functor of \eqref{eq:SHG-trunc} is a finite localization. In particular, the bottom tt-functor in \eqref{eq:SHG-trunc} induces an embedding on Balmer spectra, as displayed on the right-hand side. On the other hand, \cref{thm:truncation} establishes that the right vertical functor induces a homeomorphism on spectra. Thus, the composite map is injective. Since our two points $\cat P$ and $\cat Q$ both arise in the top-right corner, an equality $\varphi(\cat P)=\varphi(\cat Q)$ would imply that $\cat P=\cat Q$. This establishes that two distinct points of finite chromatic height are not identified by $\varphi$.

	We may similarly show that a point of finite chromatic height is not identified with a point of infinite chromatic height. Indeed, suppose that $\varphi(\cat P(H,p,n))=\varphi(\cat P(K,q,\infty))$. Since the comparison map $\rho:\Spc(\SH_G^c)\to\Spec(A(G))$ factors through $\varphi$ we may assume $p=q$ without loss of generality (see \cite[Proposition~6.7 and Theorem~3.6]{BalmerSanders17}). Then, since $\cat P(K,p,\infty) \subseteq \cat P(K,p,n)$ we would have $\varphi(\cat P(H,p,n)) \subseteq \varphi(\cat P(K,p,n))$ and we can again contemplate the diagram above. Note that the inclusion $\cat P(H,p,n)\subseteq \cat P(K,p,n)$ does \emph{not} hold in $\SH_{G,\le n}$ (see \cite[Proposition~8.1]{BalmerSanders17}). Hence, since the composite of the diagram is a homeomorphism followed by an embedding, we cannot have $\varphi(\cat P(H,p,n))\subseteq \varphi(\cat P(K,p,n))$. We thus have a contradiction.

	Finally, suppose $n=m=\infty$. Consider the diagram
	\[ \begin{tikzcd}
		\SH_G \ar[r] & \DHZG \\
		(\SH_G)_{\langle \unit \rangle} \ar[u,hook] \ar[r] & (\DHZG)_{\langle \unit \rangle} \ar[u,hook]
	\end{tikzcd}\]
	Applying the Balmer spectrum, we have 
	\[\begin{tikzcd}
		\Spc(\SH_G^c) \ar[d,"\varphi"] \ar[dd,start anchor=south west,end anchor=north west,bend right=35,"\rho"'] & \Spc(\DHZG^c) \ar[l] \ar[d] \\
		\Spc((\SH_G)_{\langle \unit \rangle}^c) \ar[d,"\rho"] & \Spc((\DHZG)_{\langle \unit \rangle}^c)\ar[l] \ar[dl,bend left=20,"\rho"',"\cong"]\\
		\Spec(A(G)).
	\end{tikzcd}\]
	where the bottom-right arrow is a homeomorphism by \cref{thm:DHZG-comp}. Since the primes $\cat P(H,p,\infty)$ and $\cat P(K,q,\infty)$ are the image of unique primes in the top-right (see \cite[Corollary~2.19 and Corollary~2.40]{PatchkoriaSandersWimmer22}), it follows from the diagram that $\varphi(\cat P(H,p,\infty))=\varphi(\cat P(K,q,\infty))$ if and only if $\rho(\cat P(H,p,\infty))=\rho(\cat P(K,q,\infty))$. By \cite[Proposition~6.7]{BalmerSanders17}, this is the case if and only if $\mathfrak p(H,p) = \mathfrak p(K,q)$ in $\Spec(A(G))$ and by Dress' Theorem (\cite{Dress69} and \cite[Theorem~3.6]{BalmerSanders17}) this is the case if and only if $p=q$ and $H\cap K$ is a $p$-subnormal subgroup of both $H$ and $K$; cf.~\cref{rem:AG-gluing}.
\end{proof}

\begin{Exa}
	Let $G=C_p$ be the cyclic group of order $p$. In this case, unitation glues together a single pair of points, namely the points at chromatic height $\infty$ at the prime~$p$ for the two subgroups. See \cref{fig:SHCp} on page \pageref{fig:SHCp}.
\end{Exa}

\begin{Exa}
	Let $G$ be a finite $p$-group and let $\SH_{G,(p)}$ be the $p$-local $G$-equivariant category. The unitation $\Spc(\SH_{G,(p)}^c) \to \Spc((\SH_{G,(p)})_{\langle \unit \rangle}^c)$ is the topological quotient which glues together all the points at chromatic height $\infty$.
\end{Exa}

\begin{Rem}
	We may illustrate the relation between the above results $p$-locally for $G=C_p$ as follows:
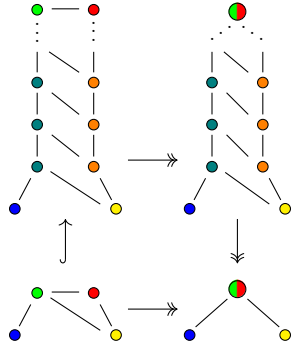
\[\begin{tikzcd}[row sep=2ex,column sep=1.0ex]
&[-1.5ex]{\tikzbullet{green}}\ar[rr,shift right=0.5ex,-] &[-0.75ex]&[-0.75ex]{\tikzbullet{red}}&[-1.5ex]&&&[-1.5ex]{} &[-1.0ex] |[yshift=-0.5ex]|{\smash{\tikztwocircle{green}{red}} } &[-1.0ex]{}&[-1.5ex]\\[-0.5ex]
&{} \ar[u,shift right=.820ex,"\raisebox{+0.75ex}{$\vdots$}",draw=none]  &&{}\ar[u,shift right=.820ex,"\raisebox{+1.0ex}{$\vdots$}",draw=none]&&&&{} \ar[ur,yshift=-0.85ex,xshift=1.0ex,shift right=.820ex,"\raisebox{+0.0ex}{$\iddots$}",draw=none]  &&{}\ar[ul,yshift=1.5ex,shift right=1.820ex,"\raisebox{0.0ex}{$\ddots$}",draw=none]  \\[-1.0ex]
&{\tikzbullet{teal}} \ar[u,start anchor=north,end anchor=real center,-]  &&{\tikzbullet{orange}}\ar[ull,start anchor=north west, end anchor=real east,-]\ar[u,start anchor=north,end anchor=real center,-]&&&&{\tikzbullet{teal}} \ar[u,start anchor=north,end anchor=real center,-]  &&{\tikzbullet{orange}} \ar[ull,start anchor=north west, end anchor=real east,-] \ar[u,start anchor=north,end anchor=real center,-]  \\
&{\tikzbullet{teal}} \ar[u,-]  &&{\tikzbullet{orange}}\ar[ull,start anchor=north west, end anchor=south east,-]\ar[u,-] &&&&{\tikzbullet{teal}} \ar[u,-]  &&{\tikzbullet{orange}}\ar[ull,start anchor=north west, end anchor=south east,-]\ar[u,-] \\
&{\tikzbullet{teal}} \ar[u,-]  &&{\tikzbullet{orange}}\ar[ull,start anchor=north west, end anchor=south east,-]\ar[u,-]&\ar[rr,two heads]& &{}&{\tikzbullet{teal}} \ar[u,-]  &&{\tikzbullet{orange}}\ar[ull,start anchor=north west, end anchor=south east,-]\ar[u,-] \\
	{\tikzbullet{blue}} \ar[ur,shift right=0.25ex,-]&  &{}&{}&{\tikzbullet{yellow}}\ar[ul,shift left=0.25ex,-] \ar[ulll,start anchor=west,shift left=0.20ex,-] &{}& {\tikzbullet{blue}} \ar[ur,shift right=0.25ex,-] &{}&{}&{}& {\tikzbullet{yellow}} \ar[ul,shift left=0.25ex,-] \ar[ulll,start anchor=west,shift left=0.20ex,-] \\
\\
													&                                                   & \ar[uu,hook] &&&&&&\ar[uu,twoheadleftarrow]\\[-2ex]
													&{\tikzbullet{green}} \ar[rr,shift right=0.50ex,-]  &&{\tikzbullet{red}} &{}\ar[rr,two heads,shift right=2.0ex]&{}&{}&{}&{\tikztwocircle{green}{red}} \\
{\tikzbullet{blue}} \ar[ur,shift right=0.25ex,-]&  &&&{\tikzbullet{yellow}}\ar[ul,shift left=0.25ex,-] \ar[ulll,start anchor=west,shift left=0.20ex,-] &&{\tikzbullet{blue}} \ar[urr,shift right=0.25ex,-] &&&& {\tikzbullet{yellow}}\ar[ull,shift left=0.25ex,-]
\end{tikzcd}\]
	In the bottom-left we have the spectrum of the $p$-localization of $\DHZCp$ and in the bottom-right we have the spectrum of the $p$-localization of $A(C_p)$. In the top-left we have the spectrum of the $p$-localization of $\SH_{C_p}$ and in the top-right we have the spectrum of its unitation. The two maps to the bottom-right are the comparison maps, the top map is unitation, and the left map is the embedding induced by the canonical functor $\SH_{C_p}\to\DHZCp$; cf.~\cref{rem:linear-embedding}. Note that the bottom map can also be regarded as the unitation of the $p$-localization of $\DHZCp$.
\end{Rem}

\begin{Rem}\label{rem:DHZG-not-burnside}
	In light of \cref{thm:DHZG-comp}, it is worth emphasizing that the unitation $\unitation{\DHZG}$ is not tt-equivalent to $\Der(A(G))$. Indeed,
	\[ 
		\End_{\DHZG}^{-\bullet}(\unit)= \pi_\bullet(\lambda^G(\triv_G(\HZ))) = \pi_\bullet(\HZ \wedge \lambda^G(\Sphere_G)) = \HZ_\bullet(\lambda^G(\Sphere_G)).
	\]
	By the tom Dieck splitting, $\lambda^G(\Sphere_G)$ contains a copy of the classifying space $BG$ and hence the homology is not concentrated in degree 0. Indeed, recall from \cref{exa:HRG} that we have a tt-equivalence
	\[
		\DHZG_{\langle \unit \rangle} \cong \Ho(\HZ \wedge \lambda^G(\Sphere_G))\text{-Mod}_{\Sp}) \eqqcolon \Der(\HZ\wedge \lambda^G(\Sphere_G)).
	\]
	On the other hand, we may also consider the derived category $\Der(\Mack(G;\mathbb{Z}))$ of the abelian category of $G$-Mackey functors. As established in \cite[Theorem~5.10]{PatchkoriaSandersWimmer22},
	\[
		\Der(\Mack(G;\mathbb{Z}))\cong \Ho(\HA_G\text{-Mod}_{\Sp_G})
	\]
	where $\HA_G$ is the Eilenberg--MacLane $G$-spectrum associated to the Burnside ring $G$-Mackey functor $\bbA_G$. Again, by \cref{rem:equiv-monadic}, its unitation $\Der(\Mack(G;Z))_{\langle \unit \rangle}$ is equivalent to modules over the ordinary Eilenberg--MacLane spectrum $\lambda^G(\HA_G)=\rmH \hspace{-0.2em}A(G)$ associated to the Burnside ring. That is, the unitation
	\[
		\unitation{\Der(\Mack(G;\mathbb{Z}))} \cong \Der(A(G))
	\]
	is equivalent to the derived category of the Burnside ring $A(G)$. Thus, we are in the curious situation where $\Der(\HZG)$ and $\Der(\HA_G)$ have inequivalent unitations, and yet their unitations have the same Balmer spectrum.
\end{Rem}

\begin{Rem}
	On the other hand, if we take rational coefficients then
	\[
		\unitation{\DHQG} \cong \unitation{\Der(\Mack(G;\mathbb{Q}))} \cong \Der(A(G)_{\mathbb{Q}}).
	\]
	Indeed, $\lambda^G(\triv_G(\HQ)) \to \lambda^G(\rmH(\bbA_G \otimes \mathbb{Q}))$ is an equivalence essentially because the cohomology of a finite group is torsion.
\end{Rem}

\begin{Rem}\label{rem:inflation-in-SHG}
	As we already mentioned in \cref{exa:SHG-inflation}, and as the above results also demonstrate, for a normal subgroup $N \lenormal G$, the inflation functor $\SHGN\to \SHG$ is not fully faithful and hence does not induce an equivalence $\SHGN \to (\SHG)_{\langle \cat F[{\supseteq}N]\rangle}$. The inflation functor is, however, faithful (since it is split by geometric fixed points, for example). Hence $\Spc(\SHG^c)\to \Spc(\SHGN^c)$ is a spectral quotient. For example, taking $N=G$, the map $\Spc(\SHG^c) \to \Spc(\SH^c)$ induced by $\triv_G:\SH\to\SHG$ is a spectral quotient map. Note that the fibers are disconnected.  Indeed, each fiber is a discrete space with one point for each conjugacy class of subgroups of $G$. This lies at the heart of our proof of \cref{thm:truncation}. In any case, this is another example showing that faithful functors need not have connected fibers.
\end{Rem}

\begin{Rem}\label{rem:unitation-different}
	Recall the sequence of tt-functors displayed in \eqref{eq:chain}. In summary, we have determined the unitation of $\SHG$ and seen that it sits properly between $\Spc(\SHG^c)$ and $\Spec(A(G)$. In contrast, we showed that the unitation of $\DHZG$ is homeomorphic to $\Spec(A(G))$ while the unitation of $\Der(\HA_G)$ is actually equivalent to $\Der(A(G))$ as tt-categories. On the other hand, the unitation of $\DPerm(G;k)$ is equivalent to $\Der(k)$ essentially because inflation in this setting is fully faithful. Finally $\DRep(G;k)$ coincides with its unitation (categorically for $p$-groups and on spectra in general) essentially because its comparison map is a homeomorphism. Note that each of these examples exhibits qualitatively different behaviour.
\end{Rem}

\section{Twisted cohomology and the Picard group}\label{sec:twisted}

\begin{Rem}
	The most general comparison map
	\[
		\rho_u : \SpcT \to \Spec(R_u)
	\]
	constructed by Balmer in \cite{Balmer10b} depends on the choice of an element $u \in \Pic(\cT)$ in the Picard group. The target is the homogeneous spectrum of the $\bbZ$-graded ring $R_u = \Hom_{\cT}(\unit,u^{\otimes \bullet})$. For $u\neq \unit$, this map $\rho_u$ does not (necessarily) factor through the unitation of $\cT$. Instead we may consider the concentration of $\cT$ at the subgroup of the Picard group generated by $u$: $\concentration{\cT}{\cat G} \hookrightarrow \cT$ for $\cat G=\SET{u^{\otimes n}}{n\in\bbZ}$. For simplicity, we'll just denote this by $\concentration{\cT}{u}$. We then have a commutative diagram
	\[\begin{tikzcd}
		\Spc(\cTc) \ar[d,"\varphi"'] \ar[dr,bend left,"\rho_u"] & \\
		\Spc(\concentration{\cT}{u}^c) \ar[r,"\rho_u"] & \Spec(R_u).
	\end{tikzcd}\]
	Moreover, note that all of these comparison maps factor through the concentration of $\cT$ at the Picard group: $\concentration{\cT}{\Pic(\cT)}$. For categories that are not generated by their Picard group, it is then of interest to consider the maps
	\[
		\Spc(\cTc) \twoheadrightarrow \Spc(\concentration{\cT}{\Pic(\cT)}^c) \twoheadrightarrow \Spcunit{\cT}.
	\]
\end{Rem}

\begin{Exa}\label{exa:ui-invertibles}
	For $\cT=\DPerm(G;k)$, we have $\unitation{\DPerm(G;k)} \simeq \Der(k)$ by \cref{exa:DPerm-inflation} and the usual graded comparison map provides very little information. However, in their computation of the spectrum of $\cT$ in \cite{BalmerGallauer23pp}, Balmer and Gallauer reduce to the case of an elementary abelian $p$-group $G=E$ and then use a multi-graded generalization of $\rho_u$. More precisely, let $N_1,\ldots, N_\ell$ be the maximal subgroups of $E$. For each $1 \le i \le \ell$, there is an associated invertible object $u_i \in \Pic(\cT)$. Also let $u_0 = \Sigma \unit$. They construct a continuous map
	\[
		\rho :\Spc(\cTc) \to \Spec(H^{\bullet,\bullet}(E))
	\]
	to the homogeneous spectrum of the ``twisted cohomology'' ring
	\[
		H^{\bullet,\bullet}(E) \coloneqq \Hom_{\cT}(\unit,u_0^{\otimes q_0}\otimes u_1^{\otimes q_1} \otimes u_2^{\otimes q_2}\otimes \cdots u_\ell^{\otimes q_\ell})
	\]
	which is graded over the monoid $\bbZ \times \bbN^{\ell}$. In general, the functoriality of this graded ring under a tt-functor involves a grading shift which depends on the particular tt-functor. However, this detail doesn't arise for a fully faithful functor. In particular, we may consider the concentration of $\cT$ at the subgroup of the Picard group generated by the $u_i$: $\cat G=\langle u_0, u_1, \ldots, u_\ell\rangle = \{\otimes_{i=0}^\ell u_i^{\otimes q_i} \mid q_i \in \bbZ \text{ for } 0\le i\le \ell\}$. We have a factorization
	\[\begin{tikzcd}
		\Spc(\cTc) \ar[d,"\varphi"'] \ar[dr,bend left,"\rho"] & \\
		\SpcTconc{\cat G} \ar[r,"\rho"] & \Spec(H^{\bullet,\bullet}(E)).
	\end{tikzcd}\]
	The top $\rho$ is an open embedding by \cite[Corollary~15.6]{BalmerGallauer23pp}. Hence the vertical~$\varphi$ is injective. Since it is a spectral quotient by \cref{thm:faithful-are-quotient}, we conclude that  $\varphi$ is a homeomorphism. It follows that any intermediate concentration also induces a homeomorphism. For example, the concentration $\Tconc{\Pic(\cT)} \hookrightarrow \cT$ to the entire Picard group induces a homeomorphism on Balmer spectra. However, this is not particularly interesting since Balmer and Gallauer prove that~$\cat T=\Tconc{\Pic(\cT)}$ in this example. Nevertheless, the above discussion provides some ideas that could be of interest in other examples.
\end{Exa}

\begin{Rem}
	Another approach to understanding the concentration $\concentration{\cT}{\cat G}$ at a subgroup $\cat G\le \Pic(\cT)$ is via the comparison map 
	\[
		\Spc(\cT^c) \to \Spec(\cat R_{\cat G})
	\]
	constructed by Dell'Ambrogio and Stevenson in \cite{DellAmbrogioStevenson14}. In their terminology, the tt-subcategory~$\thick\langle \cat G\rangle$ is a central 2-ring of~$\cat T$ and $\Spec(\cat R_{\cat G})$ is its associated Zariski spectrum. By construction, this map factors through $\Spc(\cT^c)\to\Spc(\concentration{\cT}{\cat G}^c)$.
\end{Rem}

\begin{Rem}
	An expansive and general investigation of $\SpcT \twoheadrightarrow \SpcTconc{\Pic(\cT)}$ will not be provided in this work, but let us give the following result:
\end{Rem}

\begin{Prop}\label{prop:directed-subgroups}
	Let $\cT$ be a rigidly-compactly generated tt-category and let $\cat H$ be a directed poset of subgroups of the Picard group $\Pic(\cT)$. For $G\coloneqq \bigcup_{H \in \cat H} H$, we have $\SpcTconc{G} \cong \lim_{H \in \cat H} \SpcTconc{H}$.
\end{Prop}

\begin{proof}
	The assignment $H \mapsto \concentration{\cT}{H}$ forms a (strict) diagram of tt-categories and (fully faithful) tt-functors indexed on the directed post~$\cat H$. We claim that 
	\begin{equation}\label{eq:filtered}
		\Tconc{G} = \bigcup_{H \in \cat H}\Tconc{H}
	\end{equation}
	so that $\Tconc{G}$ is the colimit of the diagram. The proposition will then follow from \cite[Proposition~8.2]{Gallauer18}. First observe that the right-hand side of \eqref{eq:filtered} is a localizing subcategory of $\cT$. It then follows that $\Loc\langle \bigcup_{H \in \cat H} \Tconc{H}^c\rangle = \bigcup_{H \in \cat H} \Tconc{H}$ and thus it suffices to establish that
	\begin{equation}\label{eq:filtered2}
		\Tconc{G}^c \subseteq \bigcup_{H \in \cat H}\Tconc{H}^c.
	\end{equation}
	If $x \in \Tconc{G}^c = \thick\langle G\rangle$ then $x \in \thick\langle u_1,\ldots,u_n\rangle$ for some collection $u_1,\ldots,u_n \in G$. Since $G=\bigcup_{H \in \cat H} H$ and $\cat H$ is directed under inclusion, there exists an $H \in \cat H$ which contains all of these $u_i$. This establishes \eqref{eq:filtered2} and the proof is complete.
\end{proof}

\begin{Exa}
	The theorem applies to $G=\Pic(\cT)$ and $\cat H$ the family of finitely generated subgroups. 
\end{Exa}

\begin{Exa}
	We can also take $G$ to be the torsion subgroup of $\Pic(\cT)$ and $\cat H$ the collection of finite subgroups of $\Pic(\cT)$. Just note that since $\Pic(\cT)$ is abelian, the subgroup $\langle H \cup K\rangle=HK$ is finite for any two finite subgroups $H,K \le \Pic(\cT)$.
\end{Exa}

\begin{Rem}
	We note in passing that since the functors in the filtered diagram are fully faithful, it follows from \cref{prop:cofiltered} that the canonical maps $\smash{\SpcTconc{G}} \to \smash{\SpcTconc{H}}$ are spectral quotient maps --- but of course we already know this directly since these maps are induced by fully faithful functors.
\end{Rem}

\section{Local unigenicity\footnote{See \cite{ArndtLappe14}.}}\label{sec:locally-unigenic}

We saw in \cref{sec:equivariant} that a fully faithful functor can induce a homeomorphism on Balmer spectra without being an equivalence (\cref{rem:not-p-group-ff}). We now investigate this phenomenon more closely.

\begin{Def}\label{def:locally-unigenic}
	We say that $\cT$ is \emph{locally unigenic} if the local category $\cT_{\cat P}\coloneqq \cT|_{\gen(\cat P)}$ is unigenic for each $\cat P\in\Spc(\cTc)$.
\end{Def}

\begin{Exa}
	For any quasi-compact and quasi-separated scheme $X$, the derived category $\Der(X)$ is locally unigenic.
\end{Exa}

\begin{Prop}\label{prop:homeo}
	Let $f^*:\cT \to \cS$ be a fully faithful functor such that $\varphi:\Spc(\cSc)\to \Spc(\cTc)$ is a homeomorphism. Assume $\Spc(\cSc)$ is noetherian. If $\cS$ is locally unigenic then $f^*$ is an equivalence.
\end{Prop}

\begin{proof}
	It suffices to establish that the right adjoint $f_*$ is conservative (\cref{rem:equiv-cons}). To this end, consider any $\cat P \in \SpcS$. Since $\varphi$ is a homemorphism, we have
	\begin{equation}\label{eq:going-up}
		\varphi^{-1}(\gen(\varphi(\cat P)))= \gen(\cat P)
	\end{equation}
	and the associated corestriction
	\begin{equation}\label{eq:corestrict-blah2}
		\cT_{\varphi(\cat P)} = \cT|_{\gen(\varphi(\cat P))} \to \cS|_{\gen(\cat P)} = \cS_{\cat P}
	\end{equation}
	is fully faithful by \cref{rem:corestriction-is-fully-faithful}. Moreover, since $\cS_{\cat P}$ is unigenic, its right adjoint is conservative. Thus, \eqref{eq:corestrict-blah2} is an equivalence. Consider the commutative diagram
	\[\begin{tikzcd}
		\cT\ar[d,shift right] \ar[r,hook,"f^*"] & \cS \ar[d, shift right]\\
		\cT_{\varphi(\cat P)}\ar[u,shift right,hook] \ar[r,"\cong"'] & \cS_{\cat P}. \ar[u, hook, shift right]
	\end{tikzcd}\]
	A diagram chase shows that if $f_*(s)=0$ then $\altmathbb{f}_{\gen(\cat P)^\cc}\otimes s=0$ so that $\Supp(s) \subseteq \gen(\cat P)^\cc$ where $\Supp$ denotes the Balmer--Favi support \cite{BalmerFavi11,BarthelHeardSanders23a}. In particular, $\cat P \not\in\Supp(s)$. Since this is true for all $\cat P \in \SpcS$, we conclude that $s=0$ since the Balmer--Favi support has the detection property when $\SpcS$ is noetherian by \cite[Theorem~3.22]{BarthelHeardSanders23a}.
\end{proof}

\begin{Cor}\label{cor:homeo-unit}
	If $\Spc(\cTc)$ is noetherian and unitation $\unitation{\cT}\hookrightarrow \cT$ induces a homeomorphism~$\SpcT \smash{\xrightarrow{\sim}} \Spc(\unitation{\cT}^c)$ then $\cT$ is unigenic if and only if $\cT$ is locally unigenic.
\end{Cor}

\begin{proof}
	Apply \cref{prop:homeo} to $\cT_{\langle \unit \rangle} \hookrightarrow \cT$.
\end{proof}

\begin{Exa}
	Let $G$ be a finite group and $k$ a field of characteristic $p$. It follows from \cref{cor:homeo-unit} and \cref{exa:modular-unitation} that $\DRep(G;k)$ is unigenic if and only if it is locally unigenic. Moreover, this is the case if and only if $G$ is a $p$-group by \cref{prop:p-group}.
\end{Exa}

\begin{Rem}
	Similarly, \cref{cor:homeo-unit} and \cref{exa:stmod-unitation} imply that $\StMod(kG)$ is unigenic if and only if it is locally unigenic. Moreover, this holds if $G$ is a $p$-group, for example, by the localization $\DRep(G;k) \twoheadrightarrow \StMod(kG)$. However, it is possible for $\StMod(kG)$ to be unigenic more generally:
\end{Rem}

\begin{Exa}\label{exa:S3-unigenic}
	Let $G=S_3$ and $k=\Ftwo$. There are two simple modules: the trivial representation $k$ and the two-dimensional standard representation $V$. The regular representation decomposes as \[ k(S_3) \simeq k(S_3/A_3) \oplus V^{\oplus 2}.\] In particular, the two-dimensional simple representation is projective and hence annihilated in the stable module category. It follows that $\StMod(\Ftwo\hspace{1pt} S_3)$ is unigenic.
\end{Exa}

With the above results in mind, it will be convenient to introduce:

\begin{Def}\label{def:unigenic-locus}
	The \emph{unigenic locus} of $\cT$ is $\SET{\cat P \in \Spc(\cTc)}{\cT_{\cat P} \text{ is unigenic}}$. It is a generalization-closed subset of $\Spc(\cTc)$.
\end{Def}

\begin{Exa}
	It follows from \cref{thm:DHZG-comp} that the category of derived Mackey functors $\DHZG$ is not unigenic (unless $G=1$) since the comparison map to the Burnside ring is not injective.\footnote{For any prime $p$ dividing $|G|$ and Sylow $p$-subgroup $S \le G$, we have $\mathfrak p(S,p)=\mathfrak p(1,p)$ in the Burnside ring but $\cat P(S,p) \neq \cat P(1,p)$ in $\DHZG$; see \cite[Section~2]{PatchkoriaSandersWimmer22} and \cite[Section~3]{BalmerSanders17}.} What is its unigenic locus?
\end{Exa}

\begin{Prop}\label{prop:unigenic-locus-DHGZp}
	Let $G$ be a finite $p$-group. The unigenic locus of $\DHZpG$ is 
	\[
		\SETT{\cat P(G,p)} \cup \SET{\cat P(H,0)}{H \le G}.
	\]
\end{Prop}

\begin{proof}
	The category of derived Mackey functors with rational coefficients splits as a product of semisimple categories (see \cite{Wimmer19pp,BoucDellAmbrogioMartos24pp} for instance):
	\[
		\DHZpG \twoheadrightarrow \DHQG \cong \prod_{(H)} \Der(\mathbb{Q}[W_G H]).
	\]
	The target is unigenic by \cref{exa:finite-product}. Hence $\DHZpG$ is unigenic at each height zero point $\cat P(H,0)$, $H\le G$. On the other hand, the geometric fixed points functor $\Phi^G:\DHZpG \to \DHZp$ exhibits~$\DHZp$ as the local category at~$\cat P(G,p)$. Hence $\DHZpG$ is unigenic at $\cat P(G,p)$. Then consider $\cat P(H,p)$ for a proper subgroup $H \lneq G$. Since $G$ is a $p$-group, there exists a $p$-subnormal tower from $H$ to $G$. In particular, there exists a subgroup $K\le G$ of index~$p$ which contains $H$. If the unigenic locus contains $\cat P(H,p)$ then it also contains $\cat P(K,p)$ since $\cat P(H,p)\subseteq \cat P(K,p)$. Thus, it suffices to prove that $\cat P(K,p)$ is not contained in the unigenic locus for each index $p$ subgroup of $G$. Since $\widetilde{\Phi}^K:\DHZpG \to \DHZpGK$ is a localization, $\DHZpG$ is unigenic at $\cat P(K,p)$ if and only if $\DHZpCp$ is unigenic at $\cat P(1,p)$. But note that $\cat P(1,p)$ is the unique closed point of $\DHZpCp$, so this would be saying that $\DHZpCp$ is unigenic. This is not the case. For example, it follows from the fact that the map on Balmer spectra of the unitation identifies with the comparison map to the Burnside ring (\cref{thm:DHZG-comp}), which is not injective.
\end{proof}

\begin{Cor}
	Let $G$ be a finite group which is not perfect. Then $\DHZG$ is not locally unigenic.
\end{Cor}

\begin{proof}
	Since $G$ is not perfect there exists a prime $p$ such that $O^p(G)$ is a proper subgroup of $G$. We have localizations $\DHZG \to \DHZpG \to \DHZpGOp$ where the first is $p$-localization and the second is geometric fixed points. The local category of $\DHZG$ at $\cat P_G(O^p(G),p)$ then identifies with the local category of $\DHZpGOp$ at $\cat P_{G/O^p(G)}(1,p)$. Since $G/O^p(G)$ is a nontrivial $p$-group, this local category is not unigenic by \cref{prop:unigenic-locus-DHGZp}.
\end{proof}

\begin{Rem}
	Stated differently, \cref{exa:S3-unigenic} demonstrates that it is possible for the unigenic locus of $\DRep(G;k)$ to be a non-empty proper subset of $\Spc(\DRep(G;k)^c)$. For $G=S_3$ and $k=\Ftwo$, it consists of the entirety of the spectrum $\Spc(\DRep(G;k)^c)\cong\Spec^h(H^\bullet(G;k))$ except for the unique closed point.
\end{Rem}

\begin{Exa}\label{exa:coho-open}
	We have already observed in \cref{prop:DPerm-unigenic} that $\DPerm(G;k)$ is unigenic if and only if $G=1$ is the trivial group. On the other hand, if $G$ is a $p$-group then the unigenic locus of $\DPerm(G;k)$ is non-empty. Indeed, it is unigenic at every prime~$\cat P$ in the so-called ``cohomological open'' $V_G \subseteq \Spc(\DPerm(G;k)^c)$ of \cite[Proposition~3.22]{BalmerGallauer23pp}. Indeed, these are precisely those primes pulled back via the localization $\DPerm(G;k)\twoheadrightarrow \DRep(G;k)$ and $\DRep(G;k)$ is unigenic when $G$ is a $p$-group.
\end{Exa}

\begin{Rem}
	If $\cT$ is not unigenic, then it is natural to attempt to construct a unigenic localization of $\cT$. For a $p$-group $G$, the localization ${\DPerm(G;k) \twoheadrightarrow \DRep(G;k)}$ of \cref{exa:coho-open} is an illustrative example. However, as the next example demonstrates, it is possible for the unigenic locus to be empty. For such categories, the only (locally) unigenic localization is the zero category.
\end{Rem}

\begin{Exa}
	Let $\cT$ be a nonzero rigidly-compactly generated tt-category and let $\cT \times_{\bbZ/2} \cT$ denote the product triangulated category with the $\bbZ/2$-graded $\otimes$-structure discussed in \cref{exa:Z2-graded-product}. The inclusion $x \mapsto (x,0)$ of the first factor is a fully faithful geometric functor which is a homeomorphism on spectra; indeed, the primes in $\Spc((\cT \times_{\bbZ/2}\cT)^c)$ are all of the form $\cat P \times \cat P$ for some $\cat P \in \Spc(\cTc)$. As pointed out in \cite[Remark~4.24]{Sanders22}, $\cT\times_{\bbZ/2}\cT$ is not unigenic at $\cat P\times \cat P$. That is, its unigenic locus is empty.
\end{Exa}

\begin{Rem}
	Nevertheless, assuming that the unigenic locus is non-empty, it is natural to attempt to find a unigenic localization of $\cT=\Loc\langle \cat G\rangle$ by annihilating generators. For example, the finite localization which annihilates $\Loco{\cat G \setminus \{\unit\}}$ is unigenic, although it is of course very possible for it to be the zero category; for example, this will be the case if $\cat G \setminus\{\unit\}$ contains an invertible object. All of these caveats notwithstanding, a situation where things work well arises in equivariant homotopy theory (and related examples) in connection with geometric fixed point functors:
\end{Rem}

\begin{Rem}\label{rem:geom-open}
	Let $\mathbb{E} \in \CAlg(\Sp_G)$. The category $\mathbb{E}\text{-Mod}_{\Sp_G}$ is rigidly-compactly generated by $\SET{F_\mathbb{E}(G/H_+)}{H \le G}$ where $F_\mathbb{E}:\Sp_G \to \mathbb{E}\text{-Mod}_{\Sp_G}$ is the canonical functor. The finite localization which annihilates the generators $F_{\mathbb{E}}(G/H_+)$ for all proper subgroups $H \lneq G$ can be identified with $\Phi^G(\mathbb{E})\text{-Mod}_{\Sp}$. In other words, the analogue for $\mathbb{E}$ of the geometric fixed point functor is the localization
	\[
		\mathbb{E}\text{-Mod}_{\Sp_G} \to \Phi^G(\mathbb{E})\text{-Mod}_{\Sp}.
	\]
	This localization singles out a certain open piece of $\Spc(\mathbb{E}\text{-Mod}_{\Sp_G}^c)$ homeomorphic to $\Spc(\Phi^G(\mathbb{E})\text{-Mod}_{\Sp}^c)$ which we might call the ``geometric open'' following \cite{BalmerGallauer23pp}. As we have already seen, for some specific choices of $\mathbb{E}$ (such as $\mathbb{E}=\Sphere_G$ or $\mathbb{E}=\triv_G(\HZ)$) the resulting geometric fixed points are well-understood, whereas for other examples (such as $\mathbb{E}=\HA_G$ or $\mathbb{E} = \smash{\HZbar}$) the resulting geometric fixed points are more mysterious; see further discussion in \cite[Section~5]{PatchkoriaSandersWimmer22} and \cite[Section~3.2]{BCHNP1}. Nevertheless, even if we do not have a good understanding of the ring spectrum $\Phi^G(\mathbb{E})$, the category of modules $\Phi^G(\mathbb{E})\text{-Mod}_{\Sp}$ is unigenic. In other words, the unigenic locus of $\mathbb{E}\text{-Mod}_{\Sp_G}$ contains the geometric open. Thus, the unigenic locus is non-empty provided the geometric open is non-empty, that is, provided that~$\Phi^G(\mathbb{E})\neq 0$.
\end{Rem}

\begin{Exa}
	The geometric open of $\DHZG$ is $\SET{\cat P(G,p)}{p \text{ prime}} \cup \SETT{\cat P(G,0)}$. Compare with \cref{prop:unigenic-locus-DHGZp}.
\end{Exa}

\begin{Exa}
	Recall from \cref{exa:DRep-monadic} that $\DRep(G;k) \cong b_G(\Hk)\text{-Mod}_{\Sp_G}$. If~$G$ is not an elementary abelian $p$-group, then $\Phi^G(b_G(\Hk))=0$. This follows from the theory of derived defect bases \cite{Mathew16pp,MathewNaumannNoel19}. In other words, the geometric open is empty if $G$ is not an elementary abelian $p$-group.
\end{Exa}

\begin{Exa}
	The geometric open of $\DPerm(G;k)$ is discussed in \cite{BalmerGallauer23pp}. It is a copy of the spectrum of $\Phi^G(\Hkbar)\text{-Mod}_{\Sp}$ which is currently not well-understood. However, see \cite[Example~8.1]{BCHNP1} for some positive results. Note that if $G$ is not a $p$-group then the geometric open is empty, in other words $\Phi^G(\Hkbar)=0$. This follows from the fact that restriction to a $p$-Sylow subgroup is faithful due to the ``cohomological Mackey'' nature of~$\DPerm(G;k)$; see \cref{rem:faithful-restriction}.
\end{Exa}

\begin{Exa}\label{exa:DPerm-locally-unigenic}
	The derived category of permutation modules $\cT=\DPerm(G;k)$ is locally unigenic for any elementary abelian $p$-group. For a cyclic group of order $p$, the cohomological open and the geometric open cover the spectrum, so in this case the claim follows from \cref{exa:coho-open} and \cref{rem:geom-open}. More generally, Balmer and Gallauer \cite[Section~13]{BalmerGallauer23pp} construct an open cover of the spectrum which contains the cohomological and geometric opens and has the property that~$\cT|_U$ is unigenic for each open $U$ in the cover; see the proof of \cite[Theorem~15.3]{BalmerGallauer23pp}.
\end{Exa}

\section{Examples in algebraic geometry}\label{sec:alg-geom}

We now discuss some examples of (fully) faithful functors in algebraic geometry.

\begin{Rem}\label{rem:Spc-of-f}
	Any morphism $f:X\to Y$ of quasi-compact and quasi-separated schemes induces a geometric functor $f^*:\Der(Y)\to\Der(X)$ which in turn induces a morphism $\Spec(\Der(X)^c)\to\Spec(\Der(Y)^c)$ of locally ringed spaces. Moreover, the diagram
	\[\begin{tikzcd}[column sep=large]
		X \ar[r,"f"] \ar[d,"\cong"] & Y \ar[d,"\cong"]\\
		\Spec(\Der(X)^c) \ar[r,"\Spec(f^*)"] & \Spec(\Der(Y)^c)
	\end{tikzcd}\]
	commutes. This can be checked from the definition of the classifying map \cite[Theorem~3.2]{Balmer05a} together with \cite[Lemma~3.3(b)]{Thomason97}.
\end{Rem}

\begin{Exa}\label{exa:affine-ff}
	Let $A \to B$ be a homomorphism of commutative rings. It follows from \cref{prop:fully-faithful-chars} that the base-change functor $f^*:\Der(A)\to\Der(B)$ is fully faithful if and only if $A \to B$ is an isomorphism. We thus need to go beyond affine schemes to find non-trivial examples.
\end{Exa}

\begin{Exa}
	Let $\Pone \to \Spec(k)$ be the projective line over a field $k$. The structure morphism induces a fully faithful functor $\Der(k)\to\Der(\Pone)$ which by \cref{exa:unigenic-fully-faithful} identifies with the unitation: $\Der(k) \cong \Der(\Pone)_{\langle \unit \rangle}$. Thus, the map on spectra induced by the unitation of $\smash{\Der(\Pone)}$ is just $\smash{\Pone}\to\Spec(k)$.
\end{Exa}

\begin{Exa}[Projective bundles]
	More generally, let $\cat E$ be a vector bundle on a smooth projective variety $S$ and let $p:\PS(\cat E)\to S$ be its projectivization. Then the pullback $p^*:\Der(S)\to\Der(\PS(\cat E))$ is fully faithful. See \cite[Lemma~2.1]{Orlov92}, for example, bearing in mind \cref{prop:ff-on-compact}. Note that in this example, the fibers are copies of projective space and hence are, in particular, connected as \cref{thm:main-thmb} asserts must be the case.
\end{Exa}

\begin{Exa}[Birational morphisms]
	If $f:X \to Y$ is a proper birational morphism of smooth varieties over a perfect field then $f^*:\Der(Y)\to\Der(X)$ is fully faithful; see \cite{ChatzistamatiouRuelling11,ChatzistamatiouRuelling15}. For historical context, see \mbox{\cite[Problem~B]{GrothendieckICM}} and \mbox{\cite[p.~55]{GrothendieckSerreCorrespondence}.}
\end{Exa}

\begin{Exa}[Singularities]
	Let $X$ be a variety over a field of characteristic zero. If $X$ has rational singularities then any resolution of singularities $f:\smash{\widetilde{X}}\to X$ induces a fully faithful functor $f^*:\Der(X) \to \Der(\smash{\widetilde{X}})$. This is essentially the definition of rational singularities; see \cite[Section~6.2]{Ishii14} or \cite[Chapter~5]{KollarMori98}. Since rational singularities are known to be a weak type of singularity, many other types of singularities (such as log-canonical or stronger \cite[Theorem~6.2.12]{Ishii14}) give examples of fully faithful functors. An explicit class of examples are quotient singularities (e.g.~Du Val singularities).
\end{Exa}

\begin{Rem}
	Although it does not play a role in our work, we would be remiss not to mention the celebrated result of Orlov \cite{Orlov97,Olander24} asserting that if $X$ and~$Y$ are smooth proper varieties over a field $k$ then any fully faithful $k$-linear exact functor $\Der(Y)^c \to \Der(X)^c$ is a Fourier--Mukai transform.
\end{Rem}

\begin{Rem}
	Although nontrivial morphisms of affine schemes do not provide examples of fully faithful functors (\cref{exa:affine-ff}), there are interesting examples of \emph{faithful} functors:
\end{Rem}

\begin{Exa}\label{exa:affine-faithful}
	Let $A \to B$ be a homomorphism of commutative rings. For this example only we will use $f^*\coloneqq B\otimes_A -:\Mod(A) \to \Mod(B)$ to denote the \mbox{underived} base-change functor. According to \cref{prop:faithful-chars}, the derived functor ${\mathbb{L}f^*:\Der(A) \to \Der(B)}$ is faithful if and only if $A \to B$ is a split monomorphism in~$\Der(A)$ which is the case if and only if $A \to B$ is a split monomorphism in the category $\Mod(A)$. This certainly implies that the underived functor $f^*:{\Mod(A)\to\Mod(B)}$ is faithful. The converse is not true. As explained by Mesablishvili~\cite{Mesablishvili00}, $f^*:{\Mod(A)\to\Mod(B)}$ is faithful if and only if $A\to B$ is a so-called pure monomorphism. This is a ring homomorphism which is ``universally injective'' in the sense that for every $A$-module~$M$, the map $M \to B \otimes_A M$ is injective. Faithfully flat maps are precisely the flat pure monomorphisms;  see \cite[Ch.~I, \S~3.5, Prop.~9]{Bourbaki_CommAlg}. In summary:
	\begin{itemize}
		\item $\mathbb{L}f^*:\Der(A)\to\Der(B)$ is faithful if and only if $\begin{tikzcd}[cramped,sep=small,column sep=tiny] A \ar[r] & B\end{tikzcd}$ is a split monomorphism\linebreak of $A$-modules.
		\item $f^*:\Mod(A) \to \Mod(B)$ is faithful if and only if~$\begin{tikzcd}[cramped,sep=small,column sep=tiny] A \ar[r] & B\end{tikzcd}$~is a pure \mbox{monomorphism}\linebreak of $A$-modules.
	\end{itemize}
\end{Exa}

\begin{Rem}
	An example of a faithfully flat map $A\to B$ which is not a split monomorphism is given in \cite[Example~5.10]{ChakrabortyGurjarMiyanishi16}. Examples of non-flat morphisms which are pure but not split can be found in the study of non-regular $F$-pure rings; see \cite{MaPolstra25bpp}, for example.
\end{Rem}

\begin{Exa}
Let $G$ be a finite group acting on a commutative ring~$R$ such that the order of $G$ is invertible in $R$. The inclusion $R^G \hookrightarrow R$ of the ring of invariants is a split monomorphism (using the Reynolds operator) hence $\Der(R^G)\to\Der(R)$ is faithful. Thus $\Spec(R) \to \Spec(R^G)$ is a spectral quotient map; cf.~\cite[{\S}\,I.2]{MumfordFogartyKirwan94}.
\end{Exa}

\begin{Rem}
	If the order of $G$ is not invertible in $R$ then $R^G \hookrightarrow R$ need not be a split monomorphism. For example, \cite[Theorem~5.5]{Singh98} establishes that if the alternating group~$A_n$ acts on the polynomial ring $R=k[X_1,\ldots,X_n]$ by permuting the coordinates and the field $k$ has characteristic $p>2$, then $R^{A_n} \hookrightarrow R$ is a split monomorphism if and only if the order $|A_n|=n!/2$ is relatively prime to~$p$.
\end{Rem}

\begin{Exa}\label{exa:brenner}
	The homomorphism of noetherian commutative rings $A \to B$ described in \cite[Example~7]{Brenner03} is a split monomorphism in $\Mod(A)$ and has the property that $\Spec(B)\to\Spec(A)$ has neither the going-up nor the going-down property. Recall from \cite[Section~5.3]{DickmannSchwartzTressl19} that for a surjective spectral map the going-up property is equivalent to being a closed quotient map while the going-down property is equivalent to being a closed quotient map for the Hochster dual topologies. Moreover, any open map satisfies the going-down property. Thus this gives an example of a faithful functor $f^*:\Der(A)\to\Der(B)$ for which the associated topological quotient is neither a closed quotient map nor an open quotient map, which completes \cref{rem:faithful-not-closed}.
\end{Exa}

\begin{Rem}\label{rem:faithful-not-all-quotients}
	The notion of a faithful geometric functor does not encompass all topological quotients of schemes. This is related to the potential non-faithfulness of Zariski localization. Indeed, suppose $X$ and $Y$ are quasi-compact and quasi-separated schemes. Any surjective flat morphism $f:X \to Y$ is a topological quotient \cite[Corollaire~2.3.13]{EGA4}. For example, if $Y=\bigcup_{i=1}^n U_i$ is a Zariski cover by affine opens then the disjoint union $\sqcup_{i=1}^n U_i$ admits a surjective flat map $X\coloneqq \sqcup_{i=1}^n U_i \to Y$. However, the induced functor $\Der(Y)\to \Der(X)$ need not be faithful. Indeed, any short exact sequence of vector bundles splits over an affine scheme. Thus, if $Y$ admits a non-split short exact sequence of vector bundles $0\to a\to b\to c\to 0$, then the morphism $c \to a[1]$ is a nonzero morphism in $\Der(Y)$ which vanishes in $\Der(X)$. This occurs even with $Y=\Pone$; see \cite[Example~2.1]{Balmer12}. On the other hand, some surjective flat morphisms do provide faithful functors, such as the next example:
\end{Rem}

\begin{Exa}\label{exa:frobenius}
	Let $k$ be an algebraically closed field of characteristic $p>0$. The absolute Frobenius morphism $f:\Pone \to \Pone$ on the projective line induces a faithful functor $f^*:\Der(\Pone) \to \Der(\Pone)$ for which, topologically, $\Spc(f^*)=f$ is the identity. It is not a full functor; in fact, it is an affine morphism with $f_*(\cat O_{\Pone}) \simeq \cat O_{\Pone} \oplus \cat O_{\Pone}(-1)^{\oplus (p-1)}$; see \cite{Thomsen00} for example.
\end{Exa}

\begin{Exa}
	More generally, let $X$ be a variety over an algebraically closed field of characteristic $p>0$. The absolute Frobenius morphism $f:X\to X$ is a finite morphism, hence the higher cohomology of the direct image $f_*(\cat O_X)$ vanishes. Thus, $f^*:\Der(X)\to \Der(X)$ is faithful if and only if the variety $X$ is Frobenius split. For further details and numerous examples of Frobenius split varieties, see \cite[Section~1.1]{BrionKumar05}.
\end{Exa}

\begin{Rem}
	Finally, according to \cref{prop:faithful-chars}, the faithfulness of derived pullback functors in algebraic geometry is related to the notion of a derived splinter; see~\cite{Bhatt12}.
\end{Rem}

\addtocontents{toc}{\protect\enlargethispage{\baselineskip}}
\section{Affinization}\label{sec:affinization}

For derived categories of schemes, the property of being unigenic is related to the property of being affine. We would like to study this relationship more closely.

\begin{Rem}
	The inclusion of the category of affine schemes into the category of all locally ringed spaces has a left adjoint, which sends a locally ringed space $(X,\cat O_X)$ to its \emph{affinization} $\Aff(X,\OX)\coloneqq \Spec(\OXX)$. The unit of the adjunction provides a canonical morphism of locally ringed spaces $(X,\OX) \to \Aff(X,\OX)$. Explicitly, this map sends a point $x \in X$ to the image under $\Spec(\OXx)\to\Spec(\OXX)$ of the unique closed point of $\Spec(\OXx)$.
\end{Rem}

\begin{Prop}\label{prop:affinization}
	Let $X$ be a quasi-compact and quasi-separated scheme. Under the identification $\Spec(\Der(X)^c)\cong X$ of locally ringed spaces, the affinization $X \to \Aff(X)$ coincides with the comparison map $\Spec(\Der(X)^c)\to\Spec(\OXX)$.	
\end{Prop}

\begin{proof}
	This is straightforward using the naturality of the comparison map with respect to the geometric functor $f^*:\Der(\OXX) \to \Der(X)$ induced by affinization $X \to \Spec(\OXX)$ bearing in mind~\cref{rem:Spc-of-f}.
\end{proof}

\begin{Exa}
	The affinization of the projective line $\Pone$ is the structure morphism $\Pone \to \Spec(k)$.
\end{Exa}

\begin{Exa}
	Let $X=\mathbb{A}^2_k\setminus \{0\}$ be the affine plane with the origin removed. The affinization is the embedding $\mathbb{A}_k^2 \setminus\{0\} \hookrightarrow \mathbb{A}^2_k$. Thus, $\Der(\mathbb{A}^2\setminus\{0\})$ is an example of a tt-category whose comparison map is not surjective; cf.~\cite[Example~8.3]{Balmer10b}. 
\end{Exa}

\begin{Prop}[EGAII]\label{prop:quasi-affine}
	Let $X$ be a quasi-compact and quasi-separated scheme. The following conditions are equivalent:
	\begin{enumerate}
		\item $X$ is quasi-affine; that is, $X$ is isomorphic to a (quasi-compact) open subscheme of an affine scheme.
		\item The affinization map $X \to \Aff(X)$ is an open immersion.
		\item The affinization map $X \to \Aff(X)$ is a topological embedding.
		\item The structure sheaf $\OX$ is ample.
	\end{enumerate}
\end{Prop}

\begin{proof}
	This is (part of) \cite[Proposition~5.1.2, p.~94]{EGAII} which is stated for a scheme that is either topologically noetherian or quasi-compact and separated, but it holds more generally for a quasi-compact and quasi-separated scheme; cf.~\cite[Tag~01P5 and 01QD]{stacks-project}.
\end{proof}

\begin{Rem}
	If $\cat L$ is an ample line bundle on $X$ then $X$ is separated and $\Der(X)$ is compactly generated by 	$\SET{\cat L^{\otimes n}}{n\in \mathbb{Z}}$; see \cite[Example~1.10]{Neeman96}. Hence, if $X$ is quasi-affine then \cref{prop:quasi-affine} implies that $\Der(X)$ is unigenic. The following example shows that the converse is false:
\end{Rem}

\begin{Exa}\label{exa:3-point-scheme}
	Let $X$ be the ``smallest nonaffine scheme'' considered in \cite[Exercise~I-25]{EisenbudHarris00}. It has three points $X=\{p,q_1,q_2\}$ where $q_1$ and $q_2$ are closed points and $p$ is a generic point. The sheaf of rings is given by $\OX(\{p\}) = k(t)$ and $\OX(\{p,q_1\})=\OX(\{p,q_2\})=\OXX=k[t]_{(t)}$ for some field $k$. It is an example of a noetherian scheme which is not separated. In particular, it is not quasi-affine.
\end{Exa}

\begin{Prop}
	Let $X$ be the 3-point scheme of \cref{exa:3-point-scheme}. The derived category $\Der(X)$ is unigenic.
\end{Prop}

\begin{proof}
	The affinization $X \to \Spec(k[t]_{(t)})$ is (topologically) the map
	\begin{equation}\label{eq:3-point-affinization}\begin{tikzcd}[column sep=tiny,row sep=small]
		\bullet &								  &\bullet  \\
		&	\bullet  \ar[ur,dash,thick]\ar[ul,dash,thick] & 
	\end{tikzcd}
	\to
	\begin{tikzcd}[column sep=tiny,row sep=small]
		&	\bullet  \ar[d,dash,thick] \\
		&		\bullet  
	\end{tikzcd}\end{equation}
	Now consider the unitation $\unitation{\Der(X)} \hookrightarrow \Der(X)$. Its spectrum fits between these two spaces by \cref{prop:affinization} and \eqref{eq:factorization}:
	\[\begin{tikzcd}[column sep=tiny,row sep=small]
			\bullet &								  &\bullet  \\
		&	\bullet  \ar[ur,dash,thick]\ar[ul,dash,thick] & 
	\end{tikzcd}
		\xrightarrow{\varphi}\;\;
		\Spc(\Der(X)_{\langle \unit \rangle}^c)
		\xrightarrow{\rho}\;\;
	\begin{tikzcd}[column sep=tiny,row sep=small]
		&	\bullet  \ar[d,dash,thick] \\
		&		\bullet  
	\end{tikzcd}\]
	Either $\varphi$ or $\rho$ must be a homeomorphism, since $\varphi$ is surjective. But \cref{thm:main-thmb} asserts that the fibers of $\varphi$ are connected. Hence we must have that $\varphi$ is a homeomorphism. Thus, the unitation induces a homeomorphism on spectra. Since the category $\Der(X)$ is locally unigenic, this implies that the functor is an equivalence by \cref{cor:homeo-unit}. We conclude that $\Der(X)$ is unigenic.
\end{proof}

\begin{Exa}\label{exa:graded-noetherian-not-connected}
	In the above example, observe that one of the fibers of the comparison map $\rho:X \to \Spec(k[t]_{(t)})$ is not connected. Note moreover that in this example the graded comparison map coincides with the ungraded comparison map since for dimension reasons all elements in $H^1(X;\OX)$ are nilpotent. Moreover, by computing the cohomology one readily checks that $H^\bullet(X;\OX)$ is a graded-noetherian ring. This thus provides an example showing that the fibers of the map in \cref{thm:coherent-spectral} need not be connected. It also shows that \cref{cor:ungraded-comparison} fails in general if the category is not connective.
\end{Exa}

\begin{Que}
	For which quasi-compact and quasi-separated schemes $X$ is $\Der(X)$ unigenic? The above discussion shows that this class of schemes properly contains the quasi-affine schemes.
\end{Que}

\begin{Rem}
	From the perspective of \cref{rem:spectra}, $\Der(X)$ is unigenic if and only if it is equivalent $\Derdg(\eend(\OX)) \smash{\xrightarrow{\simeq}} \Der(X)$ to the derived category of the $\mathbb{E}_\infty$-dg-algebra $\eend(\OX) \in \CAlg(\HZ\text{-Mod})$; see \cite{Shipley07,RichterShipley17}.
\end{Rem}

\begin{Exa}\label{exa:aff-of-spec}
	The Balmer spectrum of an essentially small tt-category $\cat K$ is a locally ringed space $\Spec(\cat K)=(\Spc(\cat K),\OK)$ and we may also consider \emph{its} affinization $\Aff(\Spec(\cat K)) = \Spec(\OK(\Spec(\cat K)))$. Let us avoid such heavy-handed notation by defining $A_{\cat K} \coloneqq \OK(\Spec(\cat K))$ to be the global sections of the structure sheaf so that $\Aff(\Spec(\cat K))=\Spec(A_{\cat K})$. By the universal property, the ungraded comparison map $\rho_{\cat K}$ factors uniquely through the affinization:
	\begin{equation}\label{eq:beta}\begin{tikzcd}
		\Spec(\cat K) \ar[r,"\alpha"] \ar[rr,bend left=20,"\rho"] &\Aff(\Spec(\cat K)) \ar[r,"\exists! \beta"] & \Spec(R_{\cat K}).
	\end{tikzcd}\end{equation}
	Recall from \cite[Section~6]{Balmer10b} that the structure sheaf $\OK$ of $\Spec(\cat K)$ is the sheafification of the presheaf $p\OK$ defined on a quasi-compact opens by $p\OK(U) \coloneqq R_{\cat K|_U}$ with stalks $\OKP \cong R_{\cat K_{\cat P}}$. Explicitly, the global sections (i.e.~the elements of $A_{\cat K}$) are locally compatible families of germs $(s_{\cat P} \in R_{\cat K_{\cat P}})_{\cat P \in \Spc(\cat K)}$. The morphism $\beta$ in \eqref{eq:beta} is given by the ring homomorphism $R_{\cat K} \to A_{\cat K}$ which sends $f \in R_{\cat K}$ to $(f_{\cat P})_{\cat P}$. Its kernel consists entirely of nilpotent elements by \cite[Proposition~2.21]{Balmer05a}. However, there is no \emph{a priori} reason for $\beta$ to be an isomorphism since morally there is no \emph{a priori} reason why the ring of endomorphisms of $\unit$ in $\cat K$ need behave like a sheaf with respect to the Balmer topology. Thus, our next aim is to provide some sufficient criteria for the affinization and the comparison map to coincide. We start with the following cheap observation:
\end{Exa}

\begin{Lem}\label{lem:AK-local}
	If $\cat K$ is local then the canonical map $R_{\cat K} \to A_{\cat K}$ is an isomorphism.
\end{Lem}

\begin{proof}
	Recall that the map of commutative rings $R_{\cat K} \to A_{\cat K}\coloneqq\OK(\Spec(\cat K))$ sends an element $a \in R_{\cat K}$ to the collection of germs $(a_{\cat P} \in R_{\cat K_{\cat P}})_{\cat P \in \Spec(\cat K)}$. Since $\cat K$ is local it has a unique closed point $\cat M$. Hence $R_{\cat K} \to A_{\cat K}$ is injective since $a = a_{\cat M}$. On the other hand, the only open neighbourhood of $\cat M$ is the whole space $\Spec(\cat K)$. Thus the local compatibility of an element $(s_{\cat P})_{\cat P}$ at $\cat M$ implies that there exists $t \in R_{\cat K}$ such that $t_{\cat P} = s_{\cat P}$ for each $\cat P$. In other words, the ring homomorphism $R_{\cat K} \to A_{\cat K}$ is surjective.
\end{proof}

\begin{Prop}\label{prop:when-scheme}
	If $\cat K$ is local then the following are equivalent:
	\begin{enumerate}
		\item $\Spec(\cat K)$ is a scheme;
		\item $\Spec(\cat K)$ is an affine scheme;
		\item The comparison map $\rho_{\cat K} :\Spc(\cat K)\to\Spec(\RK)$ is a homeomorphism.
	\end{enumerate}
\end{Prop}

\begin{proof}
	The equivalence of $(1)$ and $(2)$ follows from the fact that any open cover of $\Spc(\cat K)$ includes $U=\Spc(\cat K)$ itself since $\cat K$ is local. On other hand, the affinization coincides with the comparison map by \cref{lem:AK-local}. This provides $(2)\Rightarrow (3)$. The converse $(3) \Rightarrow (2)$ then follows from the fact that $\rho_{\cat K}$ is an isomorphism of locally ringed spaces if it is a homeomorphism; see \cite[Proposition~6.11(b)]{Balmer10b}.
\end{proof}

\begin{Exa}
	The spectrum of $\SH_{(p)}$ is not a scheme.
\end{Exa}

\begin{Exa}
	It follows from \cref{prop:when-scheme}, \cref{thm:DHZG-comp} and \cite[Proposition~6.11]{Balmer10b} that $\Spec(\unitation{\DHZG})\cong \Spec(A(G))$ is an affine scheme even though the tt-categories $\unitation{\DHZG}$ and $\Der(A(G))$ are not equivalent for any nontrivial $G$; see~\cref{rem:DHZG-not-burnside}.
\end{Exa}

\begin{Prop}\label{prop:aff-comp-domain}
	Suppose that the algebraic localization $\cat K_{\mathfrak m}$ is local for each closed point $\mathfrak m \in \Spec(\RK)$. If $\RK$ is a domain then the affinization of $\Spec(\cat K)$ coincides with the comparison map.
\end{Prop}

\begin{proof}
	In this proof we will write $\cat K/\cat P$ for the local category $\cat K_{\cat P}$ in order to distinguish it more clearly from the algebraic localization $\cat K_{\mathfrak p}$. The map $\beta:R_{\cat K} \to A_{\cat K}$ is injective since $R_{\cat K}$ has no nilpotents (\cref{exa:aff-of-spec}). Consider a family of compatible germs $(a_{\cat P} \in R_{\cat K/\cat P})_{\cat P \in \Spc(\cat K)} \in \AK$. Note that if $\cat Q \specializesto \cat P$ then the germ $a_{\cat P}$ determines the germ $a_{\cat Q}$ via the localization $R_{\cat K/\cat P} \to R_{\cat K/\cat Q}$. Thus, to show that $\RK \to \AK$ is surjective, it is enough to construct $r \in R_{\cat K}$ such that $r_{\cat M} = a_{\cat M}$ for every closed point $\cat M \in \Spc(\cat K)$. By our hypothesis, the closed points of $\Spc(\cat K)$ are all of the form $f(\mathfrak m)$ for closed points $\mathfrak m$ of $\Spec(R_{\cat K})$ using the notation of \cref{rem:alt-local-hyp}. Recall also that $\cat K/f(\mathfrak m) = \cat K_{\mathfrak m}$. Thus, we need to construct an $r \in R$ such that $r_{f(\mathfrak m)} = a_{f(\mathfrak m)} \in R_{\cat K/f(\mathfrak m)} = R_{\mathfrak m}$ for each closed point $\mathfrak m \in \Spec(R_{\mathfrak m})$. Let $\eta = (0) \in \Spec(R_{\cat K})$ be the generic point so that $R_{\eta} = K$ is the field of fractions. It follows from the commutative diagram
	\[\begin{tikzcd}
		R \ar[r] \ar[dr,hook,bend right] & A \ar[dr,bend left,end anchor={[shift={(5pt,0)}]north west}] \ar[d] & \\
									& R_{\mathfrak m} \ar[r,hook] & R_\eta = K
	\end{tikzcd}\]
	that $a_{\eta}$ belongs to $\bigcap_{\mathfrak m} R_{\mathfrak m} = R \subset K$ (see \cite[Theorem~4.7]{Matsumura89}) and thus that there exists $r \in R$ such that $r_{\mathfrak m} = a_{\mathfrak m}$ for each $\mathfrak m$.
\end{proof}

\begin{Exa}\label{exa:SH-affinization}
	The affinization of $\Spec(\SH^c)$ coincides with the comparison map ${\Spec(\SH^c) \to \Spec(\bbZ)}$. In particular, $\Spec(\SH^c)$ is not an affine scheme.
\end{Exa}

\begin{Prop}\label{prop:mayer-vietoris}
	Let $\cat K$ be an idempotent-complete rigid tt-category. Suppose that every open cover of $\Spc(\cat K)$ by quasi-compact opens can be refined into an open cover $\Spc(\cat K)=\bigcup_{i \in I} U_i$ with the property that $\Hom_{\cat K|_{U_i \cap U_j}}(\Sigma \unit,\unit)=0$ for each~$i,j \in I$. Then the presheaf $p\OK$ is separated and the canonical map $R_{\cat K} \to A_{\cat K}$ is an isomorphism. Thus the affinization of $\cat K$ coincides with its comparison map.
\end{Prop}

\begin{proof}
	The key to this result is the Mayer--Vietoris gluing of \cite{BalmerFavi07}. Let $U$ be a quasi-compact open set and let $U=\bigcup_{i \in I} U_i$ be a cover by quasi-compact opens. Let $f \in p\OK(U)=R(\cat K|_U)$ be such that $f|_{U_i} = 0$ in $R(\cat K|_{U_i})$ for each $i\in I$. We may assume without loss of generality that this cover of $U$ has the property in the hypothesis, namely that $\Hom(\Sigma \unit,\unit)=0$ in $\cat K|_{U_i \cap U_j}$ for each $i,j\in I$. We can also assume that it is a finite cover by the quasi-compactness of $U$. An inductive application of the Mayer--Vietoris exact sequence of~\cite[Corollary~5.8]{BalmerFavi07} then implies that $f=0$. This establishes that the partially-defined presheaf $p\OK$ is separated. Moreover, recall that $R(\cat K_{\cat P})$ is the stalk at $\cat P$ by \cite[Lemma~6.3]{Balmer10b}. Thus, if $f \in R(\cat K|_U)$ is such that $f_{\cat P}=0$ in $R(\cat K_{\cat P})$ for all $\cat P \in U$ then $f=0$. In particular, the map $\RK\to\AK$ is injective.

	Finally, consider a family of compatible germs $(a_{\cat P})_{\cat P} \in A_{\cat K}$. The compatibility of these germs implies that there is an open cover $\Spc(\cat K)=\bigcup_{i \in I} U_i$ by quasi-compact opens and elements $s_i \in R(\cat K|_{U_i})$ such that $(s_i)_{\cat P} = a_{\cat P}$ for each $\cat P \in U_i$. Moreover, since $p\OK$ is separated, we have that $(s_i)|_{U_i \cap U_j} = (s_j)|_{U_i \cap U_j}$ for each $i,j\in I$. Since we may assume the cover is finite, an inductive application of the Mayer--Vietoris exact sequence of \cite[Corollary~5.8]{BalmerFavi07} implies that there exists a global section~$s \in R_{\cat K}$ such that $s|_{U_i} = s_i$ for each $i\in I$. It maps to our germ $(a_{\cat P})_{\cat P}$ under the canonical map $R_{\cat K} \to A_{\cat K}$. Thus the canonical map is surjective.
\end{proof}

\begin{Exa}
	If $X$ is a quasi-compact and quasi-separated scheme then $\Der(X)^c$ satisfies the hypotheses of \cref{prop:mayer-vietoris}. Thus \cref{prop:affinization} follows from the more general \cref{prop:mayer-vietoris}.
\end{Exa}

\begin{Rem}
	Since the comparison map $\rho:\Spc(\cat K)\to \Spec(\RK)$ factors through both the affinization $\Spc(\cat K)\to \Spec(\AK)$ and the unitation $\Spc(\cat K)\to\Spc(\unitation{\cat K})$, one may wonder what relation these two constructions have with each other. Note that they need not coincide, since the affinization is not always surjective whereas the unitation always is. In general, there is no \emph{a priori} reason why either map need factor through the other. For example, if $X$ is the 3-point scheme of \cref{exa:3-point-scheme} then the unitation $X\xrightarrow{\simeq} \Spc(\unitation{\Der(X)}^c)$ does not factor through the affinization depicted in~\eqref{eq:3-point-affinization}.
\end{Rem}

\begin{Exa}\label{rem:exceptional}
	For a quasi-compact and quasi-separated scheme $X$, the affinization $X \to \Aff(X)$ induces a geometric functor $\Der(\Aff(X)) \to \unitation{\Der(X)}$ which is an equivalence if and only if $H^i(X,\OX)=0$ for $i>0$. (Recall \cref{rem:geom-between-unigenic}.) There are plenty of examples of smooth projective varieties whose structure sheaf is exceptional. For example, smooth Fano varieties over a field of characteristic zero have this property by the Kodaira vanishing theorem. For such varieties, the unitation and the derived category of the affinization coincide --- they are just the derived category of the base field. Although this may seem trivial, these are examples where the unitation, the affinization and both the graded and ungraded comparison maps all coincide and the fibers of this map need not be local. They demonstrate that the unigenic hypothesis in \cref{thm:connective-unigenic} cannot be dropped.
\end{Exa}

\addtocontents{toc}{\protect\enlargethispage{2\baselineskip}}

\bibliographystyle{alpha}\bibliography{bibliography}

\end{document}